\documentclass[11pt]{article}
\usepackage{fancyhdr}
\usepackage{bigints}
\usepackage{mathrsfs} 
\DeclareMathAlphabet{\pazocal}{OMS}{zplm}{m}{n}
\usepackage{isomath}
\usepackage{mathtools} 
\usepackage{amsbsy}
\usepackage{multirow}
\usepackage{amssymb}
\usepackage{amsthm}
\usepackage{amscd,amsfonts}
\usepackage{bigints}
\usepackage{graphicx}
\usepackage{verbatim}
\usepackage{euscript}
\usepackage{alltt}
\usepackage{stmaryrd}
\usepackage{relsize}
\usepackage{enumerate}
\usepackage{url}
\usepackage[stable]{footmisc}
\usepackage{hyperref}
\usepackage{breakurl}
\usepackage{comment}
\usepackage[maxbibnames=9,sorting=none,style=numeric]{biblatex}
\usepackage[font=small,labelfont=bf]{caption}
\usepackage[font=small,labelfont=bf]{subcaption}
\usepackage{float}
\usepackage{mwe}
\usepackage{algorithm}
\usepackage{algpseudocode}
\usepackage{xcolor}
\usepackage[super]{nth}
\usepackage{soul}
\usepackage{centernot}
\DeclareGraphicsExtensions{.pdf}
\usepackage{amsmath}
\usepackage{amsbsy}
\usepackage{amssymb}
\usepackage{amscd}
\usepackage{amsfonts}

\newcommand{\R}{\mathbb R}
\newcommand{\N}{\mathbb N}
\newcommand{\Z}{\mathbb{Z}}

\newcommand{\eps}{{\varepsilon}}

\newcommand{\beq}{\begin{equation}}
\newcommand{\eeq}{\end{equation}}
\newcommand{\beqs}{\begin{eqnarray}}
\newcommand{\eeqs}{\end{eqnarray}}
\newcommand{\beql}{\begin{equation} \label}
\newcommand{\half}{\frac{1}{2}}
\newcommand{\del}{\delta}

\newcommand{\calF}{{\cal F}}

\newtheorem{proposition}{Proposition}[section]

\newcommand{\p}{\partial}

\usepackage[margin=1in]{geometry}

\numberwithin{equation}{section}
\newcommand{\dee}{\mathcal{D}}

\newcommand{\rmd}{{\rm d}}
\newcommand{\ddt}{\frac{\rmd}{\rmd t}}
\newcommand{\D}{\partial}
\newcommand{\etal}{{\it et al.}}
\DeclareMathOperator{\dom}{\rm dom}
\DeclareMathOperator{\interior}{\rm int}
\DeclareMathOperator{\diag}{\rm diag}
\DeclareMathOperator{\sech}{\rm sech}
\DeclareMathOperator{\sinc}{\rm sinc}
\newcommand{\Snu}{{\mathfrak S}}  
\newcommand{\ma}{\mathsf{a}}
\newcommand{\mK}{\mathsf{K}}
\newcommand{\uinf}{u_\infty}

\newcommand{\inv}{^{-1}}

\date{}
\begin{document}

\title{Traveling wave profiles for a semi-discrete Burgers equation}

\author{Uditnarayan Kouskiya\thanks{Department of Civil \& Environmental Engineering, Carnegie Mellon University, Pittsburgh, PA 15213, email: udk@andrew.cmu.edu.} $\qquad$ Robert L. Pego\thanks{Department of Mathematical Sciences, and Center for Nonlinear Analysis, Carnegie Mellon University, Pittsburgh, PA 15213, email: rpego@cmu.edu.} $\qquad$ Amit Acharya\thanks{Department of Civil \& Environmental Engineering, and Center for Nonlinear Analysis, Carnegie Mellon University, Pittsburgh, PA 15213, email: acharyaamit@cmu.edu.}}

\maketitle

\begin{abstract}
\noindent 
We look for traveling waves of the semi-discrete conservation law
$4\dot u_j +u_{j+1}^2-u_{j-1}^2 = 0$, using variational principles
related to concepts of ``hidden convexity'' appearing in 
recent studies of various PDE (partial differential equations).
We analyze and numerically compute with two variational formulations 
related to dual convex optimization problems constrained by either
the differential-difference equation (DDE) or nonlinear integral equation (NIE)
that wave profiles should satisfy.
We prove existence theorems conditional on the existence of extrema
that satisfy a strict convexity criterion, and numerically
exhibit a variety of localized, periodic and non-periodic wave phenomena.
\end{abstract}
\section{Introduction}

A great many types 
of dynamic behavior are known to arise from 
different approximation schemes for solutions to the inviscid Burgers equation 
\begin{equation}
 \D_t u + \D_x \left(\half u^2 \right) = 0, \quad -\infty<x<\infty, \quad t>0.
 \label{eq:iBurgers}
\end{equation}
In particular, conservative and dispersive approximations of this equation often exhibit {\it dispersive shocks}. 
As a dispersive shock develops, its structure is that of a modulated envelope of periodic waves, 
and at its leading edge one sees the emergence of several solitary waves.
While a great deal is known about dispersive shocks for integrable approximations to \eqref{eq:iBurgers}, 
little appears to be understood for non-integrable approximations.

A recent study by Sprenger~{\it et al.}~\cite{sprenger2024hydrodynamics} focuses on a non-integrable approximation
generated by simple centered differences in space: 
With $u_j(t)\approx u(jh,ht)$, one obtains the equations
\begin{equation}
   \ddt u_j + \frac14\left(u_{j+1}^2 - u_{j-1}^2 \right) = 0 \, \qquad j\in \Z.
    \label{eq:dB}
\end{equation}
In addition to dispersive shocks,  Sprenger~\etal\ illustrate  many other different
regimes and types of solutions in numerical simulations of these equations.
They use Whitham modulation theory and weakly nonlinear asymptotics to
explain some of the observed phenomena.  Other interesting behaviors that were observed include
highly nonlinear phenomena such as discontinuous waves connecting periodic solutions 
to constant states, and periodic waves emerging from both sides of a discontinuity.

In this article we focus on studying periodic and solitary traveling waves of 
the semi-discrete Burgers equations~\eqref{eq:dB}.  
Such waves must take the form  
\begin{equation}\label{eq:SWform}
u_j(t) = f(j-ct),
\end{equation}
where the wave profile $f$ satisfies the differential-difference equation (DDE)
\begin{equation}
  -c f'(x) + \frac14 \left(f(x+1)^2-f(x-1)^2\right)  = 0. 
\label{eq:ADDE}
\end{equation}
Upon integration we find this is equivalent to the nonlinear integral equation (NIE)
\begin{equation}
  -c f(x) + \frac14 \int_{x-1}^{x+1} f(z)^2\,dz = C_0 ,
\label{eq:int}
\end{equation}
where $C_0$ is a constant.

For our study we will adapt the dual variational framework that has been developed for a variety of problems in \cite{Acharya:2024:HCC,Acharya:2023:DVP,Acharya:2022:VPN}
and computationally demonstrated in \cite{Kouskiya:2024:HCH,KA2,sga}, 
with rigorous results presented in \cite{sga, ASZ}. 
The approach is closely related to the idea of `hidden convexity' in nonlinear PDE developed by 
Brenier \cite{brenier2018initial,brenier_book}. Brenier's approach has recently been extended and employed by Vorotnikov~\cite{vorotnikov2022partial,vorotnikov2025hidden} to establish a version of 
Dafermos' maximum entropy rate principle for conservation laws, 
and by Mirebeau and Stampfli \cite{MirebeauStampfli}
to establish rates of convergence to smooth solutions for discretization schemes for
multidimensional quadratic porous medium and (in)viscid Burgers equations.

We utilize this approach for both numerical and theoretical reasons.
Regarding the numerical computation of steady wave profiles, a method known as 
{\it Petviashvili iteration} often works rather well in practice~\cite{petviashvili1976equation,PelinovskyStepanyants2004}.
Equation~\eqref{eq:int} may be the simplest for which this method can be applied, in fact. 
The global convergence properties of Petviashvili iteration are not  very well understood, however.  
The work of Pelinovsky and Stepanyants~\cite{PelinovskyStepanyants2004} established criteria for local convergence or divergence,
but requires certain assumptions about the spectrum of the linearization at an exact wave profile,
assumptions which are not known to hold in many settings, including that of \eqref{eq:int}.

Regarding theory, we are not aware of any proof of existence
for any non-trivial traveling waves of the semi-discrete Burgers equations~\eqref{eq:dB},
except that the approach of Herrmann~\cite{Herrmann_oscillatory} may capture periodic solutions that
are strictly of one sign after modifying the flux function $\Phi'(u)=\frac12 u^2$ 
to be strictly monotonic and $C^1$.
The solitary-wave problem appears similar to that for solitary waves in Fermi-Pasta-Ulam type 
particle lattices involving nearest-neighbor forces, which is the subject of 
a recent review by Vainchtein~\cite{Vainchtein2022}.
Variational methods based on concentration compactness or mountain-pass
methods have been used to prove existence theorems for solitary waves in 
Fermi-Pasta-Ulam lattices \cite{Friesecke.Wattis.94,SmetsWillem1997,Herrmann.10} and 
related equations of peridynamics~\cite{PegoVan2019,HerrmannKlein2024}.
But we are not aware of an existing variational formulation for the traveling wave
profile equations \eqref{eq:ADDE} or \eqref{eq:int}.  
The general variational framework of Herrmann and Matthies \cite{HerrmannMatthies2020variational} 
deals with equations that resemble~\eqref{eq:int}, but does not evidently apply
due to the fact that the characteristic function of an interval 
has a sign-changing Fourier transform and thus is not a convolution square.

For waves that are small-amplitude long-wave perturbations of a non-zero constant state, 
a formal Korteweg-de Vries (KdV) approximation can be described, as one may expect,
and we describe this in Appendix~\ref{s.KdV}.  
It is plausible that existence of waves in this regime could be established by adapting 
existing fixed-point and pertubation methods~\cite{Friesecke.Pego.99,Iooss2000,IoossJames2005,James2012,
HerrmannMikikitsLeitner2016,Ingimarson_2024}.
On the other hand, large-amplitude waves, and any perturbations of the zero solution,
are in a completely nonlinear regime inaccessible by such methods. 
Perhaps existence results might be had by methods based on topological degree theory, 
though, like those in \cite{BenjaminBonaBose1990,BonaChen2002}.

The dual variational formulation that we will study here provides a flexible method 
for exploring the solution space numerically.  
Using it we also prove conditional existence theorems for traveling waves. 
We prove the existence of extrema which determine an exact traveling
wave solution, conditional upon a certain domain constraint being strictly satisfied.
In many cases, our numerical computations strongly indicate that the domain constraint indeed strictly holds.
But we do not claim any mathematically rigorous existence proof.

A useful feature of our approach is that it provides a well-set solution strategy without imposing further conditions on the fundamental problem statement for ones not naturally posed with boundary conditions (such as the traveling wave problem). We exploit this feature in our formulation and computations.

Let us briefly summarize some of the key results of our computations. 
We find a considerable variety of periodic and non-periodic wave profiles on long intervals.
Our computations of waves with localized structure suggest
that solitary wave solutions should exist having limit states 
\begin{equation}\label{eq:baru}
\bar u = \lim_{t\to\pm\infty} u_j(t)\,,
\end{equation}
whenever a {\em phase-speed non-matching condition} holds.  
Namely, a solitary wave of the form \eqref{eq:SWform} that satisfies \eqref{eq:baru}
should exist whenever the speed $c$ does not match the phase velocity 
$\omega/\xi$ of any linear harmonic wave  $v_j(t)=e^{i\xi j-i\omega t}$ for the equations 
\eqref{eq:dB} linearized at the constant state $\bar u$. 
As the dispersion relation for this linearized equation is
$\omega = \bar u\sin \xi$, 
this means that the phase-speed non-matching condition reads
\begin{equation}
c \ne \bar u \sinc \xi \qquad\text{for all $\xi\in\R$,}
\end{equation}
where $\sinc\xi = \sin\xi/\xi$ for $\xi\ne0$ (and $=1$ for $\xi=0$).

The wave profiles that we find develop different types of long-wave structure in 
two regimes in which phase-speed non-matching is breaking down. 
In one regime, $c\approx \bar u$. This corresponds to the KdV long-wave regime, and indeed
we see wide, small-amplitude single-hump wave profiles there.
In the other regime, $c\approx \bar u\sinc \xi_*$ 
where $\xi_*\approx 4.4934$ is the value which minimizes $\sinc\xi$.
Here we see wave profiles oscillating with wave number near $\xi_*$ but
modulated with an envelope that slowly decays away.
Note that at the critical wave number $\xi_*$, the phase velocity matches 
the group velocity $d\omega/d\xi=\bar u\cos\xi$ of wave packets. 
A recent work by Kozyreff~\cite{Kozyreff2023} studies wave propagation 
in such a regime. Kozyreff's study involves the asymptotics of exponentially small terms, 
a topic beyond the scope of the present paper.

A further interesting result from our computations is that we find 
a range of cases where our optimization method finds localized waves with oscillatory
decaying tails, but Petviashvili iteration, initiated with the resulting wave shape, goes unstable and fails to converge.  

As our focus in the present paper is to demonstrate the utility of the variational method for 
finding wave solutions, 
we leave systematic investigation and classification of the families of nonlinear wave solutions of 
the semi-discrete Burgers system \eqref{eq:dB} for future research.
Clearly, there remain many issues to be understood more thoroughly and more rigorously.

\section{Variational formulations}
By a simple scaling (replacing $f$ by $-2cf$), equation~\eqref{eq:ADDE} for nontrivial traveling wave profiles
can be reduced to the following DDE that corresponds to setting $c=-\frac12$ in \eqref{eq:ADDE}:
\begin{equation}
     f'(x) + \frac{1}{2}\Big(\,(f(x+1))^2 - (f(x-1))^2\,\Big) = 0,
    \label{eq:primal}
\end{equation}
where $f:\mathbb{R}\rightarrow\mathbb{R}$ and $(\cdot)'\equiv\frac{d(\cdot)}{dx}$. 
Correspondingly, equation \eqref{eq:int} is reduced to the NIE
\begin{equation}\label{eq:int_cm1}
    f(x) + \frac12 \int_{x-1}^{x+1}f(y)^2\,dy = C_1,
\end{equation}
where $C_1=\half C_0/c^2$.
In this section, we will describe two related dual variational formulations for this problem, 
one that starts directly from the DDE \eqref{eq:primal} and 
one that starts from the NIE~\eqref{eq:int_cm1}.
The formulations have slightly different numerical and analytical properties, to be compared in later sections.

Schematically, the formulations we study are dual to primal optimization problems of a classic form. 
Namely, as discussed in \cite{brenier2018initial,Acharya:2024:HCC},
one seeks to  minimize an integral functional $\int H(f,x)\,dx$, constrained by the equation in question.
If $Q[f]=0$ corresponds to the equation in question, one introduces a dual field (Lagrange multiplier) $\lambda(x)$
and  writes the primal problem in the inf-sup formulation
\[
\text{Find}\quad \inf_f \sup_\lambda \int \left(\lambda Q[f] + H(f,x) \right) \,dx \,.
\]
This primal problem is set only as a device, however.  The point is that the dual problem, obtained by interchanging inf and sup,
is one that we can exploit to obtain information about solutions of $Q[f]=0$ both computationally and theoretically.\footnote{The term `duality' is used  in the physics literature when two superficially distinct theoretical structures can be mapped onto each other to facilitate the approximation of difficult nonlinear problems. This served as the initial motivation for the use of the term for the method adopted here, as discussed in \cite[Sec.~1]{acharya2022action}.}

Of particular importance is the realization that $Q[f]=0$ is a component of the first-order optimality conditions of both the $\sup_\lambda \inf_f$ and $\inf_f \sup_\lambda$ problem statements, \textit{regardless of the choice of $H$}. Thus, our solution strategy exploits the use of not just a single $H$ function, but a family of adapted, convex ones, allowing parametrization by base states which are specified functions or trajectories that can encode knowledge about approximate solutions \cite{Acharya:2023:DVP,Acharya:2024:HCC}. Indeed, base states are utilized in a crucial, adaptive way in our algorithm in Sec.~\ref{sec:numerical}.

Below, for any field $(\cdot)$, we will use the following notation:
\begin{equation*}
(\cdot)(x+h) \equiv
    (\cdot)_{x+h} 
\end{equation*}
In cases when $h=0$, we will utilize the notations $(\cdot)(x)\equiv (\cdot)_x$ and $(\cdot)(x)\equiv (\cdot)$ interchangeably.

\subsection{Dual DDE formulation}\label{sec:DDE_formulation}
Considering a dual field $\lambda$ corresponding to \eqref{eq:primal}, define the following  \textit{pre-dual} functional:
\begin{equation}
\label{eq:predual}
    \widehat{S}[f,\lambda] = \int_{-\infty}^{\infty} \,\left( - \lambda'\, f + \frac{\lambda}{2} (\,f_{x+1}^2 - f_{x-1}^2\,)\, + \, H(f,x)\right) dx,
\end{equation}
where $H$ is a free auxiliary function of the arguments shown. Since the limits of integration involved in the last equation are from $-\infty$ to $\infty$, we can write:
\begin{equation}
\label{eq:non_local_shift}
    \int_{-\infty}^{\infty} \lambda \,f_{x+1}^2 \, dx = \int_{-\infty}^{\infty} \lambda_{x - 1} \, f^2 \, dx\quad \mbox{and} \quad 
    \int_{-\infty}^{\infty} \lambda \,f_{x-1}^2 \, dx = \int_{-\infty}^{\infty} \lambda_{x + 1} \, f^2.
\end{equation}
Equation \eqref{eq:predual} can thus be rewritten as:
\begin{equation}\label{eq:hatS2}
\widehat{S}[f,\lambda] = \int_{-\infty}^{\infty} \left(  -\lambda' \, f + \frac{f^2}{2} (\lambda_{x-1} - \lambda_{x+1}) + H(f,x) \right) dx = \int_{-\infty}^{\infty} \mathcal{L}_H(f,\mathcal{D},x) \, dx,
\end{equation}
where $\dee:=\{\lambda_{x-1},\lambda',\lambda_{x+1}\}$ and we refer to the integrand $\mathcal{L}_H$ as the \textit{Lagrangian} for the current problem. The subscript $H$ in $\mathcal{L}_H$ denotes its dependence on the chosen auxiliary function. 
We now impose the following condition 
\begin{equation}
\label{eq:predtp}
    \frac{\p \mathcal{L}_H}{\p f} = 0: \quad  - \lambda' + f (\,\lambda_{x-1} - \lambda_{x+1}\,)\, + \, \frac{\p H}{\p f}=0,
\end{equation}
to solve for $f$ in terms of the dual objects  
\[
\dee := (\lambda', \lambda_{x - 1}, \lambda_{x+1}).
\]
The choice of the function $H$ is made to enable this step, for a substantial class of dual fields. One such choice of $H$ is
\begin{equation}
    H(f,x) =: \hat{H}(f(x),\bar{f}(x)) = \frac{\mathsf{a}}{2}(f - \bar{f})^2,
    \label{eq:auxiliary_func}
\end{equation}
where $\bar{f}:\mathbb{R}\rightarrow\mathbb{R}$ represents a base state and $\mathsf{a}\gg0$.
The choice of $\ma$ is discussed at the end of this subsection and 
in subsection~\ref{sec:dde_numerics}.

For such a choice of $H$, 
equation \eqref{eq:predtp} takes the form
\begin{equation}
         (\mathsf{a} + \lambda_{x-1} - \lambda_{x+1})f  =  \ma \bar f + \lambda' \,,
\end{equation}
and we define the following \textit{Dual-to-Primal} $(DtP)$ mapping to provide a solution for $f$
when one exists:
\begin{equation}
\begin{aligned}
    & f^{(H)}(\dee,\mathsf{a},x)   = \begin{cases}
    \frac{\ma \bar f + \lambda' }{\ma + \lambda_{x-1}-\lambda_{x+1}} 
    = \bar{f} \, + \, \frac{\lambda' - \bar{f}(\lambda_{x-1} - \lambda_{x+1})}{\mathsf{a} + \lambda_{x-1} - \lambda_{x+1}}
    \quad \mbox{ if } \ \mathsf{a} + \lambda_{x-1} - \lambda_{x+1} \neq 0 \\
   & \\
    0 \quad \mbox{ if } \ \ma + \lambda_{x-1}-\lambda_{x+1} = 0 \mbox{ and } \ \ma \bar f + \lambda' = 0 \end{cases}\\
\end{aligned}
\label{eq:dtp}
\end{equation}
We will often write $\hat{f}:= f^{(H)}(\dee,\mathsf{a},x)$. The dual functional can now be defined as:

\begin{equation}
\begin{aligned}
S[\lambda]  := \widehat{S}[\hat{f},\lambda] & = \int_{-\infty}^{\infty} \left( - \lambda' \, \hat{f} + \frac{\hat{f}\,^2}{2} (\lambda_{x-1} - \lambda_{x+1}) +  \frac{\mathsf{a}}{2}(\hat{f} - \bar{f})^2\right) dx
\end{aligned}
\label{eq:functional_1}
\end{equation}
which can be rewritten completely in terms of dual variables as:
\begin{equation}\label{eq:S_rlp}
\begin{aligned}
S[\lambda]  &= \int_{-\infty}^{\infty} \left( - \lambda' \, \hat{f} + \frac{\hat{f}\,^2}{2} (\lambda_{x-1} - \lambda_{x+1}) 
+ \ma \left(\frac{\hat f\,^2}2 - \bar f \hat f  +  \frac{\bar f^2}2\right)\right) dx\\
&=  \int_{-\infty}^\infty 
\left( (\ma+\lambda_{x-1}-\lambda_{x+1}) \frac{\hat f\,^2}2 - (\ma \bar f+\lambda')\hat f + \ma \frac{\bar f^2}2
\right) dx
\\
&=  \int_{-\infty}^\infty \left(
-
(\ma +\lambda_{x-1}-\lambda_{x+1}) 
\frac{\hat f\,^2}2 
+ \frac{\ma \bar f\,^2}2
\right)dx
\\ 
&= \frac12 \int_{-\infty}^\infty \left( -\frac{(\lambda'+\ma \bar f)^2}{ \ma +\lambda_{x-1}-\lambda_{x+1} }
+ {\ma \bar f\,^2}
\right)dx
\end{aligned}
\end{equation}
An alternate, but equivalent, expression arises by first writing $S[\lambda]$ in terms of $(\hat{f} - \bar{f})$:
\begin{multline*}
S[\lambda]  = \int_{-\infty}^{\infty} \left( - \lambda' \, \hat{f} + \frac{\hat{f}\,^2}{2} (\lambda_{x-1} - \lambda_{x+1}) 
+ \frac{\mathsf{a}}{2}(\hat f - \bar f)^2 \right) dx\\
= \int_{-\infty}^{\infty} \left( - \lambda' \, (\hat{f}-\bar{f}) + \frac{(\hat{f}-\bar{f})\,^2}{2} (\lambda_{x-1} - \lambda_{x+1}) + \frac{\mathsf{a}}{2}\,(\hat{f} - \bar{f})^2 \right) dx \\ 
+ \int_{-\infty}^{\infty}\,(\hat{f} - \bar{f})\bar{f}(\lambda_{x-1} - \lambda_{x+1}) + \int_{-\infty}^{\infty} -\lambda'\, \bar{f} + \frac{\bar{f}\,^2}{2}(\lambda_{x-1} - \lambda_{x+1})\,dx
\end{multline*}
which reduces, after collecting and combining linear and quadratic terms in $(\hat{f} - \bar{f})$ using the DtP mapping, to
\begin{equation}
S[\lambda] = \int_{-\infty}^{\infty} -\frac{1}{2\,\Delta_\lambda} \,\Big( \lambda' - \bar{f}(\lambda_{x-1} - \lambda_{x+1})\Big)^2\,dx + \int_{-\infty}^{\infty} -\lambda'\, \bar{f} + \frac{\bar{f}\,^2}{2}(\lambda_{x-1} - \lambda_{x+1})\,dx ,
\label{eq:functional_2}
\end{equation}
where
\begin{equation}
    \Delta_\lambda := \mathsf{a} + \lambda_{x-1} - \lambda_{x+1}.
    \label{eq:denominator}
\end{equation}

The first variation of \eqref{eq:functional_1} in a direction $\del \lambda$ is given by
\begin{subequations}
\begin{align}
\notag \del^{(1)}S[\lambda;\del \lambda] &= {\int_{-\infty}^{\infty}\left(\frac{\p \mathcal{L}_H}{\p f} \frac{\p \hat{f}}{\p \mathcal{D}}\,\del \mathcal{D} +\frac{\p \mathcal{L}_H}{\p \mathcal{D}}\,\del \mathcal{D} \right) \,dx= \int_{-\infty}^{\infty} \left(\frac{\p \mathcal{L}_H}{\p \mathcal{D}}\,\del \mathcal{D} \right)\,dx} \notag\\ 
&= \int_{-\infty}^{\infty} \left( - \del \lambda' \, \hat{f} + \frac{\hat{f}\,^2}{2} (\del\lambda_{x-1} - \del\lambda_{x+1}) \right) dx, \label{eq:second_form_DDE} \\
& =  \int_{-\infty}^{\infty} \del \lambda\left( \hat{f}' + \frac{1}{2} (\hat{f}\,^2_{x+1} - \hat{f}\,^2_{x-1}) \right) dx. \label{eq:first_form_DDE}
\end{align}
\end{subequations}
where we have used
\begin{equation}   \label{eq:inf_bc}
\del \lambda \to 0 \quad\mbox{ as } x\rightarrow\pm \infty.
\end{equation}
 On requiring $\del^{(1)}S[\lambda;\del \lambda] = 0 $ for any $\delta \lambda$ satisfying \eqref{eq:inf_bc}, the Euler-Lagrange equation for $S[\lambda]$ is given by
\begin{equation}
    \hat{f}'(x) + \frac{1}{2}\Big(\,(\hat{f}(x+1))^2 - (\hat{f}(x-1))^2\,\Big) = 0 \quad \forall x\in(-\infty,\infty),
    \label{eq:primal_in_dual}
\end{equation}
which is the same as \eqref{eq:primal} now written only in terms of the dual variables. 

Additionally, one can always choose $\delta \lambda(x) =0 $ for $x\notin(a,b)$ to establish \eqref{eq:primal_in_dual} only for $x\in(a,b)$.

An important consistency check of our scheme is that for each solution, say $f^*$, to the primal problem, there is at least one dual functional whose critical point corresponds to that solution. That functional is constructed simply by the choice of $\bar{f} = f^*$ and the critical point is given by $\lambda = 0$, as can be directly read off from the DtP mapping and the fundamental justification of the scheme that the primal equation forms the Euler-Lagrange equation of the dual functional with DtP mapping substituted.

{\em Concave maximization.}
Having demonstrated the consistency of our scheme with the problem \eqref{eq:primal} as a critical point problem of the dual functional \eqref{eq:functional_2}, for practical purposes we consider a pure maximization problem on a bounded domain.
We require the dual fields $\lambda$ to vanish outside a given finite interval $\Omega=(-L,L)$, and
consider a related functional given by
\begin{equation}\label{eq:sup_inf}
\begin{aligned}
    \tilde{S}[\lambda] := \inf_f \widehat{S}_L[f,\lambda] & = \inf_f 
    \int_{\Omega_1}
    \bigg( - \lambda' f + \frac{f^2}{2} ( \lambda_{x-1} - \lambda_{x+1}) + \frac{\mathsf{a}}{2} (f - \bar{f})\,^2 \bigg) dx\\
                    &  = 
		    \int_{\Omega_1}
		    \inf_f \bigg( \frac{f^2}{2} ( \mathsf{a} + \lambda_{x-1} - \lambda_{x+1}) - f \left(\lambda' + \mathsf{a} \bar{f}\, \right) + \frac{\mathsf{a}}{2} {\bar{f}}\,^2 \bigg) dx,
\end{aligned}
\end{equation}
where $\widehat{S}_L$ is exactly the functional $\widehat{S}$ from \eqref{eq:hatS2}, 
restricting the domain of integration to the minimal interval $\Omega_1:=(-L-1,L+1)$ that accomodates all nonvanishing values of 
$\lambda_{x-1}$ and $\lambda_{x+1}$ to be considered.
Noting that the Lagrangian of $\widehat{S}_L$ is affine in $\dee$, the integrand of $\tilde{S}$ must be concave in $\dee$. Furthermore,
\begin{equation}\label{eq:tilde_S}
\tilde{S}[\lambda] = \begin{cases} S_L[\lambda] \qquad \mbox{for } \lambda \mbox{ s.t. } a.e. \ (\mathsf{a} + \lambda_{x-1} - \lambda_{x+1}) \geq 0 \mbox{ and } \\
                                    \qquad \qquad \mbox{ if } (\mathsf{a} + \lambda_{x-1} - \lambda_{x+1}) = 0 \ \mbox{ then } \lambda'(x) + \mathsf{a} \bar{f}(x) = 0, \\
                             - \infty \qquad \mbox{otherwise},
                             \end{cases}
\end{equation}
where $S_L$ is the functional $S$ from \eqref{eq:functional_1} integrated on the interval $\Omega_1$.
This is so because at points where $\Delta_\lambda > 0$ the integrand of $\tilde{S}$ has a unique minimizer (over $f$) given by the integrand of $S$, and at $x$ where $\Delta_\lambda (x) = 0$, the associated condition $\lambda'(x) + \mathsf{a} \bar{f}(x) = 0$ again ensures that the integrands of $\tilde{S}$ and $S_L$ match, as can be seen from \eqref{eq:functional_1}-\eqref{eq:dtp}. Thus, finding a critical point of $\tilde{S}$ corresponds to a \textit{concave maximization} problem. Moreover, in such a maximization, $\lambda$ fields which do not satisfy the \textit{Convexity Condition}
\begin{equation}
\mathscr{C}: \Delta_\lambda = \quad (\mathsf{a} + \lambda_{x-1} - \lambda_{x+1}) \geq 0 \quad a.e.
\label{eq:convexity_cond}
\end{equation}
cannot be competitors for being a maximizer, and it is best to not be concerned with such fields (we note that $S_L[\lambda]$ need not be concave over the set of all $\lambda$ fields); alternatively, we can simply choose to seek maximizers of $S_L[\lambda]$ in the reduced set of fields which satisfy $\mathscr{C}$. 

We will utilize the above insights in our numerical scheme by \textit{looking for critical points of $S_L[\lambda]$ constrained by $\mathscr{C}$.}
Some rigorous analytical properties of the concave functional $\tilde S$ will be developed in Section~\ref{sec:analysis} below.
We note here that the condition \eqref{eq:convexity_cond} is the analog, in this non-local setting, of the condition that guarantees a local degenerate ellipticity of the dual critical point problem in the PDE case, as defined in \cite[Sec.~3]{Acharya:2024:HCC} and applied to the inviscid Burgers equation in \cite[Sec.~2]{KA2}.

{\em A scaling symmetry.} 
The pre-dual and dual functionals defined above depend in a simple
way upon the parameter $\ma>0$ that scales the amplitude of the function $H(f,x)$
in \eqref{eq:auxiliary_func}. 
If we make this dependence upon $\ma$ explicit, then it is evident 
from \eqref{eq:hatS2} that the pre-dual functional $\hat S$ satisfies
\begin{equation}
\label{eq:DDE_map_invariance}
\hat S(f,\lambda,\ma) = \ma\, \hat S(f,\lambda/\ma,1),
\end{equation}
the DtP map in \eqref{eq:dtp} satisfies
\begin{equation}
f^{(H)}(\mathcal{D},\ma,x) = f^{(H)}(\mathcal{D}/\ma,1,x),
\end{equation}
and the functional $S$ in \eqref{eq:functional_1} satisfies
\begin{equation}
\label{eq:DDE_variation_invariance}
S[\lambda,\ma] = \ma\, S[\lambda/\ma,1].
\end{equation}
Thus, increasing $\ma$ from $1$ simply produces a proportional increase in $S$ at a scaled down
argument $\lambda/\ma$. The first variation remains invariant with scaled arguments, satisfying
\begin{equation}
    \delta^{(1)}S[\lambda,\ma;\delta\lambda] =  \delta^{(1)}S[\lambda/\ma,1;\delta\lambda/\ma] 
\end{equation}
and correspondingly the second variation is inversely proportional to $\ma$.

What this means is that, for our present choice of the function $H$ in the Lagrangian,
the choice of $\ma$ makes no difference {\em in theory} for the purpose
of finding primal solutions in the form $\hat f = f^{(H)}(\mathcal{D},\ma,x)$ 
for critical points $\lambda$. With a different value of $\ma$ the location of the critical
points simply scales proportionally while $\hat f$ remains the same. 
In practice, however, we find it convenient to choose $\ma$ to be somewhat large.
In particular, this makes the convexity condition \eqref{eq:convexity_cond} easier to satisfy 
with numerically chosen functions $\lambda$ without having to worry about scaling down their amplitude.
The choice of $\ma$ in principle also has some effect on numerical schemes and stopping criteria.
We will discuss these issues further in Section~\ref{sec:dde_numerics} and Appendix~\ref{app:beta}.
We note that with other choices of $H$, parameters in its definition may not lead to this  kind of scaling symmetry.

\subsection{Dual NIE Formulation}
\label{sec:NIEform}

We will derive an alternative dual variational formulation 
by starting from the nonlinear integral equation \eqref{eq:int_cm1} instead of 
the differential-difference equation \eqref{eq:primal}
and formulating the problem in terms of a corresponding dual field $\nu$.
This leads to some differences in terms of the approximations that are natural to make
and the results obtained. 
We will work with both approaches and compare them at the end.

In the simplest case when $C_1=0$, we can consider $\nu=-\lambda'$, and we can then write 
\[
\lambda_{x-1}-\lambda_{x+1} = \int_{x-1}^{x+1}\nu(y)\,dy = \int_{-1}^1 \nu(x+z)\,dz.
\]
Define the right-hand side to be $K\nu(x)$. Then $K$ is a linear convolution operator satisfying
\begin{equation}\label{d:Knu}
K\nu (x) = \int_{-\infty}^\infty \Lambda(x-y)\nu(y)\,dy, \quad 
\Lambda(z) = \begin{cases}
    1 & |x|\le 1,\\
    0 & |x|>1.
\end{cases}
\end{equation}
Note that if $\nu$ is locally integrable on $\R$ then $K\nu$ is continuous, and if $\nu$
is periodic then $K\nu$ is periodic with the same period.
And, if $\mu$ and $\nu$ are $2L$-periodic locally square-integrable functions, then 
\begin{equation}\label{eq:Kadjoint}
   \int_{-L}^L \mu \,K\nu\,dx = \int_{-L}^L \nu \,K\mu\,dx. 
\end{equation}
Indeed, since the integral of any translate of a periodic function 
over any full period is the same, 
\begin{align*}
   \int_{-L}^L \mu\,K\nu\,dx 
    &= \int_{-L}^{L} \int_{-1}^1 \mu(x)\nu(x+z)\,dz\,dx
    =  \int_{-1}^1  \left(\int_{-L}^{L}  \mu(x)\nu(x+z)\,dx\right) dz
    \\ 
    &=  \int_{-1}^1  \left(\int_{-L}^{L}  \mu(x-z)\nu(x)\,dx\right) dz
    = \int_{-L}^L \nu\,K\mu\,dx \,.
\end{align*}

We will seek wave profiles $f$ as perturbations of a constant $\uinf$, taking the form
\begin{equation}
    \label{d:fform}
    f(x) = \uinf + w(x) \,.
\end{equation}
For periodic waves we require $w$ to be $2L$-periodic on the real line, and for solitary waves we say $L=\infty$ and
require $w(x)\to0$  as $|x|\to\infty$.  In these terms, equation~\eqref{eq:int_cm1} takes the form
\[
\uinf + w + \frac12 K\left( \uinf^2 + 2\uinf w + w^2\right) = C_1\,.
\]

We claim that it is no loss of generality to require 
\begin{equation}
  \label{eq:C1}
C_1=\uinf + \uinf^2\,.
\end{equation}
Indeed, if $f$ is a solitary wave profile, evidently \eqref{eq:C1} must hold.
If instead $f$ is a $2L$-periodic solution of \eqref{eq:int_cm1} with $L<\infty$, then
because $K(f^2)$ is $2L$-periodic and $K(1)=2$, 
we find
\[
\int_{-L}^L \tfrac12 K(f^2)\,dx = \int_{-L}^L f^2\,dx = 
\int_{-L}^L (-f + C_1)\,dx \le 
\int_{-L}^L (f^2 + \tfrac14 + C_1)\,dx\,,
\]
since $-f\le -f+(f+\frac12)^2 = f^2+\frac14$. 
Thus $\frac14+ C_1\ge0$ and \eqref{eq:C1} follows with $\uinf=-\frac12\pm\sqrt{\frac14+C_1}$. 

Equation \eqref{eq:int_cm1} now becomes equivalent to 
\begin{equation}
    \label{eq:int_wform}
  {  w + \uinf K w + \frac12 K(w^2) = 0\,,}
\end{equation}
which may be written more explicitly as the equation
\[
w(x) + \int_{x-1}^{x+1} \left(\uinf w(y)+\frac12 w(y)^2\right)\,dy =0\,.
\]
We are ready next to develop a dual variational formulation for equation~\eqref{eq:int_wform}.
Letting $\nu$ be a $2L$-periodic dual field and noting that \eqref{eq:Kadjoint} should hold with $\mu=w$ and $w^2$,
we define a pre-dual functional by 
\begin{equation} \label{d:hatSnu}
    \hat\Snu[w,\nu] = 
    \int_{-L}^L \left(w (\nu +\uinf K\nu) + \frac{w^2}2 K\nu + {\mathfrak{H}(w,x)}\right)dx
    = \int_{-L}^L {\mathfrak{L}_\mathfrak{H}}(w,\nu,x)\,dx \,.
\end{equation}
For convenience we take $\mathfrak{H}(w,x)$ in \eqref{eq:predual} in the modified form
\begin{equation}\label{eq:aux_nu}
    \mathfrak{H}(w,x)  = \frac{\ma}2\left((w-\bar w)^2-\bar w\,^2\right)
    = \ma\left(\frac{w^2}2 - w \bar w \right)  ,
\end{equation}
where $\bar w = \bar w(x)$ is a fixed base state, whence
\begin{equation}\label{d:LHnu}
{\mathfrak{L}_\mathfrak{H}}(w,\nu,x) = \frac{w^2}2(\ma+K\nu) + w(\nu+\uinf K\nu-\ma\bar w) \,.
\end{equation}

Aiming toward a well-posed dual maximization problem, notice that 
\begin{equation}\label{eq:LHnu}
    \inf_{w\in\R} {\mathfrak{L}_\mathfrak{H}}(w,\nu,x) = \begin{cases}
        -\infty & \text{ if $\ma+K\nu(x)<0$,}\\
        -\infty & \text{ if $\ma+K\nu(x)=0$ and $\ma\bar w(x)-\nu(x)-\uinf K\nu(x)\ne0$,}\\
        -\frac12(\ma+K\nu)\hat w\,^2 &\text{ otherwise},
    \end{cases}
\end{equation}
where $\hat w$ is given by a  Dual-to-Primal relation in the form
\begin{equation}\label{eq:DtPnu}
\hat w  = \begin{cases}
\displaystyle \frac{\mathsf{a} \bar w-\nu-\uinf K\nu}{\mathsf{a} +K\nu} 
    &\text{ if \  $\ma + K\nu \ne 0$,}
    \\
 0  & \text{ if \ $\ma + K\nu = 0$.}
\end{cases}
\end{equation}

Now, for $0<L<\infty$ 
and for any $\nu$ that is $2L$-periodic and locally square-integrable, we define
\begin{equation}\label{d:Snu1}
    \Snu[\nu] = \inf_{w} \hat\Snu[w,\nu].
\end{equation}
where the inf is taken over all  $w$ that are $2L$-periodic and locally square-integrable. 
If $L=\infty$ we require $v\in L^2(\R)$ and define $\Snu[\nu]$ by the same formula,
taking the inf over all $w\in L^2(\R)$.

Given any such $\nu$, define the sets
\begin{equation}
\begin{array}{ll}
   & N_\nu = \{x\in(-L,L): \ma+K\nu <0\}, 
     \\[6pt]
   & Z_\nu = \{x\in(-L,L): \ma+K\nu =0 \ \text{ and }\ \ma \bar w - \nu-\uinf K\nu\ne0 \}. 
\end{array}
\end{equation}
Then with $|\cdot|$ denoting the Lebesgue measure of a set,
from \eqref{eq:LHnu} we infer that
\begin{equation}\label{eq:Snu_formula1}
\begin{array}{ll}
    \displaystyle 
    \Snu[\nu] =  \int_{-L}^L \left(-\frac12 (\ma+ K\nu) \hat w\,^2\right)  dx  
    &\text{ if $|N_\nu|=0$ and $|Z_\nu| =0$,}
    \\[10pt]
\Snu[\nu] =  -\infty  &\text{ if $|N_\nu|>0$ or $|Z_\nu|>0$.}
\end{array}
\end{equation}

At a state $\nu$ subject to the strict convexity condition 
\begin{equation}\label{c:s_convexity}
    \mathscr{C}_s : \quad \mathsf{a} + K\nu >0 ,
\end{equation}
we have the formula
\begin{equation}\label{eq:Snu_finite}
    \Snu[\nu] = \int_{-L}^L \left(-\frac12 \frac{(\ma\bar w - \nu-\uinf K\nu)^2}{\ma+K\nu}\right)dx\,.
\end{equation}
Let us compute (formally) the first variation of $\Snu$ in a direction $\delta\nu$ in this case.
We find
\begin{align}
    \delta^{(1)}\Snu[\nu;\delta\nu] &= 
   \int_{-L}^L \left(\hat w(I+\uinf K)\delta\nu + \frac{\hat w\,^2}{2}(K\,\delta\nu) \right) dx
   \nonumber \\ &=
    \int_{-L}^L \delta\nu \left((I + \uinf K)\hat w +\frac12 K(\hat w\,^2) \right)dx,
    \label{eq:calS_var1}
\end{align}
due to the self-adjointness property \eqref{eq:Kadjoint} of the convolution operator $K$.  
We find that for a maximizer of $\Snu[\nu]$ that satisfies the strict convexity condition~\eqref{c:s_convexity},
{\em necessarily equation~\eqref{eq:int_wform} holds.}

We point out that here we have a consistency property similar to that in the previous subsection.
Namely, if $\bar w$ happens to be a ($2L$-periodic) solution to \eqref{eq:int_wform}, then 
with $\nu=0$ we have $\hat w=\bar w$ and the first variation vanishes in \eqref{eq:calS_var1}.
By concavity, $\nu=0$ is a then a maximizer of $\Snu$, regardless of any degeneracies as will be discussed 
in Section~\eqref{sec:analysis} below.

We remark that the functional $\Snu$ enjoys a scaling symmetry in terms of the amplitude parameter $\ma$
just like the one for $S$ previously described.
Thus $\ma$ can be chosen at will for convenience in numerical computations.
Also we mention that although the base state $\bar f$ from the previous subsection
formally corresponds to $\uinf + \bar w$ here, we may expect some differences on a bounded interval,
because $\bar w$ is considered here to be $2L$-periodic outside $\Omega=(-L,L)$,
while $\bar f$ need not be defined there.  Thus, e.g., we have no reason to expect that 
$\hat f = \uinf+\hat w$ for respective maximizers of $\tilde S[\lambda]$ and $\Snu[\nu]$ when 
$\bar f = \uinf + \bar w$ on $\Omega$.

\section{Analysis of concave maximization problems}\label{sec:analysis}

In principle, by defining the pre-dual functional and just algebraically eliminating the primal 
field by substituting in the DtP mapping,
one obtains a dual functional whose critical points should provide solutions to the primal equation,
as long as the denominator in the DtP mapping is non-vanishing.  But for the purposes of analysis, 
it is natural to study the dual function defined by minimization of the pre-dual over 
primal fields, as we have described.  In this section we develop several analytical facts
about the resulting concave maximization problem, for both the DDE and NIE formulations.

\subsection{Analysis for the dual DDE formulation}\label{ss:dualDDE}
We study the DDE formulation specified on a bounded interval with homogeneous boundary condition on dual fields.
Given $0<L<\infty$, let 
\begin{equation}
\Omega = (-L,L), \qquad \Omega_1=(-L-1,L+1).
\end{equation}
Consider the functional $\tilde{S}$ from \eqref{eq:sup_inf} and \eqref{eq:tilde_S} for dual fields restricted to 
lie the Hilbert space of functions $\lambda\in H^1_0(\Omega)$
considered as equal to zero outside $\Omega$.
We can take the norm on this space to be 
\[
\|\lambda\|_{H^1_0(\Omega)} = \|\lambda'\|_{L^2(\Omega)} \,.
\]
 Any such $\lambda \in C^0(\Omega)$, and satisfies the Poincare inequality. 
 Below, we will use the notation $\| \cdot \|$ to denote either $\| \cdot \|_{L^2(\Omega)}$ or
 $\|\cdot\|_{L^2(\Omega_1)}$ as appropriate in context.
 The pre-dual functional $\widehat{S}_L$ is affine in $\lambda$ and is evidently continuous in $\lambda$ for each $f \in L^2(\Omega_1)$. Then the infimum over $f$ renders $\tilde{S}$ an upper semicontinuous functional posed on $H^1_0(\Omega)$ . 
 We summarize the analytic properties of $\tilde S$ in the following result.
 \begin{proposition}\label{prop:tildeS} Let $0<L<\infty$ and let $\bar f\in L^2(\Omega_1)$.
    The functional $\tilde S$ defined by \eqref{eq:sup_inf} is given by \eqref{eq:tilde_S} and maps
    $H^1_0(\Omega)$ into $[-\infty,\half\ma\|\bar f\|^2]$. 
    Moreover, $\tilde S$ is concave and upper semicontinuous.
    The interior of its domain $\dom(\tilde S)=\{\lambda\in H^1_0(\Omega): \tilde S[\lambda]>-\infty\}$
    is the set 
    \[
    \interior\dom\tilde S = \{\lambda\in H^1_0(\Omega): 
    \Delta_\lambda=\ma + \lambda_{x-1}-\lambda_{x+1}>0 \text{ on }\Omega_1\}. 
    \]
    Furthermore, $-\tilde S$ is coercive, and $\tilde S$ achieves a maximum.
    If some maximizer $\lambda$ lies in $\interior\dom\tilde S$, then the function defined on $\Omega_1$ by 
     \[
     \hat f(x) = 
     \frac{\ma \bar f + \lambda'}{\ma + \lambda_{x-1}-\lambda_{x+1}}  
     \]
     is absolutely continuous inside $\Omega$ and is a strong solution of \eqref{eq:primal} there. 
 \end{proposition}
This result gives an existence result for some (weak) solution of the primal equation \eqref{eq:primal}
on $\Omega$ satisfying homogeneous boundary conditions outside this interval, 
{\em conditional} on having $\tilde S$ admit some maximizer at a state $\lambda$ where
the convexity condition \eqref{eq:convexity_cond} holds strictly inside $\Omega$.
 \begin{proof}
 Evidently  taking the infimum over $f$ in \eqref{eq:sup_inf} always yields 
 $\tilde S(\lambda)\le  \hat S_L(0,\lambda) = \half\ma\|\bar f\|^2$. 
 Clearly $\hat S(f,\lambda)$ depends continuously on $\lambda\in H^1_0(\Omega)$, 
 so the concavity and upper semicontinuity follow 
by basic results in convex analysis, see \cite[Chap.~1]{Brezis2011} or \cite[Prop.~9.2]{BauschkeCombettes2017}.

Regarding the interior of the domain,
if $\Delta_\lambda>0$ on $\Omega_1$ then it is bounded below there 
(it is continuous and approaches $\ma$ at the boundary) and clearly $\tilde S[\lambda]$
is finite in a neighborhood of $\lambda$ in $H^1_0(\Omega)$. And conversely if $\Delta_\lambda=0$
at some point in $\Omega_1$ then a small perturbation of $\lambda$ can make it negative, which makes
$\tilde S$ infinite.
 
Next we prove coercivity.
Given  $\lambda$ at which $\tilde{S}[\lambda]$ 
in \eqref{eq:tilde_S} is finite and using \eqref{eq:dtp}-\eqref{eq:functional_1},
we find
\begin{align} \label{eq:tildeS_finite}
    \tilde{S}[\lambda] = 
     \frac{\ma}2 \|\bar{f}\|^2
     - \half \int_{\Omega_1 {\backslash \Omega_1^*}}  \frac{(\lambda' + \mathsf{a} \bar{f})^2}{\mathsf{a} + \lambda_{x - 1} - \lambda_{x+1}}\, dx, 
\end{align}
where $\Omega_1^*$ is the subset of $\Omega_1$ where $\Delta_\lambda$ and $\lambda' + \mathsf{a} \bar{f}$ both vanish. 
Define
\[
\lambda_m = \sup_{x \in \Omega}  |\lambda(x)|\,,
\]
and note that
$0 \leq \mathsf{a} + \lambda_{x-1} - \lambda_{x+1} \leq \mathsf{a} + 2 \lambda_m$
so that 
\[
\frac{ - (\lambda' + \mathsf{a} \bar{f})^2 } {\mathsf{a} + \lambda_{x-1} - \lambda_{x+1}} \leq  \frac{- (\lambda' + \mathsf{a} \bar{f})^2 } {\mathsf{a} + 2 \lambda_m}\,,
\]
{and that $\lambda' + \mathsf{a} \bar{f} = 0$ on $\Omega_1^*$}. Then
\begin{align*}
    \left(\tilde{S}[\lambda] - \frac{\ma}2 \|\bar{f}\|^2\right)\cdot2(\ma+2\lambda_m) & 
    \leq  
    - \int_{\Omega_1 \backslash \Omega_1^*} (\lambda' + \mathsf{a} \bar{f})^2\, dx = 
    - \int_{\Omega_1} (\lambda' + \mathsf{a} \bar{f})^2\, dx  \\
     &\leq  \int_\Omega \left(-(\lambda')^2 -2\ma\bar f \lambda'\right) dx 
      \leq 2\ma \|\bar f\|\|\lambda'\| - \|\lambda'\|^2  \,.
\end{align*}
Since $\|\lambda'\| = \|\lambda\|_{H^1_0(\Omega)}$ it follows that $-\tilde S$ is coercive on $H^1_0(\Omega)$, 
which means that
\[
\tilde{S}[\lambda] \to - \infty \quad \mbox{ as } \quad \| \lambda \|_{H^1_0(\Omega)} \to \infty.
\]

A maximizer of $\tilde S$ therefore exists, because $-\tilde S$ is proper (i.e., finite somewhere)
and a standard convex analysis result states that a proper, convex and lower semicontinuous function 
that is coercive has a minimizer.  (E.g., see \cite[Thm.~11.10]{BauschkeCombettes2017}.)

Since $\Delta_\lambda>0$ for each $\lambda\in \interior\dom\tilde S$, it is clear from 
\eqref{eq:tildeS_finite} (with $\Omega_*$ empty) that $\tilde S$ is Frech\`et differentiable 
in a neigborhood, with (directional) derivative given by the first variation. 
Let us next compute this first variation at such a state. 
Recall that $\lambda$ and its variations $\delta\lambda$ 
{are taken to vanish} outside the interval $\Omega=(-L,L)$.
Then $\lambda_{x-1}$ vanishes for $x<-L+1$, so 
\begin{align*}
   \int_{\Omega_1}
   \hat f^2\lambda_{x-1}\,dx 
   = \int_{-L+1}^{L+1}
   \hat f^2\lambda_{x-1}\,dx 
   = \int_\Omega
   \hat f_{x+1}^2\lambda\,dx ,
\end{align*}
and similarly for integrands with $\hat f^2\lambda_{x+1}$ and with $\lambda$ replaced by $\delta\lambda$.
We find therefore that, like for the computation of the variation on the whole line, 
\begin{align}\label{e:del1SL}
    \del^{(1)}S_L[\lambda;\del \lambda] 
    &= \int_{-L}^{L} \left( -\del \lambda' \,\hat{f} + \del\lambda\cdot\frac{1}{2} (\hat{f}\,^2_{x+1} - \hat{f}\,^2_{x-1}) \right) dx.
\end{align}
Thus, at a maximizer of $\tilde S$ where $\Delta_\lambda>0$ in $\Omega_1$, we can infer that 
$\hat f\in L^2(\Omega_1)$ and is a weak solution of the primal equation \eqref{eq:primal} on the interval $\Omega=(-L,L)$.
Since the functions $\hat f_{x\pm 1}$ are square integrable in $\Omega$, 
equation \eqref{eq:primal} holds strongly in $L^1(\Omega)$ and 
(by integration) it follows $\hat f$ is absolutely continuous on $\Omega$.
 \end{proof}

{\bf Remark.} From \eqref{eq:primal}, inductively we can infer higher regularity on smaller sets:
$\hat f$ is $C^1$ inside the interval $(-L+1,L-1)$, $C^2$ inside the interval $(-L+2,L-2)$, etc. 
Even if $\bar f$ is smooth, however, $\hat f$ may be discontinuous at $\pm L$ due to a potential
discontinuity in $\lambda'$ at the endpoints of $\Omega$. In this case we could infer that
$\hat f$ is $C^1$ throughout $\Omega$ except at $\pm(L-1)$, and further that 
$\hat f$ is smooth in $\Omega$ outside the set of points $\pm(L-k)$ where $k\in \N$.

\subsection{Analysis for the dual NIE formulation}\label{ss:NIEanalysis}

In this subsection we study the NIE formulation on both finite and infinite intervals.
Using this formulation we will obtain conditional existence results for $C^\infty$ periodic solutions
of the primal equation \eqref{eq:int_wform}, {instead of imposing homogeneous boundary conditions outside 
a bounded interval}.  The price is that the coercivity analysis turns out to be more involved, and 
we do not establish coercivity for all values of the parameter $\uinf$.

Throughout this section, we fix $\ma>0$ and $\bar u\in L^2(\Omega)$.
For any $L\in(0,\infty]$, let $\Omega=(-L,L)$.  
If $L=\infty$, the formula \eqref{d:Knu} defines $K$ as a bounded linear operator
from $L^2(\Omega)$ into $H^1(\Omega)$. The same holds if $0<L<\infty$, by
extending $\nu\in L^2(\Omega)$ to be $2L$-periodic and considering $K\nu$ 
as the restriction of \eqref{d:Knu} to $\Omega$.  For later use we also define
$K_0\nu$ to be given by extending $\nu$ to be zero outside $\Omega$
and restricting formula \eqref{d:Knu} to $\Omega$.

We summarize several properties of the dual functional $\Snu$ in the following proposition.
Note that for any $\nu$ and $w$ in $L^2(\Omega)$,
the pre-dual integrand $\mathfrak{L_H}$ in \eqref{d:LHnu} is integrable on $\Omega$,
which is evident since $K\nu$ is in $L^2(\Omega)$ and is continuous and bounded. 

\begin{proposition} \label{prop:Snu_def}
    The functional $\Snu$ defined by \eqref{d:Snu1}, i.e., by 
    \[
    \Snu[\nu] = \inf_{w\in L^2(\Omega)} \int_{-L}^L \mathfrak{L_H}(w,\nu,x)\,dx,
    \]
    is given by \eqref{eq:Snu_formula1} and maps $L^2(\Omega)$ into $[-\infty,0]$. 
    Moreover, $\Snu$ is concave and upper semicontinuous. 
    The interior of its domain $\dom\Snu=\{\nu\in L^2(\Omega):\Snu[\nu]>-\infty\}$ 
    is the set where \eqref{c:s_convexity} holds, i.e.,
    \[
    \interior\dom\Snu = \{\nu\in L^2(\Omega): \ma+K\nu>0\text{ on $\Omega$}\}.
    \]
    If additionally $-\Snu$ is coercive, then a maximizer of $\Snu$ exists.
    Also, if $\Snu$ has some maximizer $\nu\in\interior\dom\Snu$, then the function $\hat w$ given by \eqref{eq:DtPnu}
    is a $C^\infty$ solution of \eqref{eq:int_wform}.
\end{proposition}

\begin{proof}
   Since $\mathfrak{L_H}(0,\nu,x)=0$, the infimum defining $\Snu[\nu]$ is non-positive. 
   Since  $\hat\Snu(w,\nu)$ is continuous and affine in $\nu$,
   the convexity and lower semicontinuity of $-\Snu$ follow as in the previous subsection.

From \eqref{eq:LHnu}--\eqref{eq:DtPnu}, evidently if \eqref{c:s_convexity} holds 
then $\Snu[\nu]$ is finite and remains finite in a neighborhood of $\nu$ in $L^2(\Omega)$, 
so $\nu\in\interior\dom\Snu$. On the other hand, if $\Snu[\nu]$ is finite but \eqref{c:s_convexity}
does not hold, then $\ma+K\nu(x)=0$ for some $x\in\Omega$. For any nearby $\tilde\nu<\nu$ locally near $x$,
we get $\ma+K\tilde\nu(x)<0$, whence $\Snu[\tilde\nu]=-\infty$ so $\nu\notin\interior\dom\Snu$.

We note that $\hat w = \bar w$ when $\nu=0$, so $\Snu[0]=-\frac12{\ma}\|\bar w\|_{L^2(\Omega)}^2$.
Then $-\Snu$ is proper, convex and lower semicontinuous, hence $\Snu$ has a maximizer.

Finally, if $\Snu$ has a maximizer $\nu$ in the interior of its domain, then $\Snu$ is differentiable 
at $\nu$ and the first variation must vanish, implying that $\hat w$ solves \eqref{eq:int_wform} as shown in
section~\ref{sec:NIEform} above. We claim $\hat w$ is $C^\infty$.
The function $K\nu$ is in $H^1(\Omega)$, so is continuous (and vanishes in the limit $x\to\pm\infty$ if $L=\infty$),  and $\mathsf{a}+K\nu$ has a positive minimum and bounded maximum:
\begin{equation}\label{b:betabounds}
 0<\mathsf{a}_{\rm min} \le  \mathsf{a}+K\nu \le \mathsf{a}_{\rm max}<\infty.
\end{equation}
The function $\hat w$ given by \eqref{eq:DtPnu} is then in $L^2(\Omega)$. 
As $K\hat w$ and $K(\hat w\,^2)$ are absolutely continuous and $\hat w$
satisfies \eqref{eq:int_wform},  
$\hat w$ is absolutely continuous also, 
whence by bootstrapping (induction) we find $\hat w$ is $C^k$ for all $k$.
\end{proof}

Recall that to say $-\Snu$ is {coercive} means that 
$-\Snu[\nu]\to\infty$ as $\|\nu\|_{L^2(\Omega)}\to\infty$.
In most circumstances we can prove the function $-\Snu$ is indeed coercive.
This is easiest when the bounded operator $I+\uinf K$ on $L^2(\Omega)$ has bounded inverse,
which is evidently the case when $|\uinf|$ is small enough, for example.
\begin{proposition}
    \label{prop:coercive}
    Suppose $I+\uinf K$ has bounded inverse on $L^2(\Omega)$.  Then $-\Snu$ is coercive.
\end{proposition}
\begin{proof} 
Throughout the proof, $\|\cdot\|$ denotes the norm in $L^2(\Omega)$.
We begin with two preliminary estimates. First, due to the invertibility hypothesis,
there is a constant $\alpha>0$ such that for all $\nu\in L^2(\Omega)$,
\[
\|(I+\uinf K)\nu\| \ge \alpha \|\nu\| \,,
\]
whence 
\[
\|\ma \bar w - (I+\uinf K)\nu\| \ge     
\|(I+\uinf K)\nu\| - \|\ma \bar w\| \ge 
\alpha \|\nu\| - \|\ma \bar w\| \,.
\]
Second, by the Cauchy-Schwarz inequality we find
(recall we extend $\nu$ as $2L$-periodic if needed)
\[
|K\nu(x)| = \left|\int_{x-1}^{x+1} 1\cdot \nu(y)\,dy \right| \le C_K \|\nu\| \,,
\]
where $C_K$ is a constant independent of $\nu$ (and equal to $\sqrt{2}$ if $L>1$).
Therefore 
\[
\ma_{\rm max} = \max_x |\ma +K\nu|\le  \ma + C_K\|\nu\|.  
\]
2. Now, define the set
\begin{equation}
    \label{d:Pnu}
P_\nu = \{x\in\Omega: \ma + K\nu>0\}.
\end{equation}
If $\Snu[\nu]>-\infty$ then by \eqref{eq:Snu_formula1} and the definition 
of $\hat w$ in \eqref{eq:DtPnu},
\[
-2\Snu[\nu] = \int_{P_\nu} \frac{(\ma\bar w-\nu-\uinf K\nu)^2}{\ma+K\nu}\,dx \ge 
\frac1{\ma_{\rm max}} \int_{P_\nu} | \ma \bar w-\nu-\uinf K\nu|^2\,dx\,.
\]
However, because $|N_\nu|=|Z_\nu|=0$, we have $\ma\bar w-\nu-\uinf K\nu=0$ a.e.~on the complement of $P_\nu$. 
Hence, the domain of integration can be extended from $P_\nu$ to all of $\Omega=(-L,L)$.
Thus for all $\nu\in L^2(\Omega)$ we have 
\begin{equation}
    -2\Snu[\nu] \ge 
\frac{\|\ma\bar w-(I+\uinf K)\nu\|^2}{\ma+C_K\|\nu\|}
\ge 
\frac{(\alpha\|\nu\| - \|\ma\bar w\|)^2} {\ma+C_K\|\nu\|} \,.
\end{equation}
The right-hand side tends to $\infty$ as $\|\nu\|\to\infty$, which establishes
the coercivity as claimed.
\end{proof}

We can determine precisely when bounded invertibility holds by locating the 
spectrum of $K$ using the Fourier transform, with the following result.
Define the Fourier transform of $\nu\in L^2(\Omega)$ by 
\[
\calF\nu(\xi) = \int_{-L}^L e^{-i\xi x}\nu(x)\,dx 
\]
on the Fourier domain given by $\xi\in\Omega^*$, where
\begin{equation} \label{d:Omega*}
\Omega^* := \begin{cases}
(-\infty,\infty) & \text{if $L=\infty$,}\\
\{k\pi/L: k\in\Z\} &\text{if $0<L<\infty$}.
\end{cases}
\end{equation}
Then by a straightforward computation,
\[
\calF(K\nu)(\xi) = (2\sinc\xi) \calF\nu(\xi), 
   \quad \text{where \ \ } 
   \sinc(\xi) = \begin{cases}
       \frac{\sin \xi}{\xi} & \xi\ne 0,\\ 1 & \xi=0.
   \end{cases}
\]
It follows that $(I+\uinf K)\nu = w$ if and only if 
\[
(1+2\uinf \sinc \xi) \calF\nu(\xi) = \calF w(\xi) \quad\text{for all $\xi\in\Omega^*$,}
\]
and $I+\uinf K$ has bounded inverse if and only if $(1+2\uinf\sinc\xi)^{-1}$ is
uniformly bounded on $\Omega^*$.

Notice that the range $\{2\sinc\xi:\xi\in\Omega^*\}$ is a closed interval $[-\sigma_0,2]$
with $-\sigma_0\approx -0.434467$ for $L=\infty$, and is a discrete sequence of not-necessarily-distinct
values converging to zero for $0<L<\infty$. Thus we find the following.
\begin{proposition}
   \label{prop:invertible} 
   The operator $I+\uinf K$ has bounded inverse if and only if 
   \[ 1+2\uinf\sinc\xi \ne 0 \qquad\text{for all $\xi\in\Omega^*$.}\]
   In particular $I+\uinf K$ is invertible whenever  
   \[
   -\frac12<\uinf<1/\sigma_0\approx 2.30167.  
   \]
   For $0<L<\infty$ each number $1+2\uinf\sinc(k\pi/L)$ ($k\in\Z$) is an eigenvalue of finite mulitiplicity.
\end{proposition}

{\bf Remark.}
The condition in proposition~\ref{prop:invertible} has a physical meaning. 
Namely, it corresponds to the {\em phase-speed non-matching condition} mentioned in the introduction.
If we undo the scaling $f\mapsto -2cf$, then the constant solution $f=\uinf$ in \eqref{eq:int_cm1}
corresponds to the constant state $\bar u = -c\uinf$ for equation \eqref{eq:dB}.
For the linearization at this state, waves $e^{i\xi j-i\omega t}$ must satisfy the dispersion relation 
\begin{equation}
\omega = \bar u\sin \xi .
\end{equation}
Then the phase-speed non-matching requirement that $c\ne \omega/\xi$ is equivalent to the condition
that $1+2\uinf\sinc\xi\ne0$ for all $\xi\in\Omega^*$ as stated in the proposition.

\medskip
In general, when $L$ is finite we can also prove coercivity for any value of $\uinf\ne-\frac12$,
with a somewhat more involved argument.  

\begin{proposition}
    \label{prop:coercive2}
    Suppose $\Omega = (-L,L)$ is bounded.
    If $\uinf\ne-\frac12$, then $-\Snu$ is coercive on $L^2(\Omega)$.
    If $\uinf=-\frac12$, the functional $-\Snu$ is coercive on the subspace of $L^2(\Omega)$ 
    consisting of functions with mean zero.
\end{proposition}

\begin{proof}
1. For the value $\uinf=-\frac12$, the self-adjoint operator
$I+\uinf K$ has one-dimensional kernel spanned by the constant function $1$. 
Restricting $\Snu$ to the orthogonal complement of this kernel, the proof of coercivity
works the same as the proof of proposition~\ref{prop:coercive}.

2. Suppose $\uinf\ne-\frac12$ but $I+\uinf K$ is singular.
Decompose $L^2(\Omega)$ into the finite-dimensional kernel $V$ of $I+\uinf K$ 
and its orthogonal complement $W$ on which $I+\uinf K$ has bounded inverse.
Each nonzero $\nu\in V$ is a trigonometric polynomial with mean zero, so
it is impossible that $K\nu(x)\ge0$ for all $x$. 
Then by a compactness argument, there exists some $\sigma_1<0$ such that 
whenever  $\|\nu\|= 1$ then $\min_x K\nu(x)\le \sigma_1$.

3. Let $(\nu_n)$ be a sequence in $L^2(\Omega)$ with $\|\nu_n\|\to\infty$
as $n\to\infty$. For each $n$ decompose $\nu_n$ as $\mu_n+\eta_n$ with $\mu_n\in V$
and $\eta_n\in W$.  There are now two cases: 
(i) Suppose that for some postive constant $C$, 
$\|\mu_n\|\le C\|\eta_n\|$ for all $n$.
Then $\|\nu_n\|\le (1+C)\|\eta_n\|\to\infty$ as $n\to\infty$, and 
\[
\|\ma\bar w - (I+\uinf K)\nu_n\| = \|\ma\bar w -(I+\uinf K)\eta_n\|\ge \alpha\|\eta_n\|-\|\ma\bar w\| \ge \hat\alpha\|\nu_n\| - \|\ma\bar w\|,
\]
where $\hat\alpha = \alpha/(1+C)$. One then infers that $-\Snu[\nu_n]\to\infty$
by the same argument as before.

(ii) If it is false that case (i) holds for some $C$,
then there must be a subsequence of $(\nu_n)$ (denoted the same)
such that $\|\mu_n\|\ge n\|\eta_n\|$ for all $n$, with $\|\mu_n\|\to\infty$.
By step 1, we note that 
\[
\min_x K\mu_n(x) \le \|\mu_n\|\sigma_1 <0,
\]
and since $K\nu_n = K\mu_n+K\eta_n$, 
\[
\min_x (\ma + K\nu_n) \le \ma + \|\mu_n\|\sigma_1 + C_K\|\eta_n\|
\le \ma + \|\mu_n\| (\sigma_1 + C_K/n). 
\]
This is strictly negative for sufficienty large $n$, and when this is the case
we must have $-\Snu[\nu_n]=\infty$.  
This finishes the proof.
\end{proof}

We suspect coercivity may always hold when $L=\infty$ as well. 
In any case,  we only get a proof that a solution to the NIE \eqref{eq:int_wform} exists
{\em on the condition} that a maximizer of $\Snu$ exists
that belongs to the interior of the  domain of $\Snu$. 
Presently, despite strong numerical evidence in favor as shown below,
we lack any proof that such a maximizer exists, for any values of $L$ and $\uinf$.

\medskip
{\em Second variation.} In general, at any point in the interior of the domain of $\Snu$,
its second variation can be found by substituting $\nu+t\,\delta w$ into \eqref{eq:DtPnu} and \eqref{eq:calS_var1}
and differentiating at $t=0$ to find that 
\[
\left.\frac{d}{dt} \hat w\right|_{t=0} = 
-\frac{(I+\uinf K)\delta w}{\ma+K\nu} - \frac{\hat w K(\delta w)}{\ma+K\nu}
 = -\frac{(I+\hat f K)\delta w}{\ma+K\nu} \,,
\]
where 
\begin{equation}\label{eq:hatf_uw}
\hat f = \uinf + \hat w = \frac{\ma(\uinf+\bar w)-\nu}{\ma + K\nu} \,.
\end{equation}
Then differentiation of \eqref{eq:calS_var1} yields
\begin{equation}\label{eq:Snu_var2}
    \delta^{(2)}\Snu[\nu;\delta\nu,\delta w] =  
    -\int_{-\infty}^\infty \delta\nu (I + K \hat f)(\mathsf{a}+K\nu)^{-1} (I+\hat f K) \delta w\,dx.
\end{equation}
Here $K\hat f$ is regarded as composition of operators with $K\hat f \,g = K(\hat f g)$. Since $K$ is self-adjoint we find
\begin{equation}\label{eq:2ndvarcalS}
    \delta^{(2)}\Snu[\nu;\delta\nu,\delta \nu] 
    =  - \int_{-\infty}^\infty 
    (\mathsf{a} + K\nu)^{-1} | (I+\hat f K)\delta \nu|^2\,dx\,.
    \end{equation}
Indicating explicitly the dependence upon the parameter $\ma$, this enjoys the scaling property
\begin{equation}
    \delta^{(2)}\Snu_\ma[\nu;\delta\nu,\delta\nu] = \ma\inv\,\delta^{(2)}\Snu_1[\ma\inv\nu;\delta\nu,\delta\nu].
\end{equation}

{\em Critical points and translational invariance.} 
Now let $L\in(0,\infty]$ and suppose $\nu$ is a maximizer of $\Snu[\nu]$ 
belonging to the interior of its domain, so that 
the strict convexity condition \eqref{c:s_convexity} holds.
Then as stated in proposition~\ref{prop:Snu_def}, $\hat w$ as given by \eqref{eq:DtPnu} is a 
smooth solution of \eqref{eq:int_wform} so that $\hat f=\uinf+\hat w$ from \eqref{eq:hatf_uw}
is a smooth solution of \eqref{eq:int_cm1} with $C_1=\uinf+\uinf^2$ as in \eqref{eq:C1}.

Now, equation~\eqref{eq:int_cm1} is translation invariant, meaning that if 
$f$ is a solution on $\R$, then the function $x\mapsto f(x+h)$ is a solution for any real $h$.
Differentiating with respect to $h$ at $h=0$ we find that 
\[
f'(x) + \int_{x-1}^{x+1} f(y)f'(y)\,dy = 0, \quad\text{i.e.,} \quad
f'+K(ff') = 0.
\]
Thus the operator $I+Kf$ has $f'$ in its kernel.
Multiplying this equation by $f(x)$ we find that 
\[
f(x)f'(x) + f(x) \int_{x-1}^{x+1} f(y)f'(y)\,dy = 0
\quad\text{i.e.,} \quad  (I+fK)(ff') = 0.
\]
That is, the (adjoint) operator $I+fK$ has the function $ff'$ in its kernel. 

For the maximizer $\nu$ this means that the second variation vanishes in \eqref{eq:2ndvarcalS} 
for the variation
\begin{equation}\label{eq:kernu}
\delta \nu = \hat f \hat f'\,.
\end{equation}
Indeed, $(I+\hat f K)(\hat f\hat f')=0$.
Now consider first the case when $L$ is finite.
The operator $\hat f K$ acting on $L^2(\Omega)$ is then always a compact operator, 
since the embedding of $H^1(\Omega)$ into $L^2(\Omega)$ is compact.
From the Riesz-Schauder spectral theory of compact operators, the eigenvalue $0$
is necessarily an isolated eigenvalue of $I+\hat f K$ and has a finite-dimensional 
(generalized) eigenspace, which we denote by $Z_0$.

In the case $L=\infty$ when $\Omega=\R$, the operator $K$ maps $L^2(\R)$ into $H^1(\R)$ but is not compact.
Because $\hat w$ is smooth with limit 0 at $\pm\infty$, though, the operator $\hat w K$ is compact
on $L^2(\R)$, due to the convenient compactness criteria of \cite{Pego.85}.
If we assume $I+\uinf K$ has bounded inverse on $L^2(\R)$, which is natural to ensure coercivity
according to proposition~\ref{prop:coercive}, then since $\hat f=\uinf+\hat w$,
the operator $I+\hat f K$ will be the sum of an invertible operator and a compact one,
i.e., Fredholm of index zero. Then the Riesz-Schauder theory ensures again that the
eigenvalue $0$ is isolated with finite-dimensional generalized eigenspace $Z_0$.

{\em Conditional strict coercivity of second variation.} 
We expect, but are unable to prove, 
that $Z_0$ is one-dimensional and spanned by $\hat f\hat f'$.
In any case, if $Y\subset L^2(\Omega)$ is any subspace complementary to $Z_0$,
then necessarily the operator $I+\hat fK$ is bounded below on $Y$, meaning that 
for some constant $\kappa_Y>0$,
\[
\| (I+\hat f K)u\|_{L^2} \ge \kappa_Y \|u\|_{L^2} \quad\text{for all $u\in Y$.}
\]
This means that we have (conditional) strict coercivity in \eqref{eq:2ndvarcalS} 
for variations in $Y$,  with
\begin{equation}
    -\delta^{(2)}\Snu[\nu;u,u]  \ge \frac{\kappa_Y^2}{\mathsf{a}_{\rm max}} \|u\|_{L^2}^2
    \quad \text{for all $u\in Y$.}
\end{equation}
Our numerics suggests that $Z_0$ is one-dimensional and the solitary wave can be chosen even, 
so $\hat f\hat f'$ is odd.  One could take $Y$ to consist of the even functions in $L^2$, then.

\section{Approximation and numerical examples for the DDE formulation}
\label{sec:numerical}
\subsection{Approximation for the DDE formulation}
We approximate weak solutions of the DDE \eqref{eq:primal_in_dual}, 
generating a weak form for solutions on a finite domain $\Omega=(-L,L)$ as follows: 
 We generate a residual by multiplying \eqref{eq:primal_in_dual} with a test function $\delta \lambda$ 
 that vanishes outside $\Omega$ and integrating.
 After integration by parts, this yields:
    \begin{equation}
\label{eq:residual}
    R[\lambda;\del \lambda] := \int_{-L}^{L} \,\left( - \del\lambda'\, \hat{f} + \frac{\del\lambda}{2} (\,\hat{f}\,^2_{x+1} - \hat{f}\,^2_{x-1}\,)\right) dx, 
\end{equation}
where we have eliminated the boundary terms by imposing boundary conditions $\delta\lambda(\pm L) = 0$.
Since the value of $\hat{f}(x)$ from \eqref{eq:dtp} depends upon values of $\lambda$ at $x-1$ and $x+1$,
defining the terms $\hat{f}^2_{x-1}$ and $\hat{f}^2_{x+1}$ in this integrand 
requires that $\lambda$ be defined in 
the {extended} domain $\Omega_2 = (-L-2,L+2)$.
Thus we find it suffices to describe a weak form for the dual problem as follows:

\begin{quote}
Find $\lambda:(-L-2,L+2)\rightarrow \mathbb{R}$,  
satisfying $\lambda(x)=0$ whenever $x\notin (-L,L)$, 
such that for any $\delta \lambda$ satisfying $\delta\lambda(x)=0$ whenever $x\notin (-L,L)$,
\begin{equation}
\begin{gathered}
\label{eq:residual_full}
    \int_{-L}^{L} \,\left( - \del\lambda'\, \hat{f} + \frac{\del\lambda}{2} (\,\hat{f}\,^2_{x+1} - \hat{f}\,^2_{x-1}\,)\right) dx = R[\lambda;\del \lambda] = 0.
    \end{gathered}
\end{equation}
 Here $\hat f(x)$ is determined for $x\in (-L-1,L+1)$ in terms of $\lambda$ by the DtP map \eqref{eq:dtp}.
\end{quote}

The weak form in \eqref{eq:residual_full} is the same problem that is satisfied by a maximizer 
of the functional $\tilde S$ that lies in the interior of its domain, as shown in Proposition~\ref{prop:tildeS}.
By the same arguments as in the proof of that result,
for any $\lambda\in H^1(\Omega_2)$ satisfying the weak formulation \eqref{eq:residual_full},
$\hat f$ is absolutely continuous inside $\Omega=(-L,L)$ and is a strong solution of \eqref{eq:primal} there. 
And $\hat f$ enjoys additional regularity properties as described in the remark following the proposition.

\subsubsection{A modified Newton-Raphson scheme with step-size control}
\label{sec:NR}
The solutions to \eqref{eq:residual_full} are obtained via a Newton Raphson (N-R) scheme based algorithm:
A nonlocal Galerkin Finite Element method has been implemented to approximate discrete solutions to \eqref{eq:residual_full}. We start by considering the variation of $R[\lambda;\delta\lambda]$ (eq.~\eqref{eq:residual_full}) in a direction $d\lambda$ given by:
\begin{equation}
\label{eq:jacobian}
    J|_{\lambda}[d\lambda;\del \lambda] = \int_{-L}^{L} \,\left( -\del \lambda'\, \frac{\p\hat{f}}{\p \dee} d\dee + \del\lambda \left(\,\hat{f}_{x+1}\, \frac{\p\hat{f}_{x+1}}{\p \dee_{x+1}} d\dee_{x+1} - \hat{f}_{x-1}\,\frac{\p\hat{f}_{x-1}}{\p \dee_{x-1}} d\dee_{x-1}\right)\right) dx,
\end{equation}
where \begin{equation*}
\dee_{x+c} \equiv
\{\lambda_{x+c-1},\lambda'_{x+c},\lambda_{x+c+1}\},
\end{equation*}
and
\begin{equation}
\label{qe:dual_deriv}
    \begin{aligned}
       \frac{\p \hat{f}_{x+c}}{\p \,\lambda'_{x+c}} = \frac{1}{\lambda_{x+c-1} - \lambda_{x+c+1} + \mathsf{a}}; \quad \frac{\p \hat{f}_{x+c}}{\p \,\lambda_{x+c \pm 1}} = \pm\frac{\mathsf{a} \bar{f}_{x+c} + \lambda'_{x+c}}{(\lambda_{x+c-1} - \lambda_{x+c+1} + \mathsf{a})^2}.
    \end{aligned}
\end{equation}

In the following, we will use the summation convention on repeated indices. We discretize the extended domain and approximate various fields on it as follows:
\begin{equation*}
\lambda(x) = \lambda^i N^i(x); \qquad \delta \lambda(x) = \delta \lambda^i N^i(x); \qquad d\lambda(x) = d\lambda^i N^i(x),
\end{equation*}
where first-order $C^0(\Omega)$ shape functions $N^i$ are considered, and $i$ runs over the nodes of the discretized extended domain. Let $x^A$ denote the position of any node $A$ on the extended domain and define a set  $\mathcal{S}$ of nodal indices as follows:
\begin{equation*}
    \mathcal{S} = \{B\,|\,x^B\in(-L,L)\}.
\end{equation*}
Our objective is to identify the coefficients $\lambda^A$ for all nodes $A\in\mathcal{S}$ such that the discrete residual generated from \eqref{eq:residual_full}, when equated to $0$, gets satisfied. The discrete residual can be given as:
\begin{equation}
    R^A = \int_{-L}^{L} \,\left( - (N^A)'\, \hat{f} + \frac{N^A}{2} (\,\hat{f}\,^2_{x+1} - \hat{f}\,^2_{x-1}\,)\right) dx,
    \label{eq:discrete_resi}
\end{equation}
where $\hat{f}$ is now depends on the discretized dual field of $\lambda$. Correspondingly, the Jacobian \eqref{eq:jacobian} can be discretized as:
\begin{equation*}
    J|_{\lambda}[d\lambda;\del \lambda] = \delta \lambda^A \, J^{AB}(\lambda) \, d\lambda^B,
\end{equation*}
where
\begin{multline}
    \label{discrete_jaco}
    J^{AB} = \int_{-L}^{L} \,\Bigg( - (N^A)'\left( \frac{\p\hat{f}}{\p \lambda'_x} (N^B)'(x) +\frac{\p\hat{f}}{\p \lambda_{x-1}} N^B(x-1) 
    +\frac{\p\hat{f}}{\p \lambda_{x+1}} N^B(x+1)\right) \\ + 
     N^A\,\hat{f}_{x+1}\left( \frac{\p\hat{f}_{x+1}}{\p \lambda'_{x+1}} (N^B)'(x+1) +\frac{\p \hat{f}_{x+1}}{\p \lambda_x} N^B(x) 
    +\frac{\hat{f}_{x+1}}{\p \lambda_{x+2}} N^B(x+2)\right) \\ - 
     N^A\,\hat{f}_{x-1}\left( \frac{\p\hat{f}_{x-1}}{\p \lambda'_{x-1}} (N^B)'(x-1) +\frac{\p \hat{f}_{x-1}}{\p \lambda_{x-2}} N^B(x-2) 
    +\frac{\hat{f}_{x-1}}{\p \lambda_x} N^B(x)\right)\Bigg) dx,
\end{multline}
and \eqref{eq:dtp} and \eqref{qe:dual_deriv} can be utilized to evaluate the above expression.

To generate corrections for dual field, we implement a modification to the generic N-R scheme based on the following steps:
\begin{equation}
\label{eq:increment}
    \begin{aligned}
     - R^A\left(\lambda^{(k-1)}\right) & =  J^{AB}\left(\lambda^{(k-1)}\right)\, d\lambda^B;\\
         \lambda^{(k)} & = \lambda^{(k-1)} + \alpha \,d\lambda,
    \end{aligned}
\end{equation} 
where $\lambda^{(k)}$ denotes the discretized dual field of $\lambda$  at $k^\text{th}$ iterate and $\alpha$ is a \textit{step-size control} factor. $\alpha=1$ in \eqref{eq:increment} implies a simple N-R scheme. The introduction of $\alpha$ has been motivated below.

In cases where the base state is set far away from any of the  potential solutions of \eqref{eq:primal}, the correction obtained via a simple N-R scheme can potentially lead to a dual field which violates the convexity condition $\mathscr{C}$  \eqref{eq:convexity_cond}. For such cases, we stipulate the following condition on any discretely obtained dual iterate:  
\begin{equation*}
    \mathscr{C}_d : \quad \underset{\Omega}{\mbox{min}}\,\frac{\Delta_\lambda}{\mathsf{a}}>T
\end{equation*}
where $T\in[0,1]$ represents a threshold value such that a large value of $T$ indicates a large denominator. We generally opt for  $T>0.1$ with larger values implying that the dual fields take values away from the convexity boundary. However, smaller values of $T$ can also be chosen; a value of $T=0.01$ yields an approximation with a residual tolerance of $10^{-11}$.

To satisfy $\mathscr{C}_d$, we control the value of $\alpha$ in the N-R iterates via Alg.~\ref{algo}.
 Starting from any fixed base state, we use $\lambda^{(0)}=0$ at each node and $\alpha=1$. Each time a correction leads to a dual field which violates $\mathscr{C}_d$ at any point in the domain, we reduce $\alpha$ by a factor of 2 and re-evaluate the correction until the criteria is met. If the factor $\alpha$ attains too a small value judged by a threshold (set to $0.01$ in the presented calculations), we stop the step-size controlled N-R and declare the primal field $f^{(H)}$ (evaluated at Gauss points) corresponding to the current dual iterate $\lambda$ as the best improvement that can be obtained starting from the current base state. Using this primal field as a base state, we restart the step-size controlled N-R scheme with $\alpha=1$. Such occasional base state resets followed by controlled N-R steps are carried out until convergence on residual \eqref{eq:residual_full} is reached while ensuring that $\mathscr{C}_d$ remains satisfied. The algorithm has been summarized in Alg.~\ref{algo}.
 
 The convergence criteria is given as:
\begin{equation}
    \textbf{tol}: \quad\underset{A}{\mbox{max}}\,|R^A| < \textbf{tol} \quad \forall A\in\mathcal{S},
    \label{eq:converge_NR}
\end{equation}
where $\textbf{tol}$ is a user-defined tolerance. In the following, 

\begin{equation}
\textbf{tol}=10^{-12} \qquad \mbox{for all problems solved.}
\label{eq:tol}
\end{equation}

\begin{table}
{\begin{algorithm}[H]
\caption*{{\bf Algorithm for modified N-R scheme with step-size control}} 
\begin{algorithmic}
\\
\State Choose the values for $L$, $N$ and $\textbf{tol}$. $T=0.95$ and $\mathsf{a}=10^6$ and can be changed. Set $\alpha=1$ and choose a starting base state $\bar{f}$ . Set $c=0$ ($c$ : base state reset counter). 
\\\hrulefill\vspace{1mm}
\begin{enumerate}
\item \label{step:start_globalNR} \textbf{Global Loop}:
\item [] \textbf{WHILE} $c \geq 0$:
    \begin{enumerate}[(i)]
    \item Set $k = 0$ ($k$ : N-R counter)  and $\lambda^{(0)}=0$ at each node. 
    \item Find $\lambda^{(1)}$ using step no.~\ref{step:increment} and set $k=1$.  \item \label{step:abort} \textbf{IF} $\underset{A}{\mbox{max}}\,\big|R^A[\lambda^{(1)}]\big|>\underset{A}{\mbox{max}}\,\big|R^A[\lambda^{(0)}]\big|$ then \textbf{ABORT} 
    \item \label{eq:NR_steps} \textbf{WHILE} $k > 0$ : 
    \begin{enumerate}[(a)]
        \item  \textbf{IF} $\underset{A}{\mbox{max}}\,\big|R^A[\lambda^{(k)}]\big|<\textbf{tol}$ : \textbf{EXIT} and \textbf{GOTO} step no.~\ref{step:end_global_NR}. \vspace{1mm}
        \item \textbf{IF} $\alpha<0.01$ $\rightarrow$ Reset base state (\textbf{GOTO} step no.~\ref{step:resetNR}).  
        \item Find $\lambda^{(k+1)}$ using using step no.~\ref{step:increment}.
        \item \textbf{DO} $k=k+1$ and \textbf{GOTO} step no.~\ref{eq:NR_steps}. 
        \end{enumerate}
    \end{enumerate}

\\[]\hrulefill

    \item \label{step:increment} \textbf{Evaluate Increment}:
    \begin{enumerate}[(i)]
    \item Evaluate $R^A[\lambda^{(k)}]$ and $J^{AB}[\lambda^{(k)}]$.
        \item Find $\lambda^{(k+1)}$ based on current $\alpha$ from \eqref{eq:increment} and check for  condition \textbf{($\mathscr{C}_d$)} (step no.~\ref{step:CC}).
        \item[]  \textbf{IF} $\mathscr{C}_d$ is not satisfied:
         \textbf{DO} $\alpha = \alpha/2$ and redo current step.  
    \item[] \textbf{ELSE}
    return $\lambda^{(k+1)}$.
     \end{enumerate}

\\[]\hrulefill
    \item \label{step:CC} \textbf{Condition $\mathscr{C}_d$}:
    \begin{enumerate}[(i)]
    \\ Evaluate $\Delta_\lambda$ (using \eqref{eq:denominator}) on discretized domain of $\Omega$ based on $\lambda^{(k+1)}$.
    \\ Check for
     $$ \underset{\Omega}{\mbox{max}}\,\frac{\Delta_\lambda}{\mathsf{a}}>T$$
     \end{enumerate}

\\[]\hrulefill
    \item \label{step:resetNR} \textbf{Reset base state and restart N-R}:
    \begin{enumerate}[(i)]
    \\ Evaluate $f^{(H)}$ based on $\lambda^k$ (using \eqref{eq:dtp}) on the Gauss points of discretized domain.
    \\ Set $\bar{f} = f^{(H)}$ at the Gauss points.
    \\ Do $c=c+1$, $\alpha=1$ and restart N-R from step no.~\ref{step:start_globalNR}.
     \end{enumerate}
\\[]\hrulefill
  \item \label{step:end_global_NR} Perform an $L^2$ projection to obtain $f^{(H)}$ at nodes. This establishes the solution.
  \end{enumerate}
\end{algorithmic}
\end{algorithm}}
\caption{Algorithm to solve \eqref{eq:primal_in_dual}. $N$ represents the number of elements used to discretize the extended domain. Step \ref{step:abort} indicates that condition $\mathscr{C}_d$ was satisfied but the Newton residual did not decrease. The algorithm currently is applicable to a wide range of base states, even those far from the solutions. It fails for base states close to being outside the function class allowed for the problem (for e.g. discontinuous base states).  }
\label{algo}
\end{table}

\subsection{Numerical examples for the DDE formulation}\label{sec:dde_numerics}
In each of the following examples, we discretize the extended domain where $x \in (-L-2, L+2)$, choose a base state for this extended domain, and allow the dual scheme to pick up a solution within the domain of interest, $x \in (-L, L)$.
Following \eqref{eq:residual_full}, we employ $\lambda_s(x)=0$ (without loss of generality). The figures produced in the following sections are based on a standard $L^2$ projection performed from Gauss points to the nodal points (projection only performed in the domain of interest).

For all the following problems,  $\mathsf{a}=10^6$ and $L=8$ unless otherwise stated. The justification for the choice of $\mathsf{a}$ is as follows: while $\mathsf{a} \neq 0$
is a free choice in the theoretical scheme (for the critical point formulation of $S$), it is clear from the convexity condition \eqref{eq:convexity_cond} (cf.~\cite[Sec.~3]{Acharya:2024:HCC} for conclusions on degenerate ellipticity of the dual problem in the PDE case) that in seeking solutions, a large value of $\mathsf{a} > 0$ is practically useful in allowing more freedom to sample dual states (`centered' around the state $\lambda = 0$ in the entire domain) where the problem is concave, and for obtaining solutions. In App.~\ref{app:beta} we demonstrate this fact by a computed example.

For each of the examples presented below (except the first one since it is a trivial example), we compute the maximum of absolute difference (normalized with respect to the RMS value of the field) across the domain when the number of elements in the mesh is doubled. Based on a uniform mesh of $6400$ elements, the RMS value is defined as:
\begin{equation*}
   f_{RMS} = \sqrt{\sum_A\frac{(f^A)^2}{n}}, 
\end{equation*}
where $A$ ranges over the total number of nodes in the domain of interest ($n$). Accordingly,
we define the difference $\mathfrak{D}(m)$ (percentage measure), where $m$ is the number of elements in the mesh under consideration, as follows:
\begin{equation}
\label{eq:diff_converge}
    \,\mathfrak{D}(m) = \underset{A}{\mbox{max}}\frac{\left|f_{2 m}^A - f_{m}^A\right|}{f_{RMS}} \times 100,
\end{equation}
 where $f_m^A$ and $f_{2m}^A$ represent the primal fields obtained using meshes with $m$ and $2m$ elements, respectively, at node $A$. $\mathfrak{D}(m)$ values for each of the examples subsequently presented can be found in App.~\ref{app:tables}: Table \ref{table:convergence}. The obtained primal field in each of the examples is further subjected to a Finite Difference approximation of \eqref{eq:primal}. Details can be found in App.~\ref{app}.

 In the following, we will collectively refer to the last two operations as convergence of results w.r.t mesh refinement in the discussion of computed results.

\subsubsection{Approximating solutions with prior knowledge} \label{sec:prior_knowledge}
In this section, we examine cases where the base states are chosen close to the solutions of \eqref{eq:primal}. These base states are specifically designed so that, when used in Alg.~\ref{algo}, a simple N-R method can be employed. Accordingly, the Convexity Condition is met at each N-R iterate without employing the step-size control and the initial base state remains unchanged throughout the execution of the algorithm.

A fixed point iteration scheme due to Petviashvili (cf.~\cite[Sec.~4.3]{IP24}, \cite{petviashvili1976equation}) is utilized to generate solutions to \eqref{eq:primal}, which we will refer to as the $\mathcal{PV}$ solutions. For $C_1 = 0$ in \eqref{eq:int_cm1}, we define:
\begin{equation}
    g(x) := -f(x): \quad g(x) = \frac{1}{2}\int_{x-1}^{x+1}\,g(s)^2 ds = \frac12(\Lambda\, *\, g^2)(x),
\end{equation}
where 
\[
\Lambda(x) =
\begin{cases} 
1 & \text{if } |x| < 1, \\
0 & \text{otherwise} ,
\end{cases}
\]
and employ the following iterative scheme $(n\geq1)$:
\begin{equation}
\label{eq:PVscheme}
\begin{aligned}
    \text{Step 1:} & \quad \tilde{g}_{n+1}(x) = \frac12(\Lambda * g_n^2)(x); \\
    \text{Step 2:} & \quad { \displaystyle \tilde{C}_{n+1} = \frac{\int_{-\infty}^{\infty} g_n\, dx}{\int_{-\infty}^{\infty} \tilde{g}_{n+1}\,dx};} \\
    \text{Step 3:} & \quad g_{n+1}(x) = \tilde{C}_{n+1}^q \, \tilde{g}_{n+1}(x),
\end{aligned}
\end{equation}
where $q=1.4$ has been used, and the integrals in second step are for $x\in(-\infty,\infty)$.
We set 
\begin{equation*}
    g_{1}(x) = \frac{1}{ \sqrt{2\pi}} e^{-\frac{x^2}{2}}
\end{equation*} and set the following tolerance for convergence:
\begin{equation*}
    |\tilde{C}_n - 1| < 10^{-6}.
\end{equation*}
The integrations in Step 2 of \eqref{eq:PVscheme} are computed on a finite domain. Due to the rapid decay of the $\mathcal{PV}$ solution away from $x = 0$, the domains considered are large enough to minimize the impact of the finite domain on the integrals, and it has been verified that the final $\mathcal{PV}$ profiles obtained on doubling the domain are very close.

    One of the $\mathcal{PV}$ solutions is shown in  Fig.~\ref{fig:P1}. Employing this solution as a base state for the dual scheme, Fig.~\ref{fig:P1} also demonstrates that the obtained primal field $f^{(H)}$ remains close to it.

\begin{figure}
    \centering
    \begin{minipage}[b]{0.45\textwidth}
        \centering        \includegraphics[width=\linewidth]{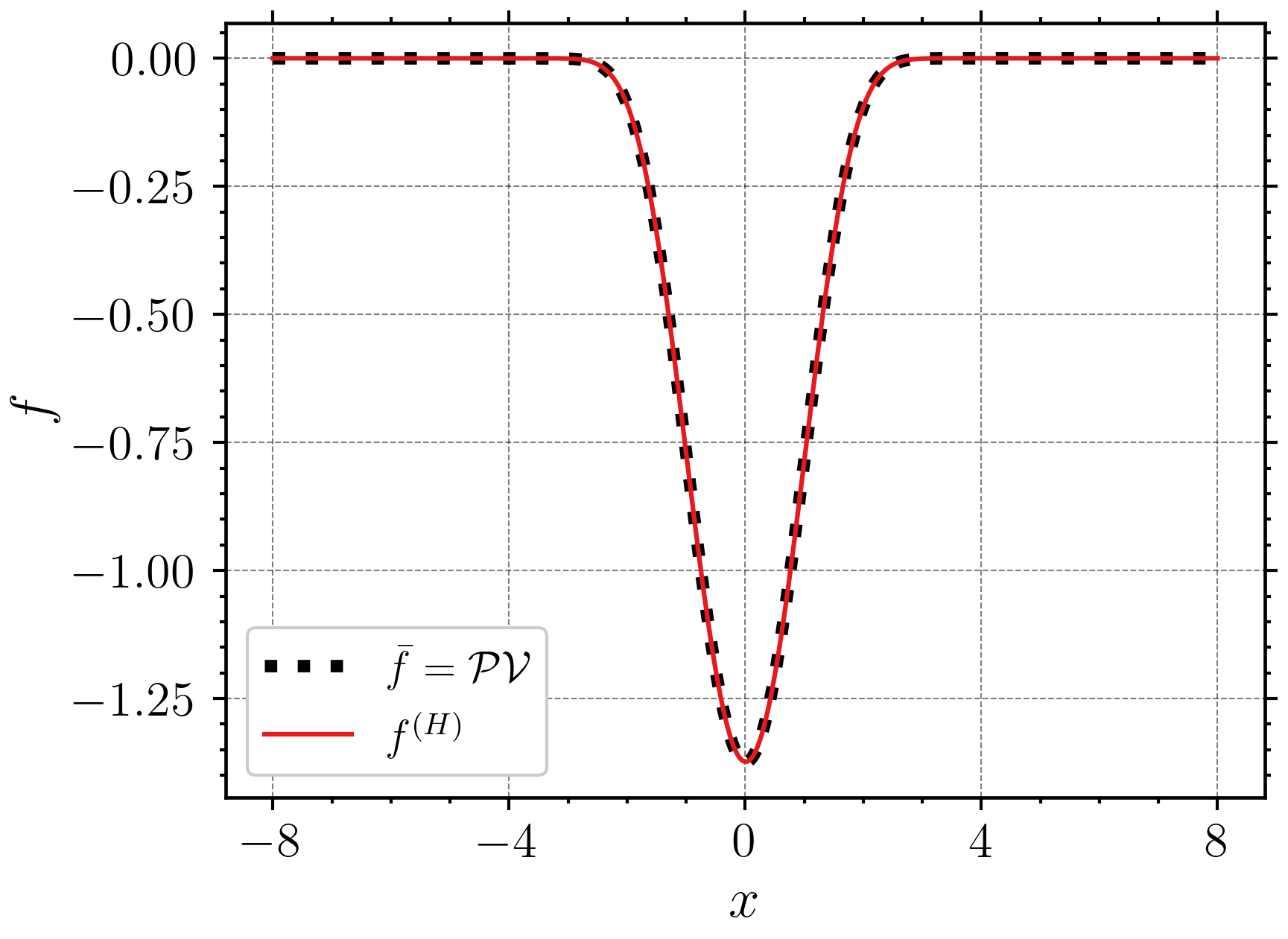}  
\subcaption{Base state: $\mathcal{PV}$ solution}
\label{fig:P1}
    \end{minipage}
    \hspace{0.05\textwidth}
    \begin{minipage}[b]{0.45\textwidth}
        \centering        \includegraphics[width=\textwidth]{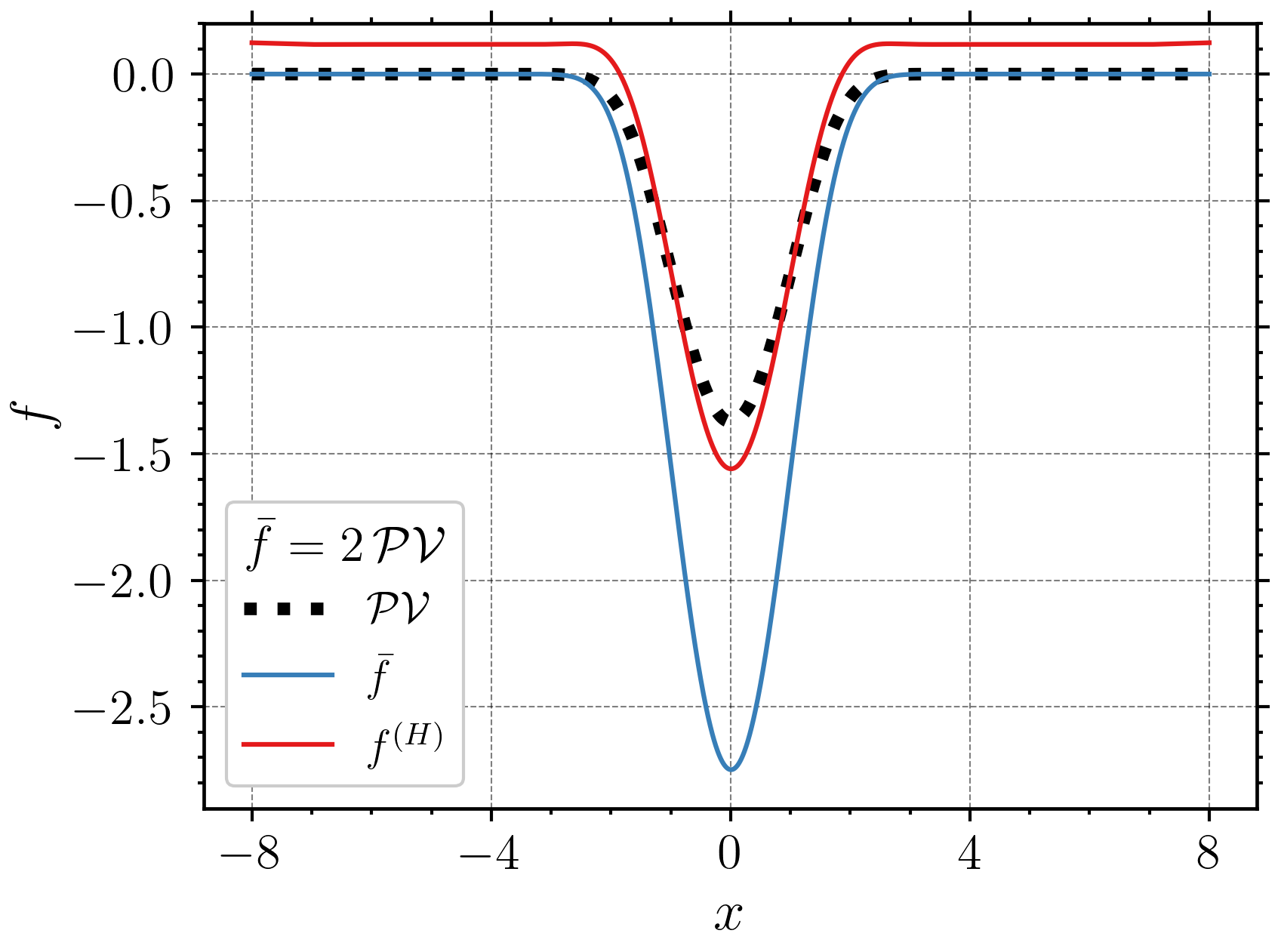} 
        \subcaption{Base state: $2\,\mathcal{PV}$ solution}
        \label{fig:sub-P2_1}
    \end{minipage}
    
    \caption{Figure~(a) employs the $\mathcal{PV}$ solution directly as the base state, whereas Figure~(b) utilizes the $\mathcal{PV}$ solution scaled by a factor of 2 before using it as the base state.  }
    \label{fig:P2}
\end{figure}

We now examine how much we can deviate from the $\mathcal{PV}$ solution used as a base state. We consider the following base state:
\begin{equation}
\label{eq:2PV_base_state}
    \bar{f}(x) = \tilde{\alpha} \, \mathcal{PV}(x),
\end{equation} where $\tilde{\alpha}$ denotes a scaling factor. For $\tilde{\alpha}=2$, the results are shown in Fig.~\ref{fig:sub-P2_1}.  Comparing against Fig.~\ref{fig:P1}, it is evident that given an input condition
that differs from the $\mathcal{PV}$ solution, the dual scheme can pick up solutions different from the $\mathcal{PV}$ solution. Convergence of the obtained primal profile w.r.t mesh refinement can be found in App.~\ref{app:figures}: Fig.~\ref{fig:sub-P2_2}.

Results obtained for the base states set up with different scaling factors are shown in Fig.~\ref{fig:sub-P3_1}, which indicate that for $\tilde{\alpha}=0.2$, we obtain a constant-in-space type primal field. Additionally for an approximate range of $0.2<\tilde{\alpha}<0.8$ and $\tilde{\alpha}>2.3$, the dual scheme fails to converge with a simple N-R. Convergence results for the obtained primal profile w.r.t mesh refinement for these examples can be found in App.~\ref{app:figures}: Fig.~\ref{fig:CA}.

For certain examples, the obtained primal fields exhibit slight bending near the domain boundaries at $x=\pm8$. This behavior becomes apparent when scaling the y-axis to smaller scales, as clearly illustrated in App.~\ref{app:figures}: Fig.~\ref{fig:sub-CA_1} (primal profiles obtained for $\tilde{\alpha}=0.2$ on different meshes: Range of plot is $\mathcal{O}(10^{-3})$). The primal field $f^{(H)}$ satisfies the governing equation in $(-L,L)$ and the solutions obtained satisfy the prescribed tolerance. Since $L$ is an adjustable parameter, solutions on arbitrarily large domains without such bends, if deemed undesirable, can be obtained, up to computational cost.

The fact that the primal problem can be solved without any boundary condition specified on the primal field may be considered an interesting aspect of the dual scheme.

Finally, Fig.~\ref{fig:sub-P3_2} shows the result obtained when a $\mathcal{PV}$ solution with two self-similar structures (humps) on the same domain and scaled by a factor $\tilde{\alpha}=2$ is used as a base state. Corresponding profiles w.r.t mesh refinement are shown in App.~\ref{app:figures}: Fig.~\ref{fig:sub-CA_5}.

\begin{figure}
    \centering
    \begin{minipage}[b]{0.45\textwidth}
        \centering
        \includegraphics[width=\textwidth]{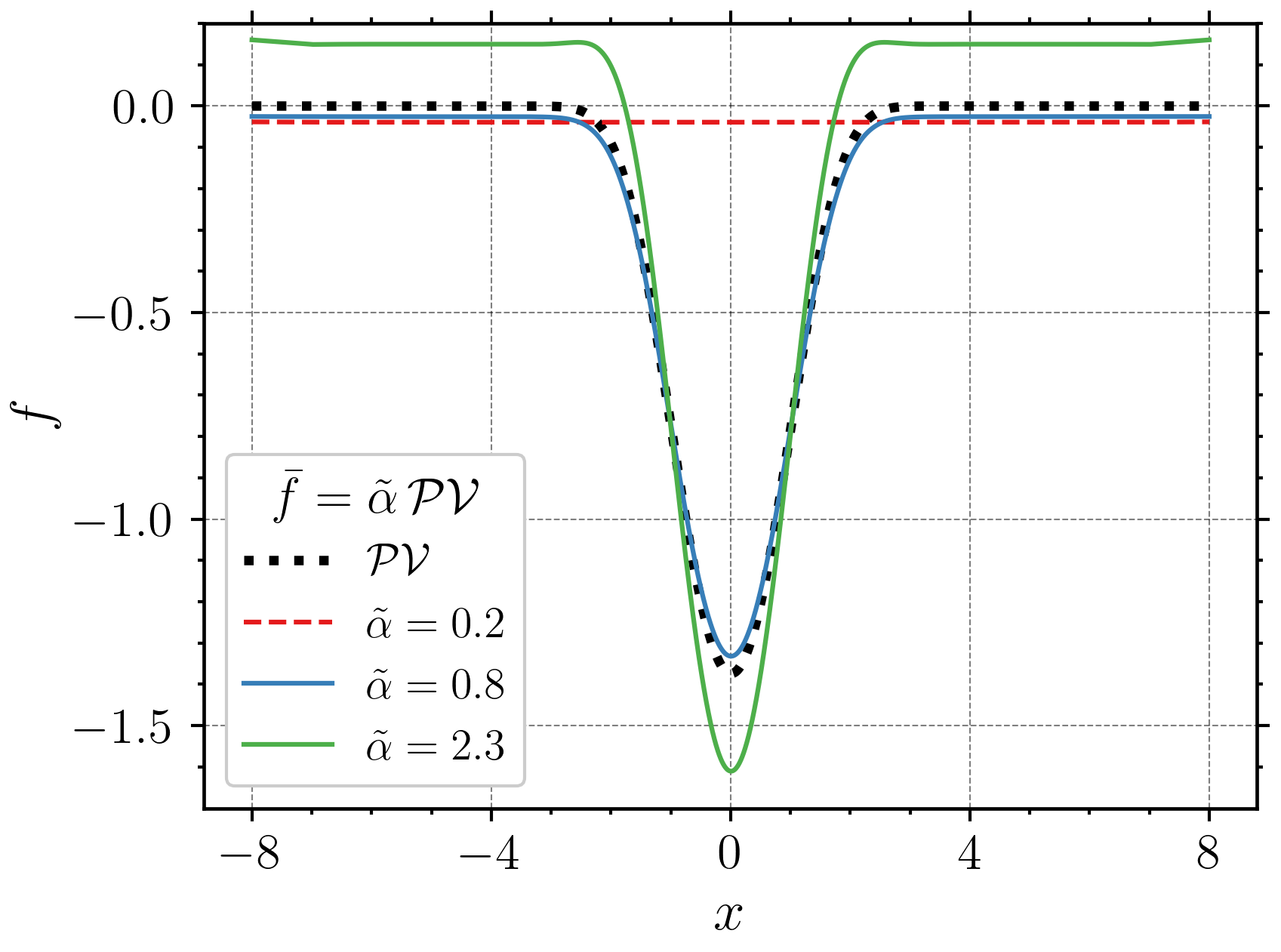} 
        \subcaption{Single hump primal profiles}
        \label{fig:sub-P3_1}
    \end{minipage}
    \hspace{0.05\textwidth}
    \begin{minipage}[b]{0.45\textwidth}
        \centering
        \includegraphics[width=\textwidth]{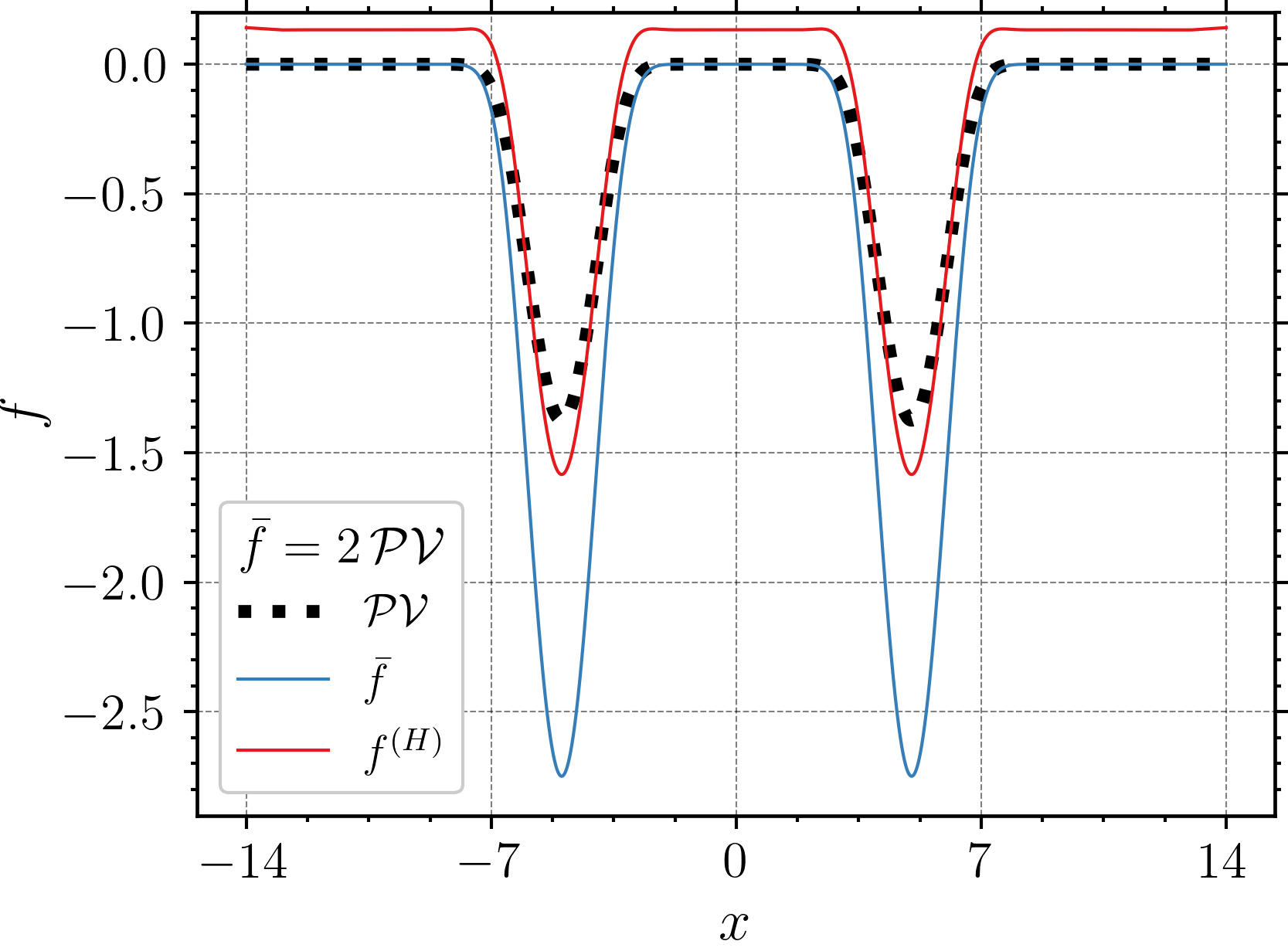} 
        \subcaption{Double hump profile}
        \label{fig:sub-P3_2}
    \end{minipage}

    \caption{Profiles in Fig.~(a) are produced with a single hump $\mathcal{PV}$ solution scaled by $\tilde{\alpha}$ set as the base state. Fig.~(b) uses a double hump $\mathcal{PV}$ solution scaled by a factor of 2 as the base state.} 
    \label{fig:P3}
\end{figure}

\subsubsection*{Scaling invariance of the primal solution with $\ma$ for $H = \half \ma (f - \bar{f})^2$}
The scaling invariance indicated and explained in Sec.~\ref{sec:DDE_formulation}, preamble of Sec.~\ref{sec:dde_numerics} and App.~\ref{app:beta} has been demonstrated in Fig.~\ref{fig:invar}. 
Numerically, a larger value of $\ma$ allows us to search for the dual solution in an enlarged space. This has been discussed in App.~\ref{app:beta}.

\begin{figure}
    \centering
    \begin{minipage}[b]{0.45\textwidth}
        \centering
        \includegraphics[width=\textwidth]{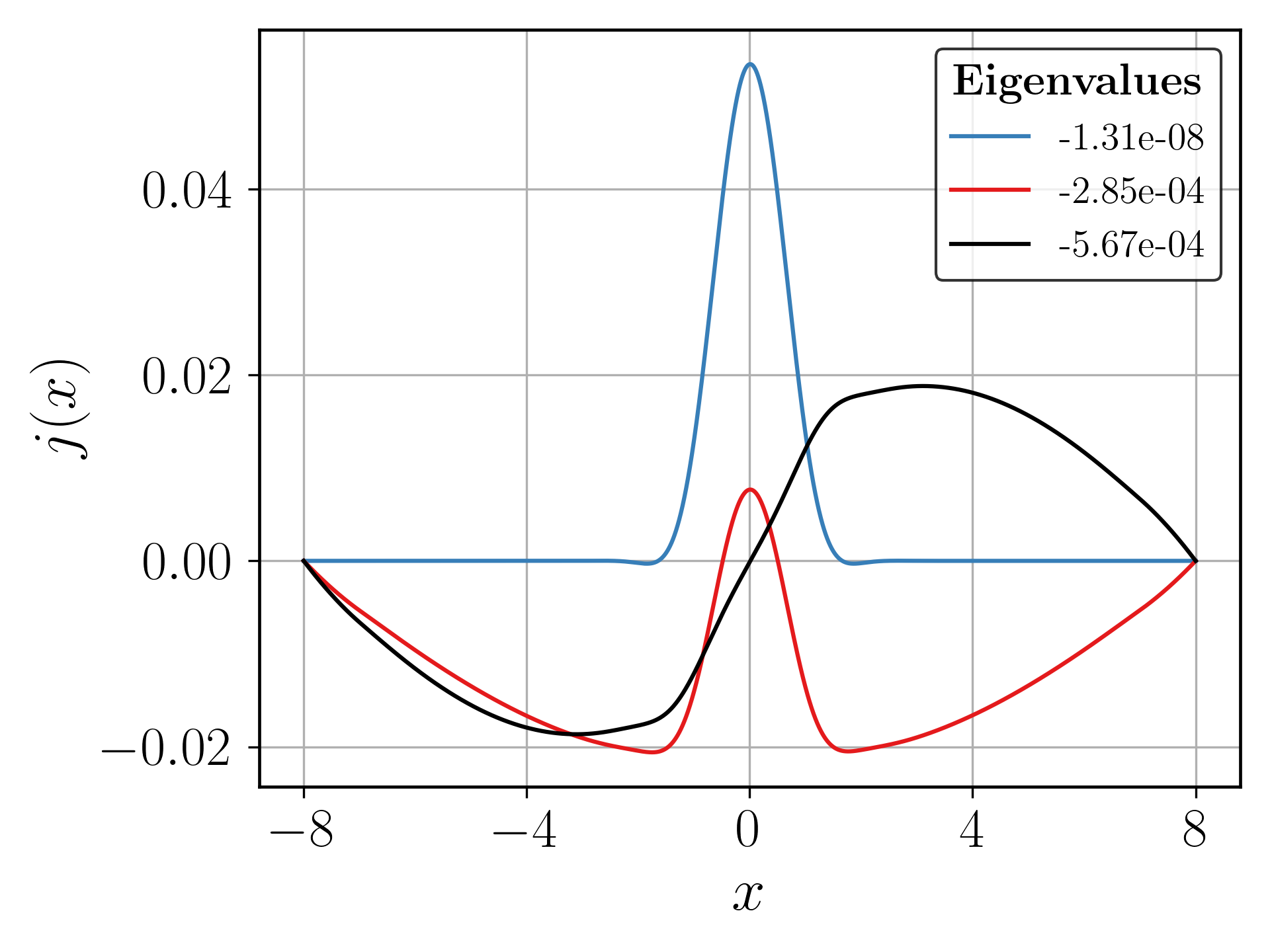} 
        \subcaption{$\ma = 1$}
        \label{fig:invar_1}
    \end{minipage}
    \hspace{0.05\textwidth}
    \begin{minipage}[b]{0.45\textwidth}
        \centering
        \includegraphics[width=\textwidth]{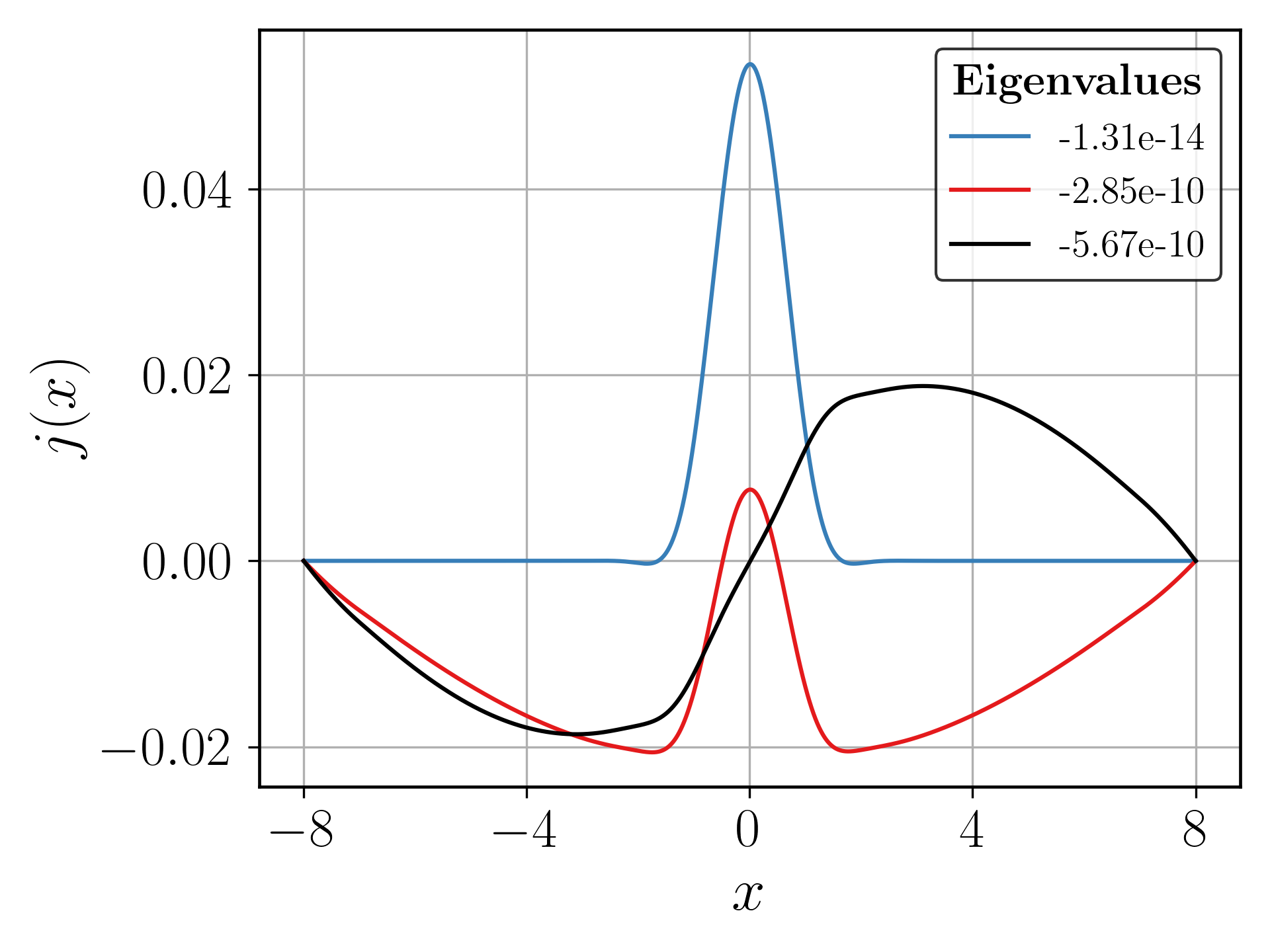} 
        \subcaption{$\ma = 10^6$}
        \label{fig:invar_2}
    \end{minipage} 

     
    \caption{Invariance in the eigenfunctions upon scaling the factor $\ma$: The DtP mapped image of the eigenvectors $j$ associated with the discrete Jacobian matrix \eqref{discrete_jaco} are plotted against nodal positions $x$ for the converged dual field $\lambda$. Three smallest eigenvalues have been shown. Following parameters were set: $\bar{f} = 2 \,\mathcal{PV}$, Mesh = 6400.} 
    \label{fig:invar}
\end{figure}

\subsubsection{Approximating solutions without any prior knowledge}\label{sec:without_prior}

{\bf \underline{Gaussian base states}:}

\noindent We start by considering the following standard Gaussian:
\begin{equation}
    \mathcal{G}(x) = \frac{1}{ \sqrt{2\pi}} e^{-\frac{x^2}{2}},
    \label{eq:gaussian}
\end{equation}
and the base states of the following type:
\begin{equation*}
    \bar{f}(x) = \gamma\, \mathcal{G}(x)
\end{equation*}
For a range of $\gamma$ values, we generate primal fields from dual solutions obtained from a simple N-R scheme. The simple N-R converges only for \(-5.2 < \gamma < -2.7\) and for \(-0.5<\gamma < 5\) (amongst the possibilities tried). For the latter range of $\gamma$, the dual scheme tends to pick up constant-in-space primal fields. The primal profiles obtained for a few different values of $\gamma$ are shown in Fig.~\ref{fig:P4}. 
The corresponding profiles on mesh refinement are presented in App.~\ref{app:figures}: Fig.~\ref{fig:CG}.

\begin{figure}
    \centering
    \begin{minipage}[b]{0.45\textwidth}
        \centering
        \includegraphics[width=\textwidth]{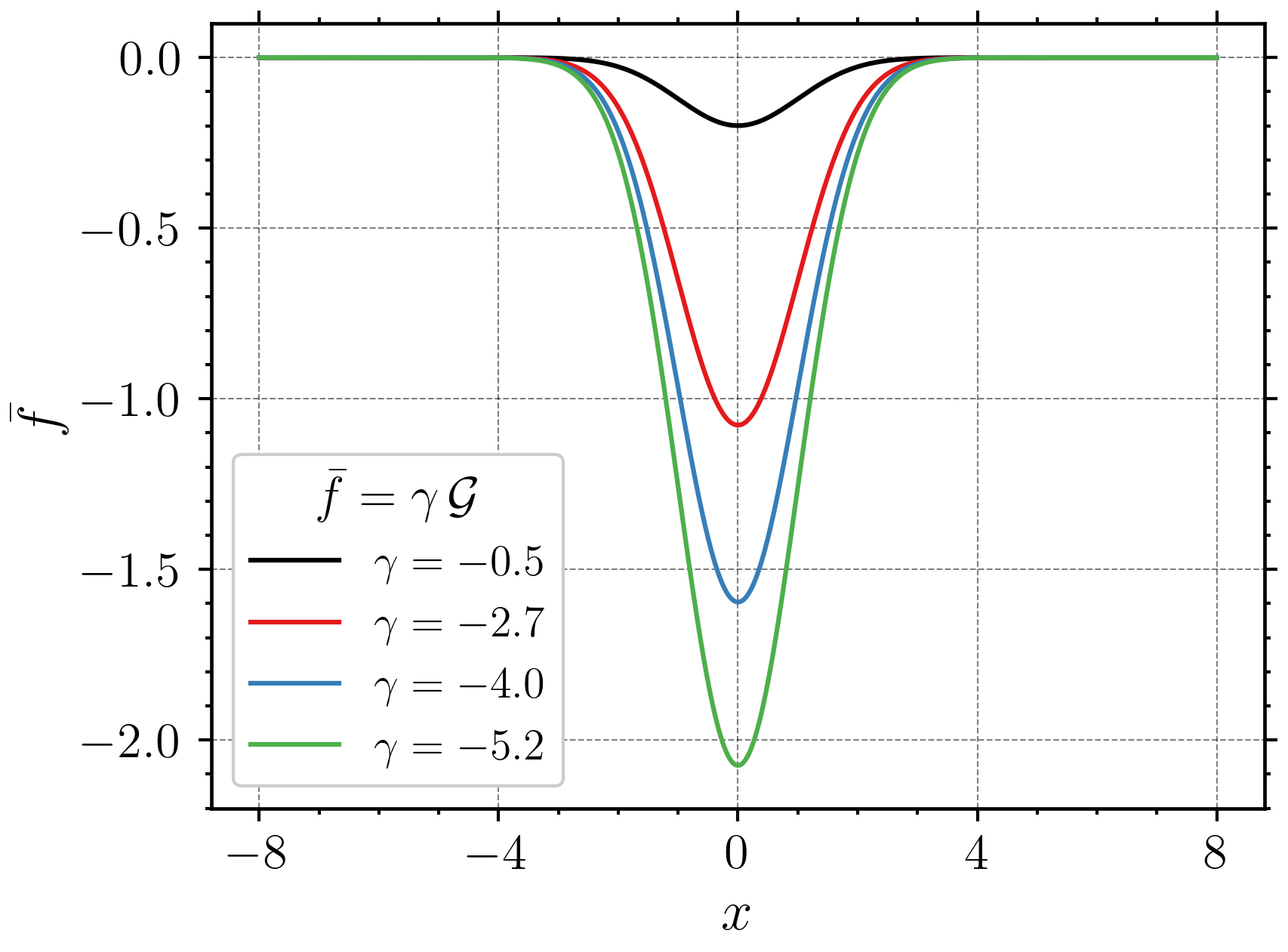} 
        \subcaption{Scaled gaussian base states}
        \label{fig:sub-P4_1}
    \end{minipage}
    \hspace{0.05\textwidth}
    \begin{minipage}[b]{0.45\textwidth}
        \centering
        \includegraphics[width=\textwidth]{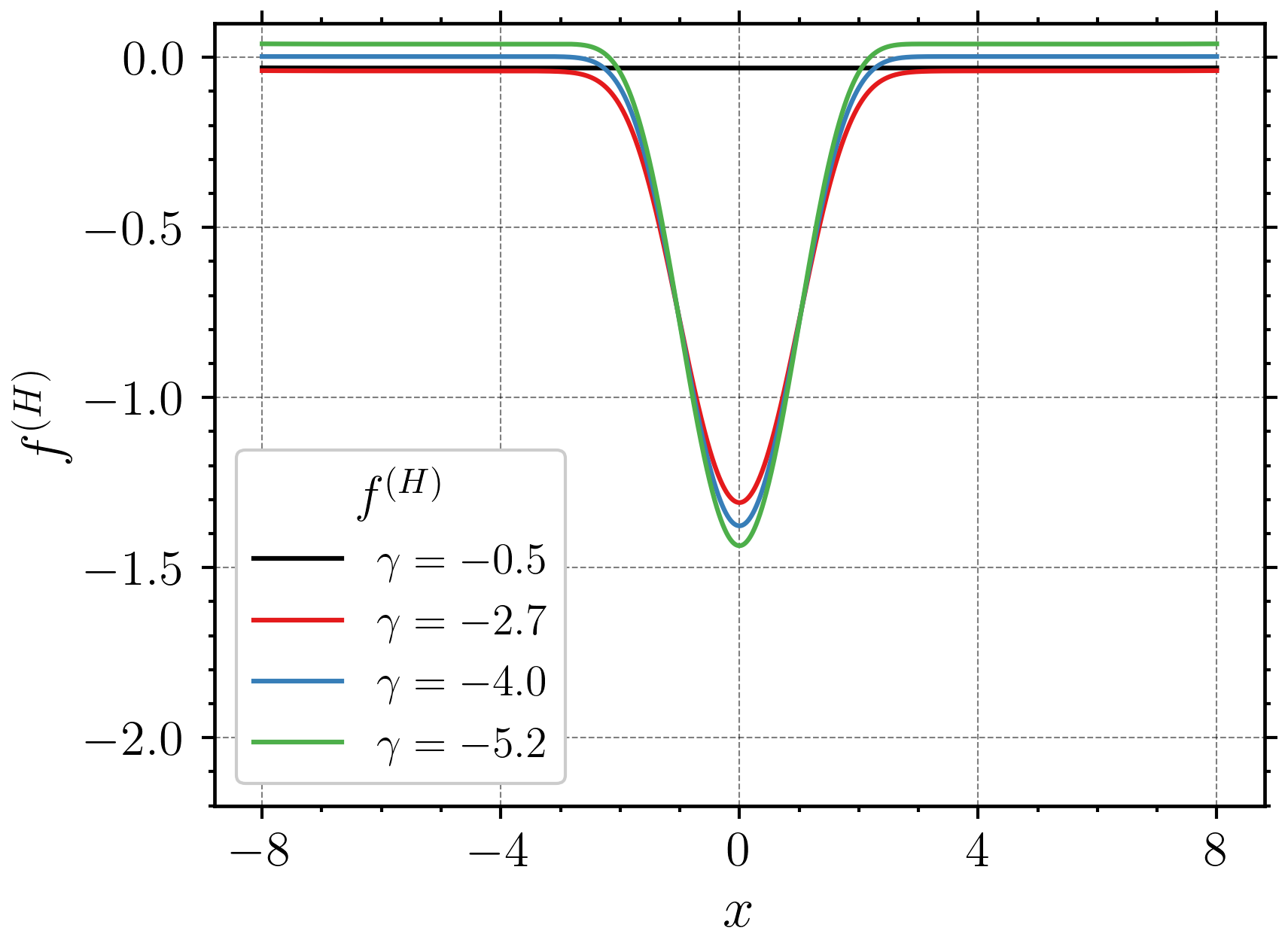} 
        \subcaption{Obtained primal fields}
        \label{fig:sub-P4_2}
    \end{minipage}
    
    \caption{Primal field profiles in Fig.~(b) are obtained using the corresponding scaled Gaussian profiles set as base states from Fig.~(a) (drawn on same scale). Simple N-R scheme was used to produce the results. convergence results upon can be found in App.~\ref{app:figures}: Fig.~\ref{fig:CG}.} 
    \label{fig:P4}
\end{figure}

Employing Alg.~\ref{algo} allows us to pick up solutions to \eqref{eq:primal_in_dual} starting from a wide range of $\gamma$ in \eqref{eq:gaussian}. As an example, Fig.~\ref{fig:GF_P1} shows the results of starting the dual scheme from two nearby Gaussian base states ($\gamma=-1.7$ and $\gamma=-1.9$), each approximating to a different primal field. We note that these base states fail to converge with a simple N-R scheme on a mesh of 6400 elements. The corresponding mesh refinement profiles are presented in App.~\ref{app:figures}: Fig.~\ref{fig:CGF_P1}.

\begin{figure}
    \centering
    \begin{minipage}[b]{0.45\textwidth}
        \centering
        \includegraphics[width=\textwidth]{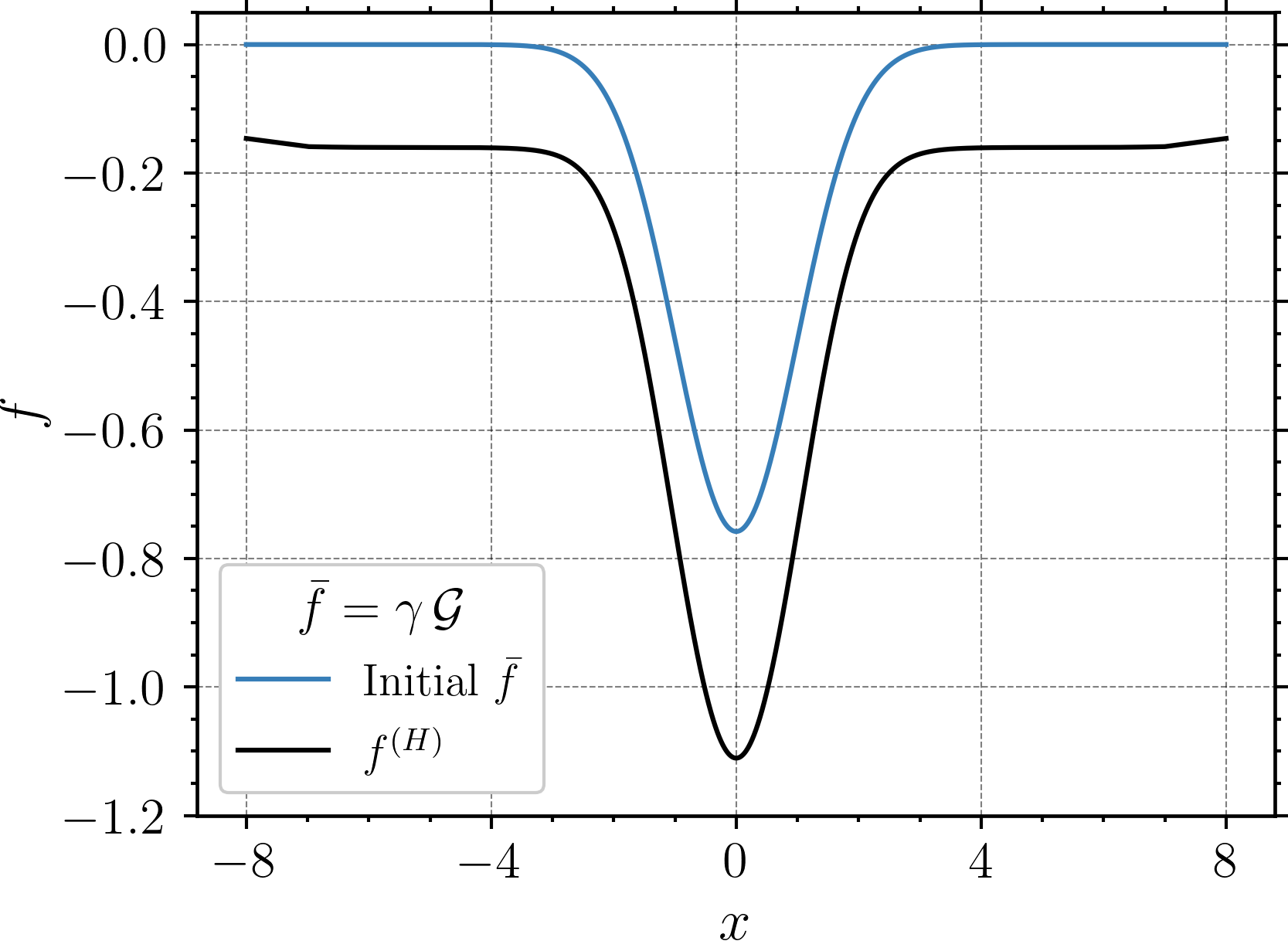} 
        \subcaption{$\gamma=-1.9$}
        \label{fig:sub-GF_P1_1}
    \end{minipage}
    \hspace{0.05\textwidth}
   \begin{minipage}[b]{0.45\textwidth}
        \centering
        \includegraphics[width=\textwidth]{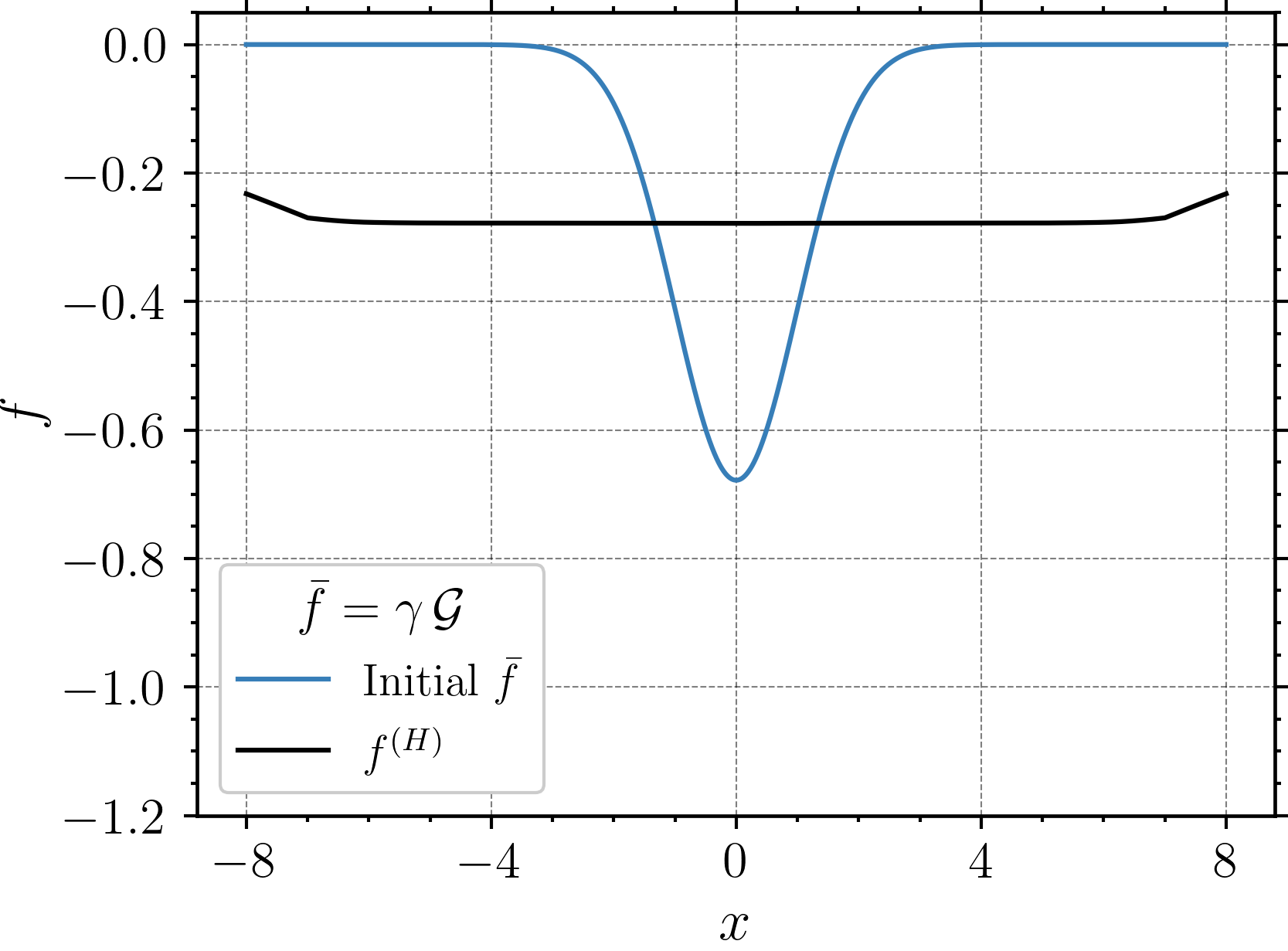} 
        \subcaption{$\gamma=-1.7$}
        \label{fig:sub-GF_P1_2}
    \end{minipage}
    
    \caption{Fig.~(a) and Fig.~(b) (drawn on same scale) are produced using the step-size controlled N-R with scaled Gaussian profiles (by factor $\gamma$)  set as base states. Convergence results w.r.t mesh refinement can be found in App.~\ref{app:figures}: Fig.~\ref{fig:CGF_P1}.} 
    \label{fig:GF_P1}
\end{figure}

\vspace{12pt}

\noindent {\bf \underline{Sinusoidal base states}:}

\noindent Unlike a simple N-R scheme, Alg.~\ref{algo} also allows us to pick up dual solutions using a truncated, smoothed sinusoidal base state of generated from
\begin{equation}
\label{eq:sine}
    \bar{f} = \begin{cases}
        \sin{(\omega x)}  & \,\text{for } -2 \pi <x< 2\pi \\ 0
        & \,\text{otherwise} 
\end{cases},
\end{equation}
where the kinks at the sharp transitions at \( x = \pm 2\pi \) are smoothed out while preserving the overall shape of the profile.
 For $\omega=0.5,1$ and $\omega=2$,  the corresponding primal fields with mesh of $6400$ elements are shown in Fig.~\ref{fig:GF_P2} and the corresponding profiles on mesh refinement can be found in App.~\ref{app:figures}: Fig.~\ref{fig:GFC_P2}. It is evident from these results that the dual scheme prefers to pick up constant primal fields for higher frequency base states.

\begin{figure}[h!]
    \centering
    \begin{minipage}[b]{0.45\textwidth}
        \centering
        \includegraphics[width=\textwidth]{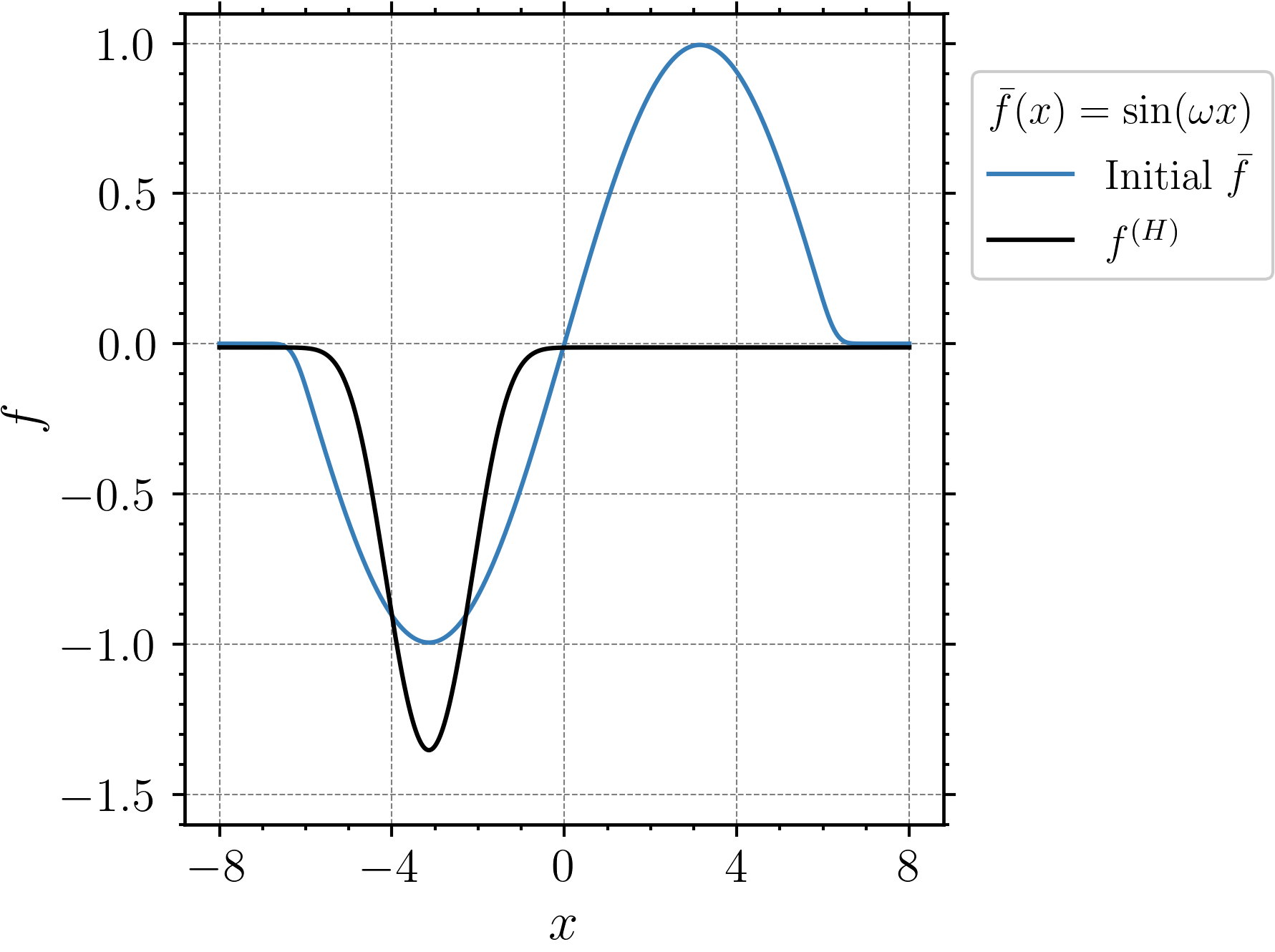} 
        \subcaption{$\omega=0.5$}
        \label{fig:sub-GF_P2_1}
    \end{minipage}
    \hspace{0.05\textwidth}
    \begin{minipage}[b]{0.45\textwidth}
        \centering
        \includegraphics[width=\textwidth]{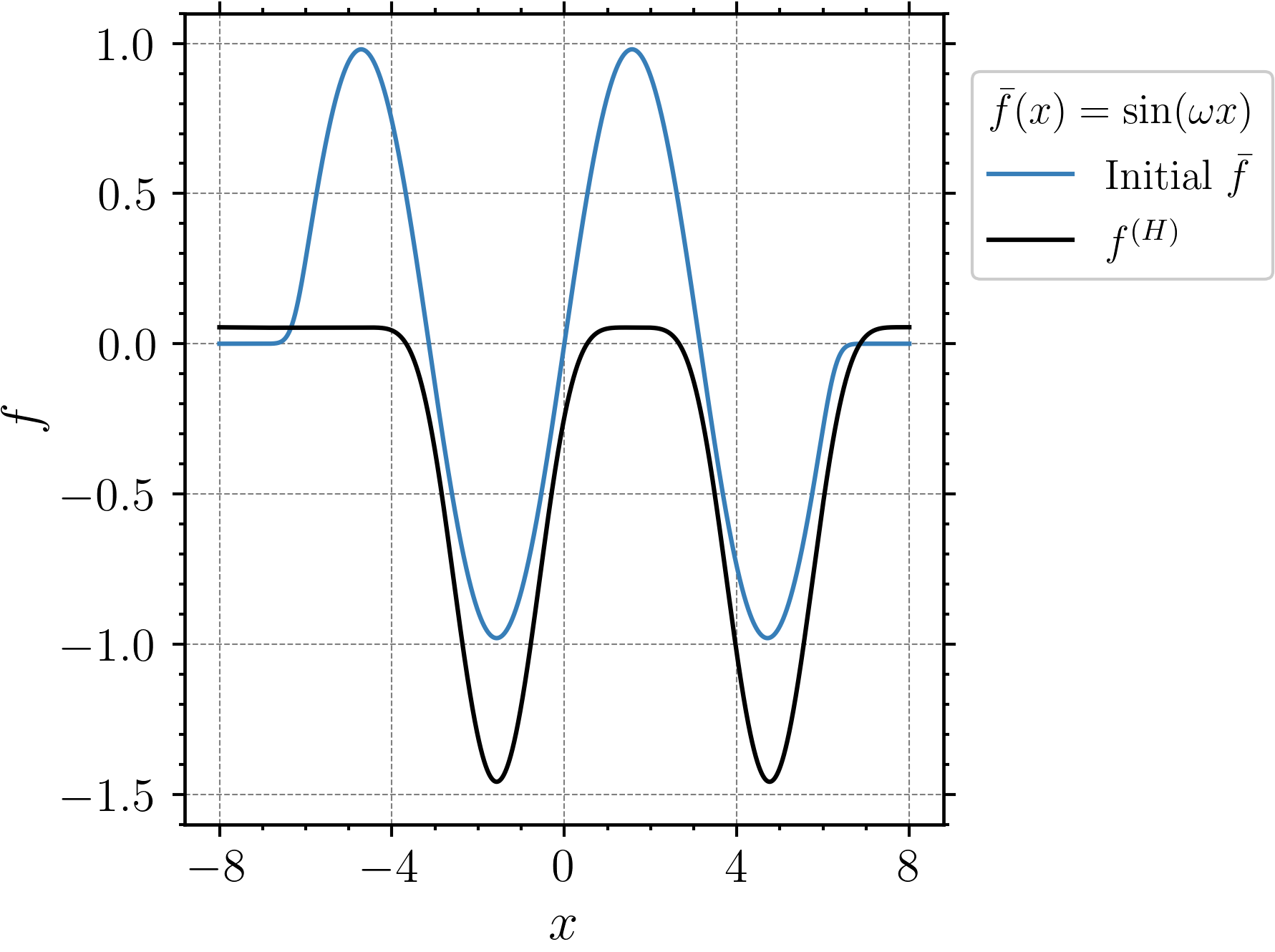} 
        \subcaption{$\omega = 1$}
        \label{fig:sub-GF_P2_3}
    \end{minipage}

    \vspace{0.5cm}

    \begin{minipage}[b]{0.45\textwidth}
        \centering
        \includegraphics[width=\textwidth]{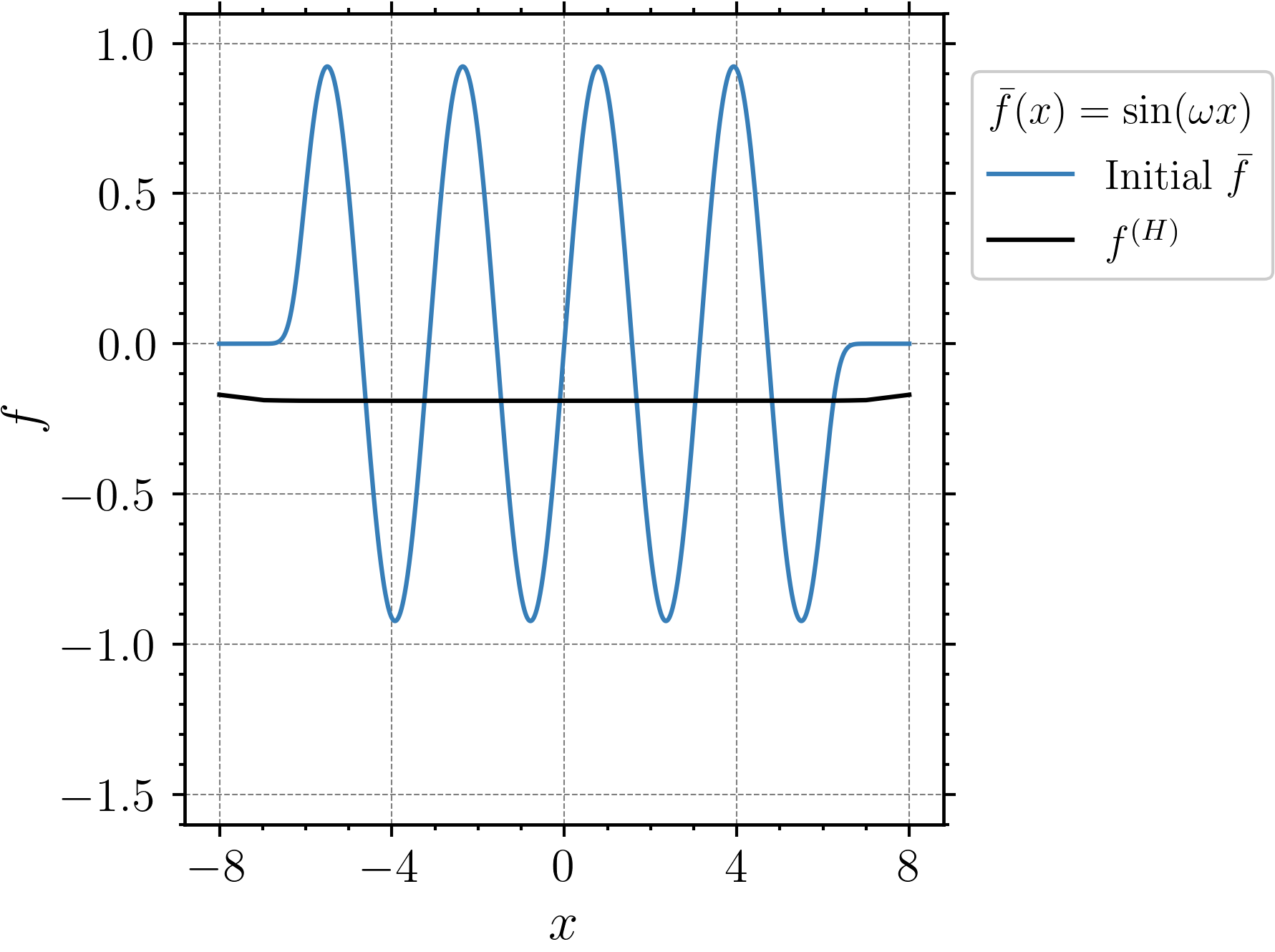} 
        \subcaption{$\omega = 2$}
        \label{fig:sub-GF_P2_5}
    \end{minipage}
    \hspace{0.05\textwidth}
    \begin{minipage}[b]{0.45\textwidth}
        \centering
        \includegraphics[width=\textwidth]{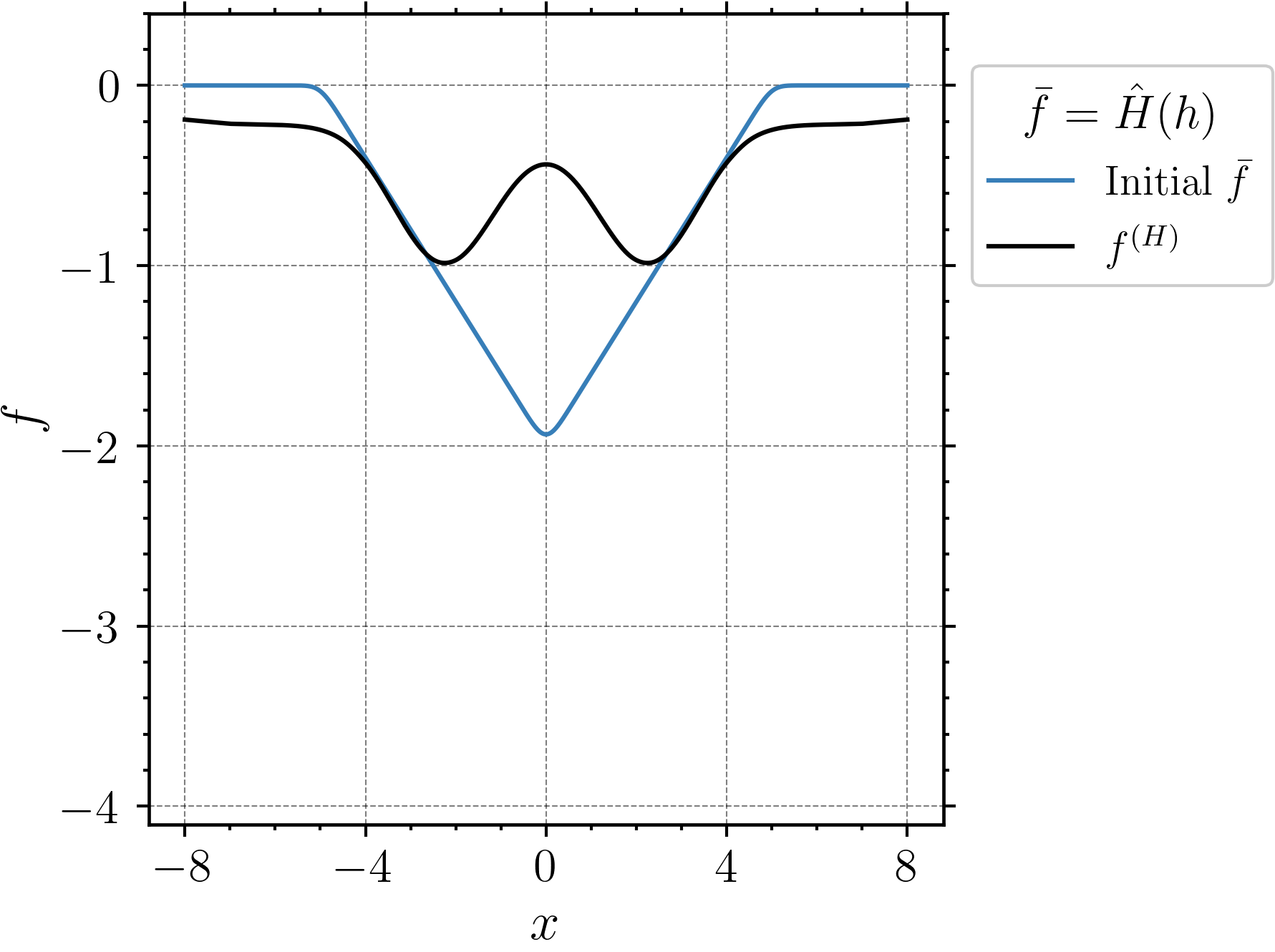} 
        \subcaption{$h=-0.4$}
        \label{fig:sub-GF_P3_1}
    \end{minipage}
    
    \vspace{0.5cm}  

    \begin{minipage}[b]{0.45\textwidth}
        \centering
        \includegraphics[width=\textwidth]{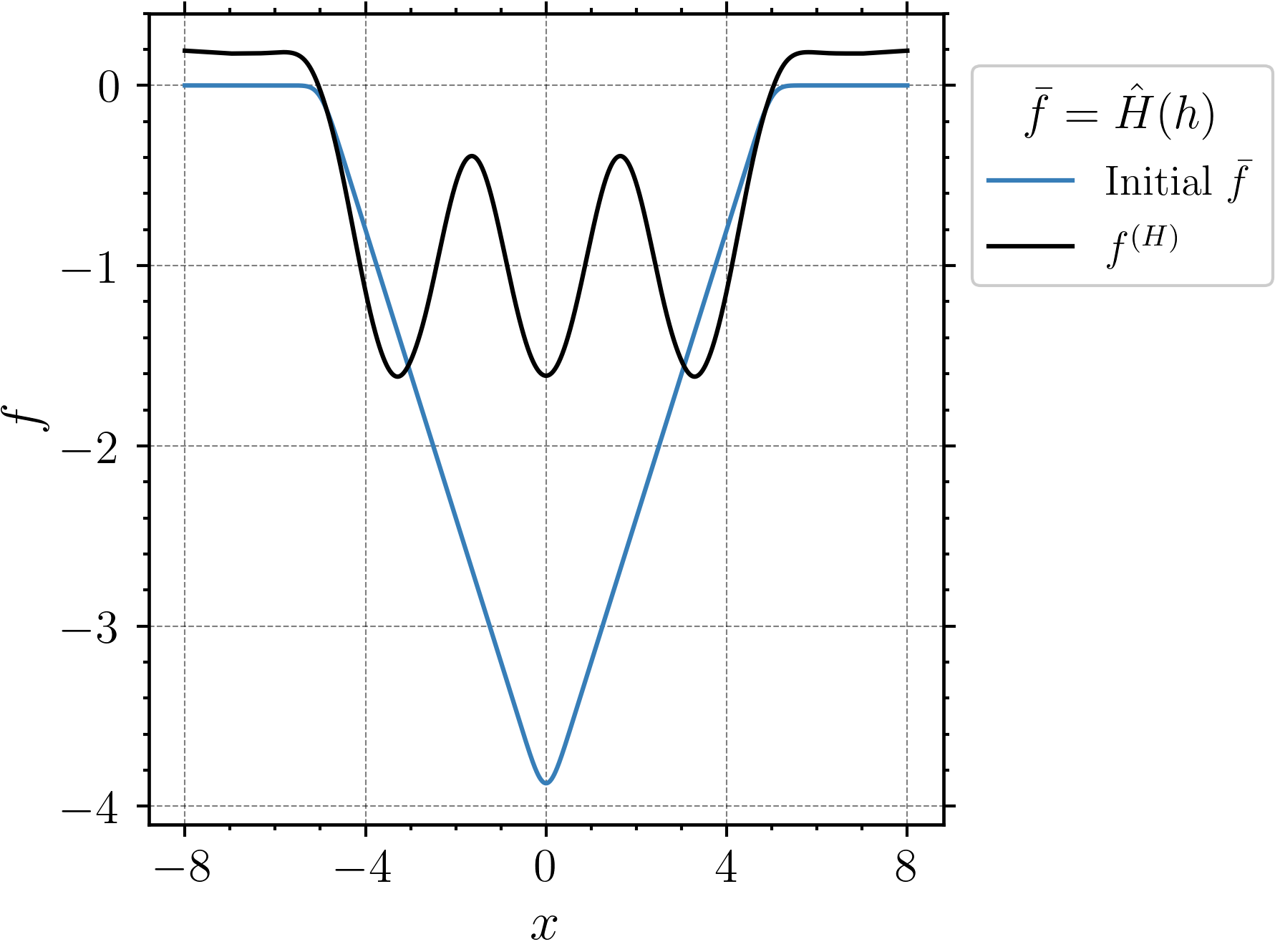} 
        \subcaption{$h=-0.8$}
        \label{fig:sub-GF_P3_3}
    \end{minipage}
    \hspace{0.05\textwidth}
    \begin{minipage}[b]{0.45\textwidth}
        \centering
        \includegraphics[width=\textwidth]{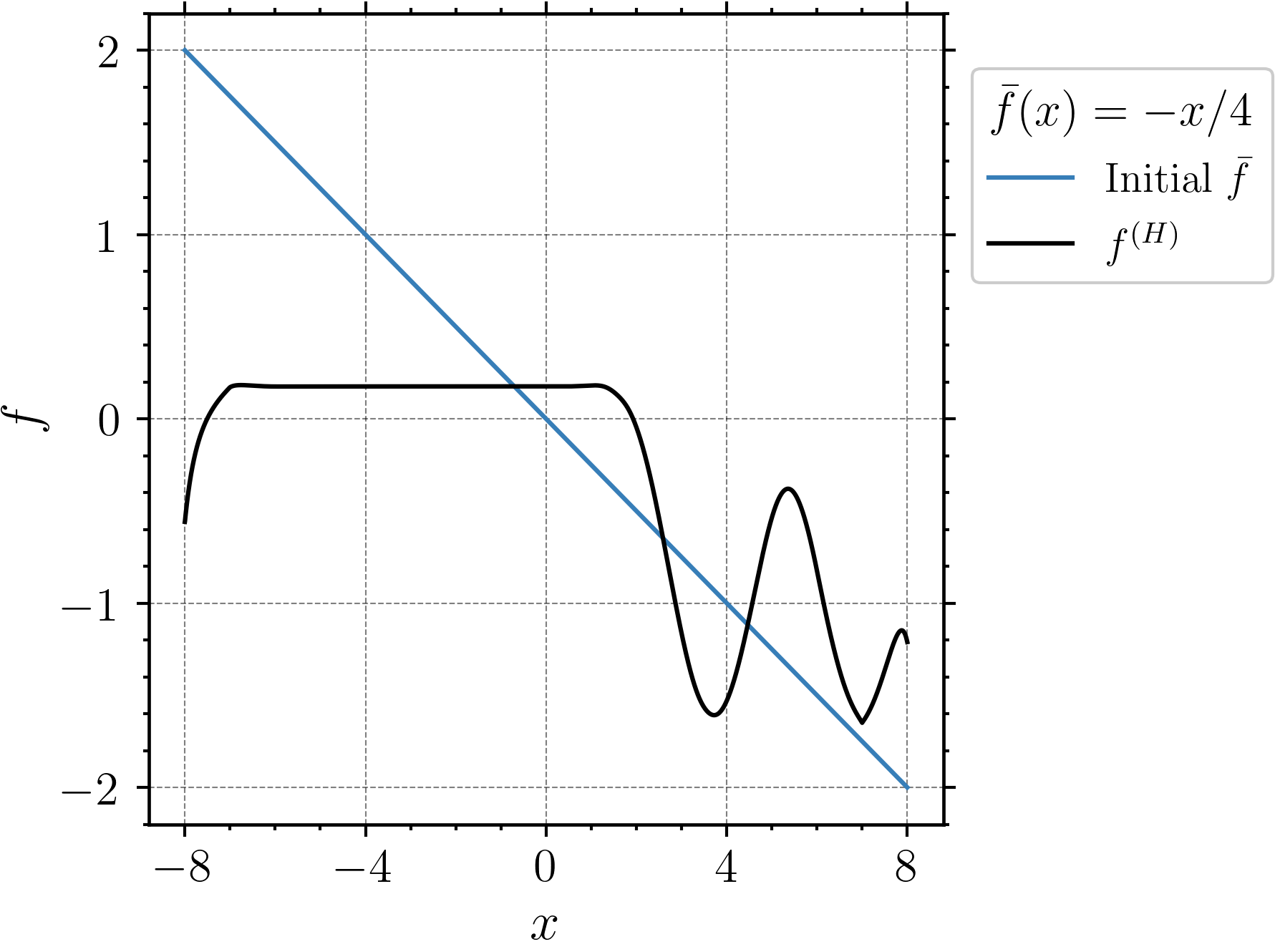} 
        \subcaption{ $\bar{f} = -x/4$}
        \label{fig:sub-GF_P4_1}
    \end{minipage}
    \vspace{0.5cm}  
    
    \caption{Fig.~(a), Fig.~(b) and Fig.~(c) use a sinusoidal base state \eqref{eq:sine} (parameter: $\omega$). Fig.~(d) and Fig.~(e) use a negative-smoothened hat-type initial base state \eqref{eq:hat} (parameter: $h$). Fig.~(f) uses a linear profile \eqref{eq:linear}. Results are produced using a step-size controlled N-R. Convergence results w.r.t mesh refinement for the above examples can be found in App.~\ref{app:figures}: \ref{fig:GFC_P2}.} 
    \label{fig:GF_P2}
\end{figure}

\vspace{12pt}
\noindent {\bf \underline{Piecewise Linear functions as base states}:}

\noindent In this section, we employ  piecewise linear functions as base states for the Alg.~\ref{algo}. We start with a negative smoothed hat function generated from
\begin{equation}
\label{eq:hat}
    \bar{f} = \hat{H}(h) := \begin{cases}
        -(x-5)\,h  & \,\text{for } 0 <x\leq 5 \\
        (x+5)\,h  & \,\text{for } -5  <x\leq 0 \\0
        & \,\text{otherwise} 
\end{cases}; \quad h<0
\end{equation}
with smoothed out the kinks between its piecewise linear segments.
Based on the results obtained using the dual scheme for two different heights $h$, as shown in Fig.~\ref{fig:sub-GF_P3_1} and Fig.~\ref{fig:sub-GF_P3_3}, it is evident that the negative peak disperses into several small humps. Mesh refinement profiles are shown in App.~\ref{app:figures}: Fig.~\ref{fig:GFC_P2}

As a final test, we use linear profiles across the entire domain as the base state. Since no external boundary conditions (from the primal problem description) are imposed on the problem, this test aims to evaluate how the method handles the problem, when non-uniform base states are employed near the boundary.

We adopt the following base state
\begin{equation}
\label{eq:linear}
    \bar{f} = -\frac{x}{4}.
\end{equation}
Evident from the result  for this setup (Fig.~\ref{fig:sub-GF_P4_1}, mesh refinement results shown in  App.~\ref{app:figures}: Fig.~\ref{fig:sub-GF_P4_2}), the $f$ profile exhibits a dip near the boundary, for the problem solved with the given boundary conditions on the dual field.  However, the primal field satisfies \eqref{eq:primal} up to the prescribed $tol = 10^{-13}$ on each of the meshes.


\subsubsection{Dispersive solitons and their disintegrations} \label{sec:dispersive_DDE_example}

We consider the scaled 
$\mathcal{PV}$ profile with an added constant $c$, and define it as the base state: 
\begin{equation*}
    \bar{f}(x) = \tilde{\alpha} \, \mathcal{PV}(x) + c.
\end{equation*}
 Fig.~\ref{fig:PV_waves} shows the result obtained for such a setup. As we try to move away from the $\mathcal{PV}$ solution by increasing $\tilde{\alpha}$ and $c$, we start capturing solutions formed at larger $u_\infty$ values which exhibit the emergence of modulating envelopes around the central peak, representing the characteristics of a solitary wave structure. Fig.~\ref{fig:sub-PV_waves_1} shows one such structure obtained using the following parameters: $L=38$, $\tilde{\alpha}=3$ and $c=2$. The solution reaches an approximate value of $u_\infty = 1.87$. Primal profiles obtained w.r.t mesh refinement are presented in Fig.~\ref{PV_waves_refine1}.
 
 Pursuing a higher $u_\infty$ value by adjusting $\tilde{\alpha}$ and $c$ leads to a breakdown, such that the obtained pattern exhibits a dispersive profile throughout the domain without any ostensible compact support (here, compact support refers to the profile after subtracting $u_\infty$) and has a smooth central dip, as shown in Fig.~\ref{fig:sub-PV_waves_3}. This result was obtained using the following parameters: $ L=38$, $\tilde{\alpha}=4$ and $c=3$ and we will refer to this example as a disintegrated soliton (d-soliton).  This result also matches well with the proposition \ref{prop:invertible}, where the operator $I + u_\infty K$ loses its invertibility upon pursuing $u_\infty>2.301$ approximately and we do not obtain a solitary wave structure. Primal profiles obtained on refinement are presented in Fig.~\ref{fig:PV_waves_refine2}.

For quantitative comparisons of convergence w.r.t mesh refinement on par with the other computed examples, results for the dispersive soliton and the disintegrated soliton (d-soliton) are presented on a domain of $L=8$ in Tables \ref{table:err} and \ref{table:convergence}.

\begin{figure}
    \centering
    \begin{minipage}[b]{0.45\textwidth}
        \centering
        \includegraphics[width=\textwidth]{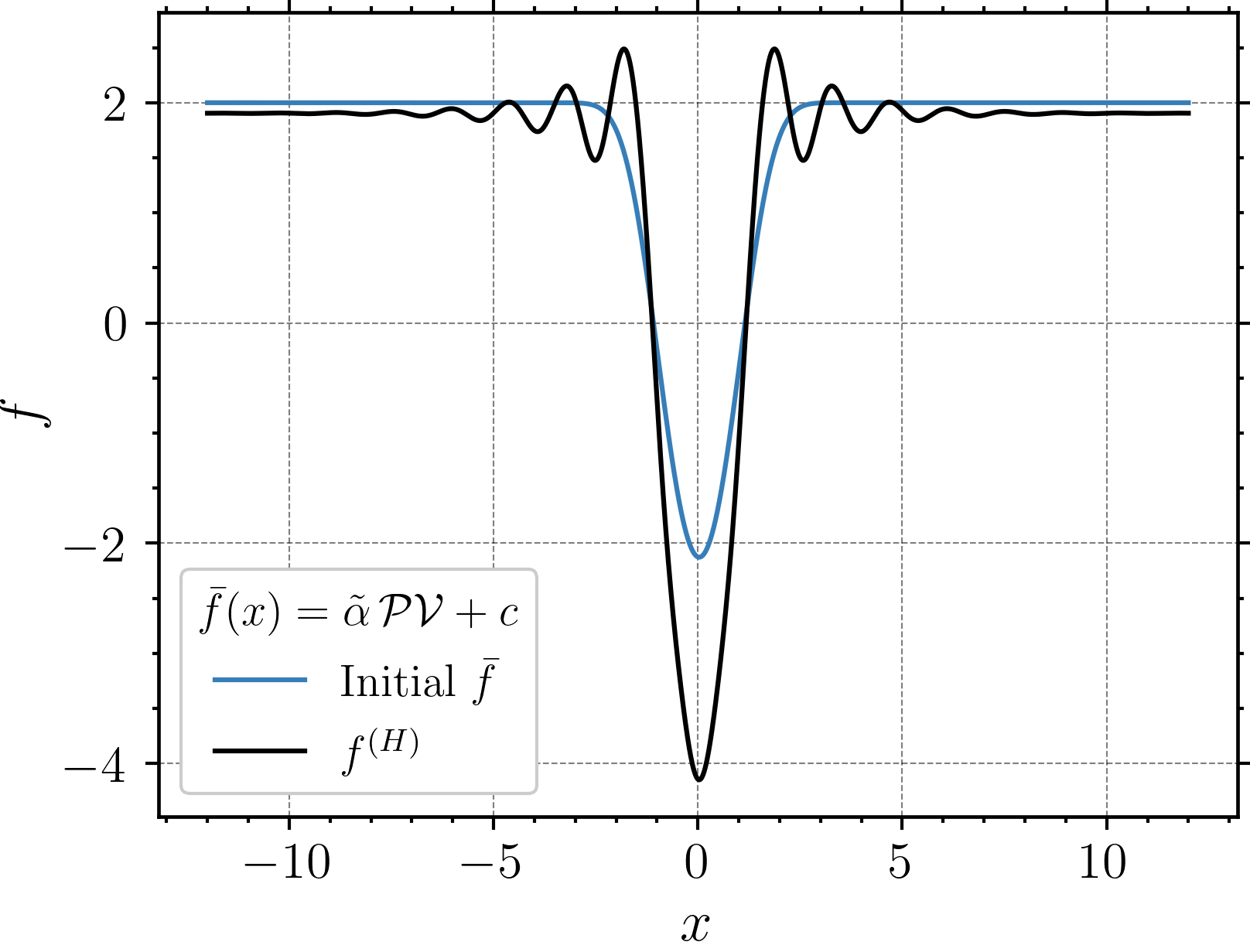} 
        \subcaption{$\tilde{\alpha}=3,c=2$}
        \label{fig:sub-PV_waves_1}
    \end{minipage}
    \hspace{0.05\textwidth}
    \begin{minipage}[b]{0.45\textwidth}
        \centering
        \includegraphics[width=\textwidth]{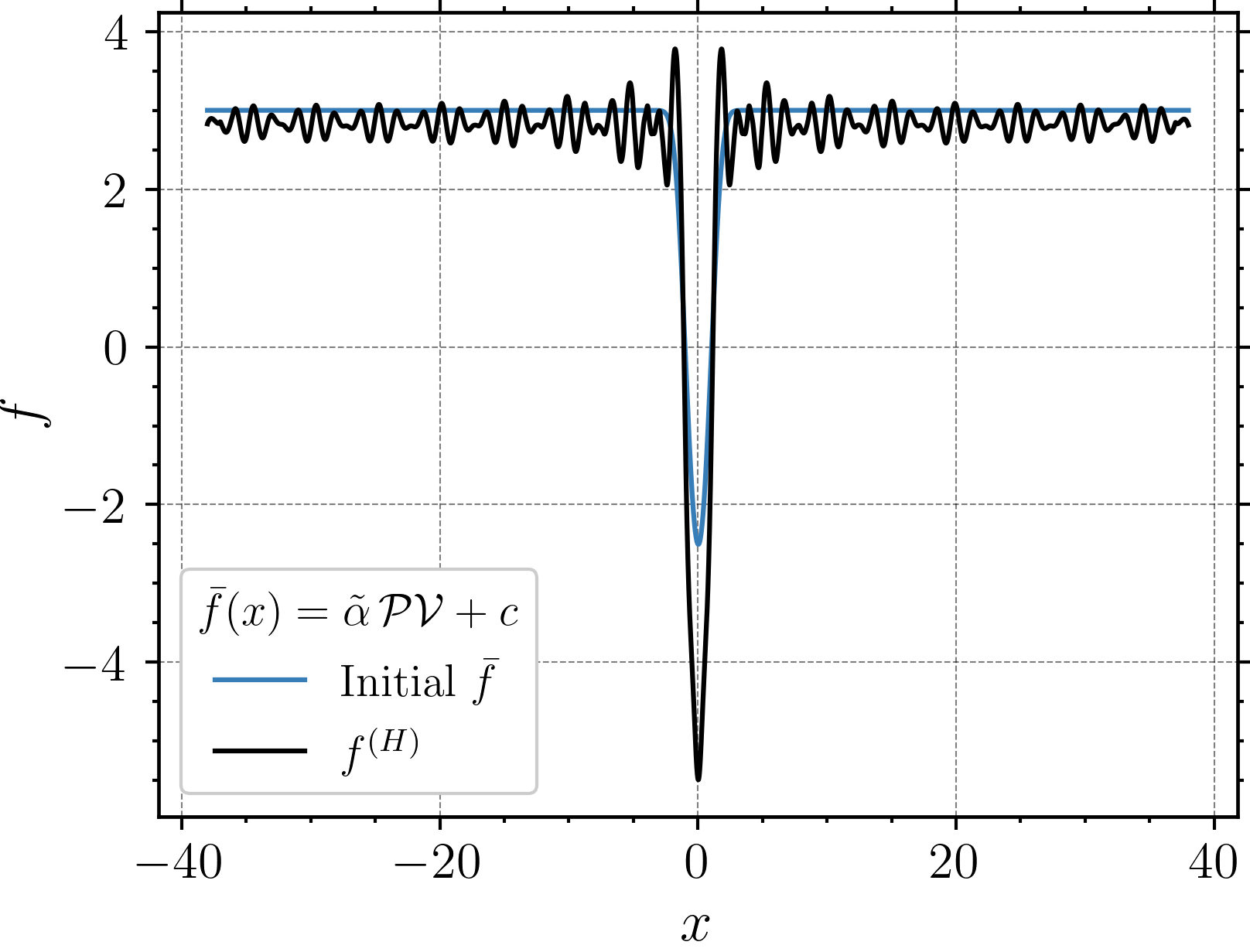} 
        \subcaption{$\tilde{\alpha}=4,c=3$}
        \label{fig:sub-PV_waves_3}
    \end{minipage} 

     
    \caption{Both the plots are produced using scaled $\mathcal{PV}$ profiles added to constants set as base states on $L=38$ and a smaller domain in Fig.~(a) is shown for better demonstration purposes. Fig.~(a) displays a solitary wave structure and as one moves outside the shown region, $f^{(H)}$ approximates to a constant value. Fig.~(b) shows the case where the dual scheme pursues a $u_\infty$ greater than the allowed invertibility limit in proposition \ref{prop:invertible}, leading to a disintegration of the solitary wave.} 
    \label{fig:PV_waves}
\end{figure}

\section{Approximation and numerical examples for the NIE formulation}

\subsection{Approximation for the NIE formulation}\label{sec:approx_NIE}

To approximate solutions of the NIE \eqref{eq:int_wform}, we discretize
the functional $\Snu[\nu]$ as given in \eqref{eq:Snu_finite} 
by left-endpoint-rule quadrature.
This yields a  spectrally accurate approximation for smooth $2L$-periodic functions $\nu$
represented by the vector $\vec{\nu}=(\nu_j)$ of their values 
$\nu_j=\nu(x_j)$ on a uniform grid of $N$ points $x_j = jh$ with grid spacing $h = 2L/N$.
The integrals in the definition of $K$ from \eqref{d:Knu} are approximated at $x_j$ by 
the trapezoid rule, and calculated using the discrete Fourier transform.
Thus the values $K\nu(x_j)$ are approximated by the components $\mK\vec\nu_j$ of $\mathsf{K}\vec\nu$,
where $\mathsf{K}=(K_{jk})$ is an $N\times N$ symmetric banded Toeplitz matrix 
whose nonzero entries are $h$ or $h/2$. 

With this notation, our approximation to $\Snu[\nu]$ is given by 
\begin{equation}
  \label{eq:Snu_h}
\Snu_h(\vec\nu) = -\frac12 \sum_{j=1}^N (\ma +\mK\vec\nu_j) \hat w_j^2\, h ,
\qquad
 \hat w_j  =  \frac{\mathsf{a} \bar w_j-\nu_j-\uinf \mK\vec\nu_j}{\mathsf{a} +\mK\vec\nu_j} 
 \,,
\end{equation}
when the strict convexity condition $\ma+\mK\vec\nu_j>\eps_a$ holds, for some specified tolerance $\eps_a>0$.
When this condition fails, we set $\Snu_h(\vec\nu)$ to be some large negative value. 
The gradient of $\Snu_h$ is explicitly given as
\begin{equation}\label{eq:gSnu_h}
    \frac{\D \Snu_h}{\D \nu_j} = h\rho_j \,, \qquad 
    \rho_j = \hat w_j + \sum_{k=1}^N K_{jk}(\uinf \hat w_k + \tfrac12\hat w_k^2) \,.
\end{equation}
We minimize $-\Snu_h$ using standard optimization software to determine
an approximate minimizer $\vec{\nu}=(\nu_j)$. Provided the strict convexity condition holds,
this determines an approximate solution $(\hat w_j)$ to \eqref{eq:int_wform} 
since the gradient $\D\Snu_h/\D\vec{\nu}$ is small.

The Hessian of $\Snu_h$ has the matrix entries
\begin{equation} \label{eq:Hess_h}
   \frac{\D^2\Snu_h}{\D\nu_j\D\nu_k}(\vec\nu)  
   = -h \sum_{j=1}^N
   \frac{(\delta^i_j+ K_{ij}\hat f_j)(\delta^j_k+\hat f_j K_{jk})}{\ma+\mK\vec\nu_j}\,,
   \qquad
\hat f_j = \uinf+\hat w_j. 
\end{equation}
For later reference, we note that when $\vec\nu=0$ we have $\hat w_j=\bar w_j$,
and the second variation of the functional $\ma \Snu[\nu]$ at $\nu=0$ 
(which is independent of the parameter $\ma$) 
is approximated by the matrix 
\begin{equation}\label{d:mM}
    \mathsf{M}(\bar w) = \frac{\ma}h\left( \frac{\D^2\Snu_h}{\D\nu_j\D\nu_k}(0)  \right)
    = -(I+\mathsf{K}\bar{\mathsf{F}})(I+\bar{\mathsf{F}}\mathsf{K})\,,
\qquad \bar{\mathsf{F}}=\diag(\uinf+\bar w_j).
\end{equation}

In order to carry out numerical computations for a range of values of $\uinf$ while maintaining the strict convexity condition, 
we implement a primitive kind of path-following method. 
We start with $\uinf=0$ and take the base state $\bar w$ 
to be a numerical solution to \eqref{eq:int_wform} computed by Petviashvili iteration, 
as described below and in \cite{Ingimarson_2024}.  
Then we change $\uinf$ in small increments, changing the base state $\bar w$ 
to be the approximate solution $\hat w$ obtained numerically for the previous value of $\uinf$.  
In this way the values of $\nu_j$ can be kept small and the convexity condition 
$\ma + \mK\vec\nu_j>\eps_a$ maintained. 

\subsection{Numerical results for the NIE formulation}
\label{sec:numerics_NIE}

The computations in this section are performed with $\ma=10$.
We minimize $-\Snu_h[\vec\nu]$ using the BFGS algorithm as implemented in the julia package Optim.jl \cite{mogensen2018optim}. 
We take $L=25$ and discretize using $N=1000$ so $h=2L/N=0.05.$
The stopping criterion is based on the maximum norm of the gradient, with tolerance $2\cdot 10^{-9}$.

We present two sets of solutions computed for various values of $\uinf$. We start
with $\uinf=0$ in each case, with base state $\bar w = \bar f$ as computed 
using 50 steps of Petviashvili iteration, similarly as in section \ref{sec:prior_knowledge}.
For succeeding values of $\uinf$ we then reset the base state $\bar w$ to be the 
solution $\hat w$ last computed, as described above in section~\ref{sec:approx_NIE}.
We note that in all cases the residual norm $\|\vec\rho\|_{\infty}<10^{-7}$.

The first set of solutions is computed for nonpositive values of $\uinf$ ranging from 0
down to $-0.475$, a value slightly above the lower threshold $\uinf=-0.5$ 
at which the operator $I+\uinf K$ first loses invertibility on the infinite line 
according to proposition~\ref{prop:invertible}.
Results for these solutions appear in Fig.~\ref{fig:nie_1}.
In Fig.~\ref{fig:nie_1a} we plot the numerically determined wave shapes $\hat f = \uinf + \hat w$ vs. $x$
for 8 successively decreasing values of $\uinf$ ranging from $0$ to $-0.475$,
and in Fig.~\ref{fig:nie_1b} we plot $\log_{10}|\hat w|$ vs. $x$.

We compute a second set of solutions for positive values of $\uinf$ ranging from 0.5 to 2.0, 
a value somewhat less than the upper threshhold $\approx 2.30167$ 
at which the operator $I+\uinf K$ first loses invertibility on the infinite line 
according to proposition~\ref{prop:invertible}.
The results for $\hat f = \uinf+\hat w$ vs. $x$ are shown in Fig.~\ref{fig:nie_2a}, and
we plot $\log_{10}|\hat w|$ vs. $x$ in Fig.~\ref{fig:nie_2b}.

In Table~\ref{table:NIE1}, for selected solutions in both sets we tabulate 
the lowest four non-negligible eigenvalues $\kappa_j$ of the matrix $-\ma D^2\Snu_h[0]$
(with base state set as $\bar w = \hat w$).
This is the negative Hessian scaled by $\ma$, i.e., 
the negative Hessian scaled to be independent of $\ma$.
In all cases  the first eigenvalue satisfied $|\kappa_1|<2\cdot10^{-15}$.
We tabulate as well as the minimum of the (continuous) spectrum of the operator 
$(I+\uinf K)^2$ corresponding to the `spectrum at infinity' of the scaled second variation
$-\ma\delta^{(2)}\Snu[0]$. According to the Fourier analysis in section~\ref{ss:NIEanalysis},
this value is given by 
\begin{equation}
    \min\sigma_c = \begin{cases}
        (1-\sigma_0\uinf)^2 & \text{if $\uinf>0$},\\
        (1+2\uinf)^2 & \text{if $\uinf\le0$},
    \end{cases}
    \qquad
-\sigma_0=\min_{\xi\in\R} 2\sinc\xi \approx -0.434467.
\end{equation}

The results in both sets of solutions appear consistent with the possibility
that periodic and solitary waves exist on the line for $\uinf$ in the whole range from $-0.5$ 
to at least $2.2$, consistent with 
the phase-speed non-matching condition mentioned in the introduction 
and with the coercivity result from propositions~\ref{prop:coercive} 
and~\ref{prop:invertible}.
For $\uinf$ between $-0.5$ and $0$, the wave perturbation $\hat w$ has a single hump
shape, monotonic for $0<x<L$. Log plots suggest that values of $|\hat w|$ smaller
than about $10^{-6}$ are not computed accurately with the precision 
and tolerances that were used.

For $\uinf\ne0$,
the values of $|\hat w|$ decay toward zero at an exponential rate that depends
upon $\uinf$ and diminishes as $\uinf$ approaches $-0.5$. 
This regime, where $2\uinf\to-1$, is where we can expect the KdV approximation to be valid,
in fact---the wave amplitude becomes small and wave length becomes large.

For $\uinf=0$, on the other hand, the wave profile $\hat w=\hat f$ 
may decay to zero at a rate that is faster than exponential. 
This is reminiscent of the solitary wave pulse for a chain
of beads in Hertz contact, which was shown in \cite{EnglishPego2005} to decay
at a rate that is faster than double exponential.

For the positive values of $\uinf$ between $0.2$ and $2.2$, on the other hand,
the computed wave perturbations $\hat w$ decay toward zero in an oscillatory, sign-changing way.
The decay rate of the envelope diminishes as $\uinf$ approaches the upper threshold
near $2.3$.  
The oscillation frequency appears well approximated by the value $\xi_*\approx 4.4934$
which minimizes $\sinc \xi$ and at which the phase velocity matches the group velocity 
of harmonic waves.  This corresponds to the regime investigated in the general study by 
Kozyreff~\cite{Kozyreff2023}.

\begin{figure}
    \centering
    \begin{minipage}[b]{0.45\textwidth}
        \centering        \includegraphics[width=\linewidth]{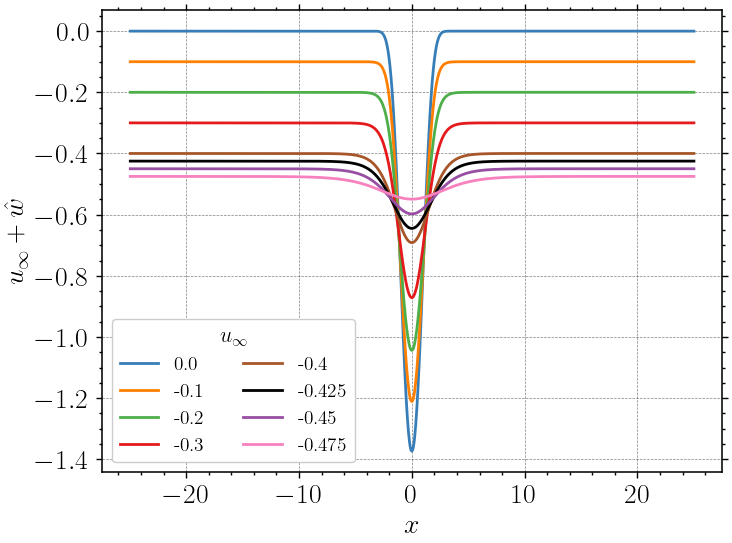}  
\subcaption{Solutions $\uinf+\hat w$ vs. $x$}
\label{fig:nie_1a}
    \end{minipage}
    \hspace{0.05\textwidth}
    \begin{minipage}[b]{0.45\textwidth}
        \centering        \includegraphics[width=\textwidth]{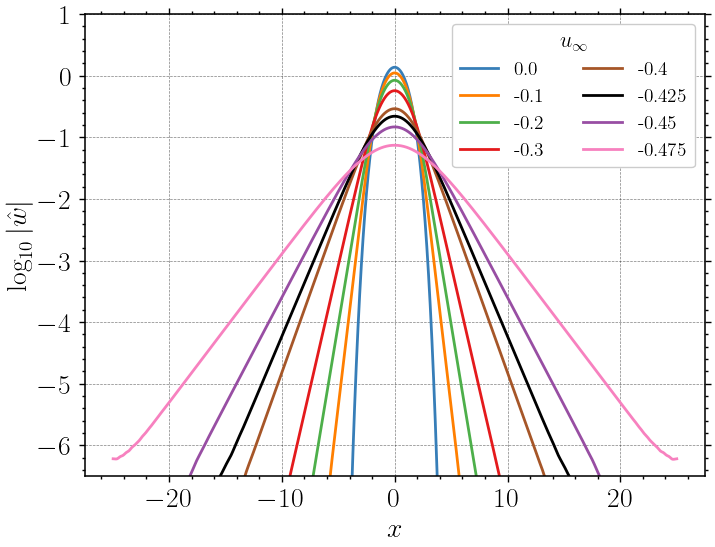} 
        \subcaption{$\log_{10}|\hat w|$ vs. $x$}
        \label{fig:nie_1b}
    \end{minipage}
    
    \caption{Solutions and dual fields for $\uinf = 0.0, -0.1, -0.2, -0.3, -0.4, -0.425, -0.45, -0.475$.}
    \label{fig:nie_1}
\end{figure}

\begin{figure}
    \centering
    \begin{minipage}[b]{0.45\textwidth}
        \centering        \includegraphics[width=\linewidth]{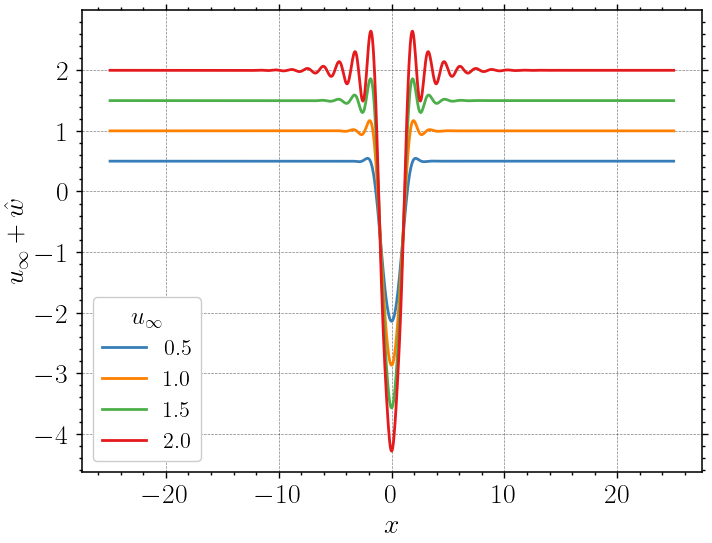}  
\subcaption{Solutions $\uinf+\hat w$ vs. $x$}
\label{fig:nie_2a}
    \end{minipage}
    \hspace{0.05\textwidth}
    \begin{minipage}[b]{0.45\textwidth}
        \centering        \includegraphics[width=\textwidth]{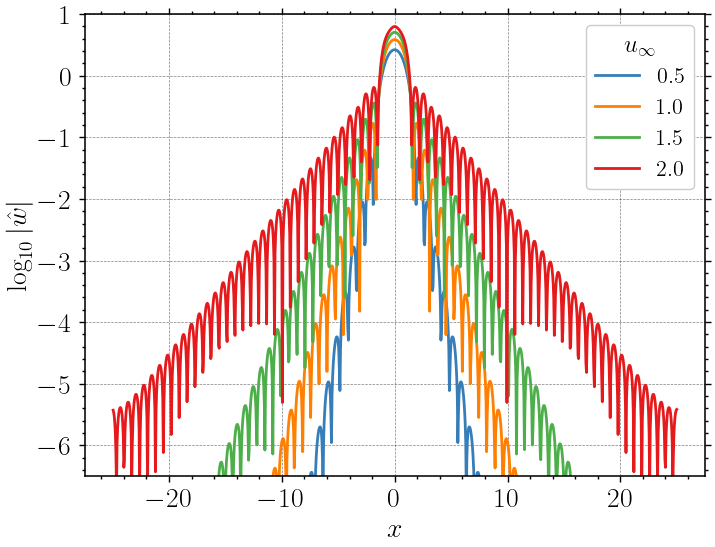} 
        \subcaption{$\log_{10}|\hat w|$ vs. $x$}
        \label{fig:nie_2b}
    \end{minipage}
    
    \caption{Solutions for $\uinf = 0.5, 1.0, 1.5, 2.0$.}
    \label{fig:nie_2}
\end{figure}

\begin{table}[ht]
\centering
\begin{tabular}{c|c|c|c|c|c|c|c}
$\uinf$ &  $\kappa_2$ & $\kappa_3$ & $\kappa_4$ & $\kappa_5$ & $\min\sigma_c$ \\
    \hline
 -0.45 &  0.00525 & 0.01008 & 0.01017 & 0.01155 & 0.01000 \\
 -0.40 &  0.01940 & 0.03960 & 0.04027 & 0.04210 & 0.04000 \\
 -0.30 &  0.06512 & 0.15327 & 0.16035 & 0.16209 & 0.16000 \\
 -0.20 &  0.12258 & 0.32787 & 0.36029 & 0.36172 & 0.36000 \\
 -0.10 &  0.18357 & 0.54470 & 0.63819 & 0.64042 & 0.64000 \\
 0.00 &  0.24398 & 0.64190 & 0.64190 & 0.78068 & 1.00000 \\
 0.50 &  0.48865 & 0.55764 & 0.55764 & 0.61489 & 0.61272 \\
 1.00 &  0.32294 & 0.32296 & 0.32611 & 0.32613 & 0.31983 \\
 1.50 &  0.12421 & 0.12423 & 0.12718 & 0.12721 & 0.12131 \\
 2.00 &  0.01867 & 0.01869 & 0.02025 & 0.02031 & 0.01718 \\
\end{tabular} 
\caption{Lowest eigenvalues of $-\ma D^2\Snu$, and bottom of continuous spectrum. 
$|\kappa_1|<2\cdot 10^{-15}$ in all cases.}
\label{table:NIE1} 
\end{table}

\subsection{Usage and failure of Petviashvili iteration}

As mentioned in the previous subsection, using the
optimization package Optim.jl in julia we found approximate solutions $(\hat w_j)$
to equation~\eqref{eq:int_wform} for which the residual $\vec\rho$ from \eqref{eq:gSnu_h}
has norm $\|\vec\rho\|_{\infty}<10^{-7}$ for all the values of $\uinf$ listed in Table~\ref{table:NIE1},
ranging from $-0.45$ to $2.2$.

A natural idea to improve the quality of the numerical solution is by 
post-processing, applying a Petviashvili iteration for several steps, 
starting with the solution found by Optim.jl.  
For solving \eqref{eq:int_wform}, one rewrites the equation in the equivalent form
\begin{equation}
    w(x) = -g(x) , \qquad g(x) = (I+\uinf K)\inv K(\tfrac12 g^2)(x),
\end{equation}
and replaces Step~1 in the iteration scheme \eqref{eq:PVscheme} by 
\begin{equation}
    \text{Step 1':}  \quad \tilde{g}_{n+1}(x) = (I+\uinf K)\inv K(\tfrac12 g_n^2)(x). 
\end{equation}
The operator $(I+\uinf K)\inv K$ is easily discretized and applied using the discrete Fourier transform.

With this method, we found that we could achieve residual norm $\|\vec\rho\|_\infty< 10^{-14}$  
for all values of $\uinf$ attempted ranging from $-0.45$ to $0.7$.
However, the number of iterations required increased greatly for $\uinf$ close to $0.8$,
and the Petviashvili iteration became unstable  and failed to improve the  solution found by Optim.jl,
for $\uinf$ in the range (0.8,2.3). 
This is a subset of the range where localized solutions with oscillatory tails are found.

\section{Discussion}

In this paper, we have taken a variational approach that has been developed
to solve PDEs through duality for convex optimization problems constrained by
field constraints, and extended it to handle two different nonlocal equations
that determine traveling waves for the semi-discrete inviscid Burgers equation:
a nonlinear advance-delay differential-difference equation (DDE) and a
corresponding nonlinear integral equation (NIE).

Particularly, we made extensive use of the flexibility of selecting the base state,
changing it to facilitate the numerical computation of solutions 
in many cases when optimization algorithms produce sequences approaching
the boundary of the functional domain where the objective functional is finite.
For the DDE case, we implemented an algorithm that adaptively adjusts dual field
increments to stay within the finiteness domain, and incorporates base state
resets that have the effect of shifting or reshaping the domain to
allow the search for a solution to continue and not simply stall at the boundary. 
For the NIE case, we used standard software to carry out each optimization
for a sequence of values of the parameter $\uinf$, setting as base state 
the solution found for the previous parameter value. This enables a new solution nearby
to be found with dual field presumably of small amplitude well inside the domain.

The automatic adaptivity built into the DDE code may be responsible for the fact
that it continued to work and produced ``disintegrated'' wave profiles that
appear delocalized and  may be non-periodic, in a parameter regime
($\uinf>2.35$, e.g.) where the periodic NIE code failed and analysis of the
periodic problem indicated difficulties with coercivity.

On the other hand, the periodic NIE code performed much better to find
well-localized (solitary) wave profiles with specified limiting states $\uinf$.
Unfortunately, we lack a convincing explanation for why this should be so.

For the nonlocal problems that we treated, truncation of the wave-profile problem on the
infinite line to a bounded interval requires extended boundary conditions for
base states and dual fields. 
We have done this in different ways for the DDE and NIE mainly as an experiment,
implementing Dirichlet-type conditions for the DDE and periodic conditions for
the NIE.  Switching the treatments is plausibly feasible; 
e.g., for the DDE case one could 
require that the base state $\bar f$ and dual field $\lambda$ be 
extended as periodic outside the interval $(-L,L)$. 
Analytically this is almost equivalent in principle 
to the periodic NIE formulation that we treated, with a subtle
difference, in that periodic variations $\delta\lambda$ in the DDE dual
field would have derivatives $\delta\nu=-\delta\lambda'$ constrained
to have integral zero. This means that solutions found by the 
two schemes might in principle correspond to different constants $C_1$
in \eqref{eq:int_cm1}.

We re-emphasize that in this paper we have focused on the properties of
the variational approach for the nonlocal wave profile problem, and not
on an exhaustive exploration of the family of solutions.
The use of software for continuation and path-following such as AUTO 
or pde2path would plausibly allow one to track branches of solutions and their
possible bifurcations more systematically than we have done. 

Our results nevertheless provide clues about certain parameter regimes 
that appear interesting to examine more closely. 
Because we are unable to guarantee that maximizers of the relevant concave
objective functionals do not lie on the boundary of the finiteness domain,
however,  we have not managed to prove an unconditional existence theorem for 
traveling-wave profiles using either the DDE or NIE formulations. 

The variational approach with base state changes does
suggest a possible avenue towards a convergence proof, though.  
E.g., in $L^2$-gradient flow for a convex functional,
the norm of the gradient is non-increasing in time.  In the problems we treat,
the functional gradient agrees with the equation residual for the DDE or NIE.
If one runs gradient flow and resets the base state 
with dual field reset to zero
(as in our numerical algorithm which was based on Newton-type iteration 
rather than gradient flow, however), 
the equation residual would not change with the reset
and would be ensured to be non-increasing in gradient flow afterwards. 
Perhaps one could stay away from the domain boundary and have the gradient flow equilibrate this way.

\section*{Acknowledgments}
The work of UK was supported by funds from the NSF grant OIA-DMR 2021019 and the Paul P. Christiano Professorship in the Dept.~of Civil \& Environmental Engineering at CMU.
This material is based upon work supported by the National Science Foundation
under grant DMS 2106534 to RLP.

\printbibliography
\appendix

\section{Formal KdV asymptotics}\label{s.KdV}
Here we describe solutions of
the semi-discrete Burgers equation \eqref{eq:dB} that are long-wave perturbations of a
constant state $u_*\ne0$, by a well-known formal asymptotic approximation. 
The approximation takes the form
\begin{equation}
    \label{d:kdv_u}
    u_j(t) = u_* + \eps^2 v(x,\tau), \qquad\text{with}\quad x = \eps(j-c_* t), \quad \tau =\eps^3 t, 
\end{equation}
where $c_*=u_*$ is the limiting speed of long waves in the linearization of \eqref{eq:dB}.
Then  straightforward use of the chain rule and Taylor expansion yields
\begin{align}
    \ddt u_j &= \eps^5\D_\tau v - \eps^3\,c_* \D_xv\,,
    \\
    u_{j\pm 1} &= u_* + \eps^2\left(v\pm \eps \D_x v + \tfrac12 \eps^2 \D_x^2 v 
    \pm \tfrac16\eps^3  \D_x^3 v+\tfrac1{24}\eps^4 \D_x^4 v\right)+O(\eps^7)\,.
\end{align}
By straightforward substitution this results in the residual, or equation error,
\begin{align}
   & \ddt u_j + \frac14\left(u_{j+1}^2 - u_{j-1}^2 \right) 
    = \eps^5 \left(
    \D_\tau v + v\, \D_x v + \frac{u_*}6 \D_x^3 v \right) + O(\eps^7)\,.
\end{align}
Thus \eqref{eq:dB} is formally approximated by the KdV equation  
\begin{equation}
    \D_\tau v + v\, \D_x v + \frac{u_*}6 \D_x^3 v = 0  \,.
\end{equation}
This KdV equation
has solitary wave solutions with any wave speed $\hat c$ having the same sign as $u_*$, given by 
\begin{equation}
    v(x,\tau) = {3\hat c} \sech^2\left(\frac{x-\hat c\tau}{2}\sqrt{\frac{6\hat c}{u_*}}\right)\,.
\end{equation}
This provides an approximate traveling wave for the semi-discrete Burgers equation~\eqref{eq:dB}
in the form 
\begin{equation}
    \hat u_j(t) = u_* + {3\gamma}\sech^2\left(\frac{j-ct}2 \sqrt{\frac{6\gamma}{u_*}}\right)
    ,\quad c=u_*+\gamma, \quad \gamma = \eps^2\hat c.
\end{equation}

\section{Effect of the choice of $\mathsf{a}$}\label{app:beta}
This appendix motivates the usage of a larger value of $\mathsf{a}$ in the auxiliary potential function $H$ \eqref{eq:auxiliary_func}. We consider the problem for a scaled $\mathcal{PV}$ set as base state (see Sec.~\ref{sec:without_prior} and \eqref{eq:2PV_base_state}) with $\tilde{\alpha}$ set as $2$. We employ a simple N-R scheme with tolerance set as per \eqref{eq:tol} for this problem with two different types of initial guesses on $\lambda$ given by: 
\begin{equation*}
    \lambda^{(0)}(x) = 0 \quad \mbox{and} \quad \lambda^{(0)}(x) = e^{-x^2}.
\end{equation*}
The convergence results and the corresponding residual values are shown in the Table \ref{table:Converge}.

\renewcommand{\arraystretch}{1.5} 
\begin{table}[h!]
\centering
\begin{minipage}{0.45\textwidth}
\centering
\begin{tabular}{|c|c|}
\hline
\multicolumn{2}{|c|}{$\lambda^{(0)}(x) = 0$} \\ \hline \hline
$\mathsf{a}$ & N-R Result \\ \hline
$10^{-6}$ & Converged    \\ \hline 
$1$       & Converged     \\ \hline
$10^{6}$  & Converged     \\ \hline
\end{tabular}
\end{minipage}
\hfill
\begin{minipage}{0.45\textwidth}
\centering
\begin{tabular}{|c|c|}
\hline
\multicolumn{2}{|c|}{$\lambda^{(0)}(x) = e^{-x^2}$} \\ \hline \hline
$\mathsf{a}$ & N-R Result \\ \hline 
$10^{-6}$ & No Convergence \\ \hline 
$1$       & No Convergence \\ \hline
$10^{6}$  & Converged       \\ \hline
\end{tabular}
\end{minipage}
\caption{Convergence results for different initial conditions: ``Converged" and ``No Convergence" indicate whether the residual in the N-R iterations converged below the set tolerance \eqref{eq:tol}, $tol=10^{-12}$, or not.}
\label{table:Converge}
\end{table}

\renewcommand{\arraystretch}{1}

The difference in the (non)convergence to a solution of the N-R iterations for the two cases can be understood as follows. Direct inspection shows that our algorithm produces invariant results under the following scaling transformation: if a residual $R^{(k,1)}$ is the result produced at iteration $k$ for a choice of $\mathsf{a} = 1$ and initial guess $\lambda^{(0)}$ then, for a choice of $\mathsf{a} = \mathsf{a}^*$, if the initial guess is scaled to $\mathsf{a}^* \lambda^{(0)}$, the residual remains invariant, i.~e.,
\[
R^{(k,\mathsf{a}^*)} = R^{(k,1)}.
\]
Moreover, $\lambda^{(k,\mathsf{a}^*)} = \mathsf{a}^* \lambda^{(k,1)}$.  

Indeed, from \eqref{eq:dtp} we observe that
\begin{equation*} 
f^{(H)}(\dee, 1, x) = f^{(H)}(\mathsf{a} \, \dee, \mathsf{a}, x), \end{equation*} 
where $\mathsf{a}\,\dee$ denotes scaling each element of $\dee$ by the scalar $\mathsf{a}$. Based on \eqref{qe:dual_deriv}, it can also be verified that 
\begin{equation*} 
\frac{\p f^{(H)}}{\p\,\dee}(\mathsf{a} \, \dee, \mathsf{a}, x) = \frac{1}{\mathsf{a}} \,\frac{\p f^{(H)}}{\p\,\dee}(\dee, 1, x). 
\end{equation*}
Thus, in the N-R iterations
\[
\lambda^{(k+1)} - \lambda^{(k)} = - J^{(k)-1} R^{(k)},
\]
if $\lambda^{(k)} \to \mathsf{a} \lambda^{(k)}$,  $J^{(k)-1} \to \mathsf{a} J^{(k)-1}$, $ R^{(k)} \to  R^{(k)}$ when $\mathsf{a}$ changes from $1 \to \mathsf{a}$ then this implies that at each step $\lambda^{(k+1)} \to \mathsf{a} \lambda^{(k+1)}$.

For $\lambda^{(0)} = 0$, the scaling hypothesis on the initial guess is satisfied and there is no change in the N-R iteration convergence profile as $\mathsf{a}$ is varied. When $\lambda^{(0)}$ is not scaled with $\mathsf{a}$, the hypothesis is not satisfied and there is a significant difference in the ability of the algorithm to obtain solutions with varying $\mathsf{a}$.

\newpage
\section{Testing primal fields using a Finite Difference approximation} \label{app}
Based on the dual solution obtained at nodes of any FE mesh, we use a finite difference to approximate the terms in \eqref{eq:primal} and check how well the equation gets satisfied at these nodes. For a primal field $f$ and corresponding to any node $A$, this agreement is evaluated based on the following expression:
\begin{equation}
\label{eq:fdm_approx}
    Err^A = \frac{f^{A+1}-f^{A-1}}{2\,dx} + \frac{1}{2}\left( \Big(f^{A+\frac{1}{dx}}\Big)^2 - \Big(f^{A-\frac{1}{dx}}\Big)^2 \right),  
\end{equation}
where $dx$ represents the element length and the meshes are chosen in such a way that for a node at $x$, there always exists nodes at $x+1$ and $x-1$.

Focusing on the internal part of the domain, the maximum absolute value of \( Err^A \) within the region \( x \in (-L+2.5, L-2.5) \) for examples in this work is summarized in Table \ref{table:err}.

\newpage
\section{Results obtained on mesh refinement}

\label{app:figures}


\begin{figure}[h!]
    \centering
    \begin{minipage}[b]{0.45\textwidth}
        \centering
        \includegraphics[width=\textwidth]{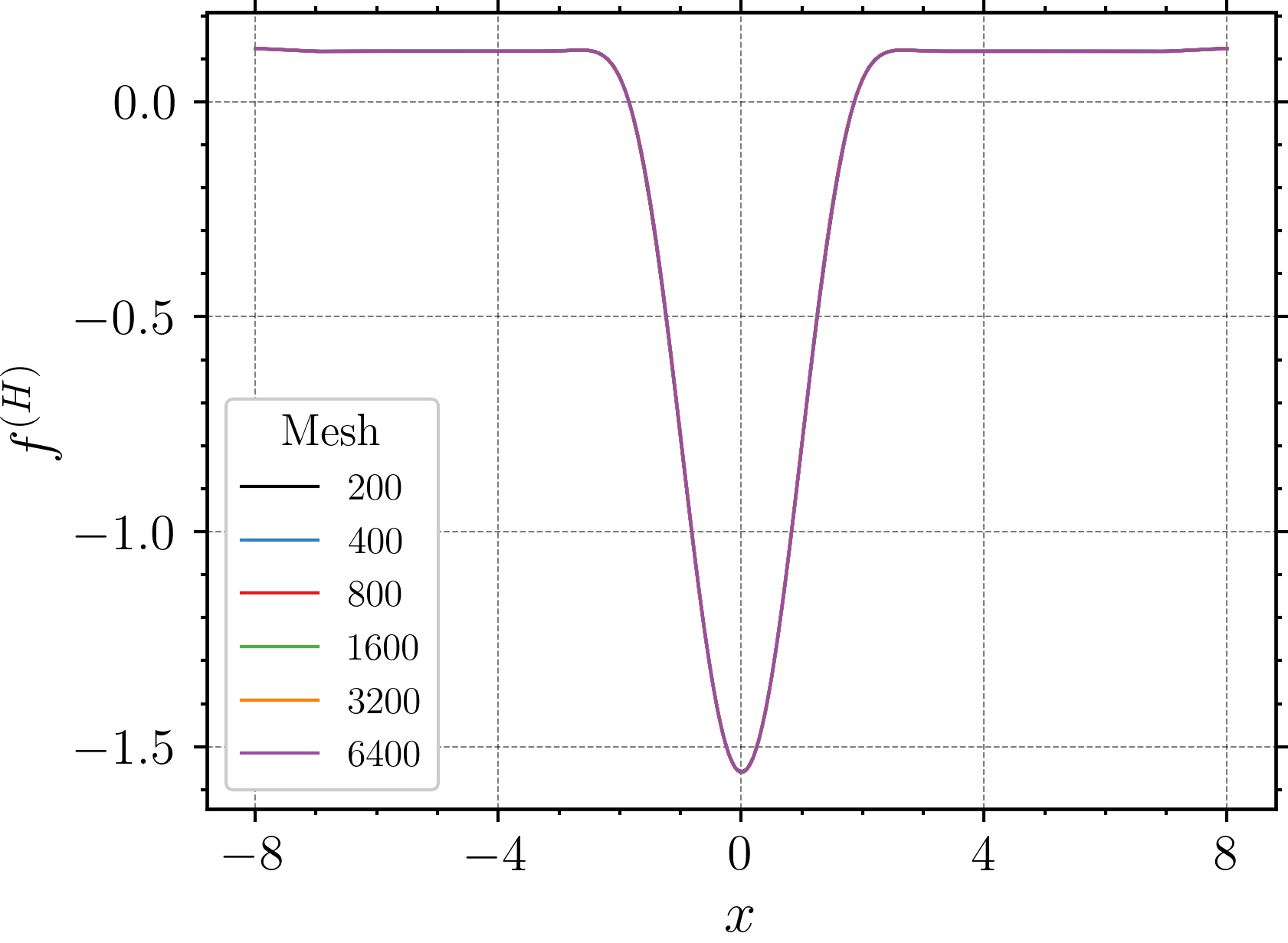} 
        \subcaption{$\tilde{\alpha}=2$}
        \label{fig:sub-P2_2}
    \end{minipage}
    \hspace{0.04\textwidth}
    \begin{minipage}[b]{0.46\textwidth}
        \centering
        \includegraphics[width=\textwidth]{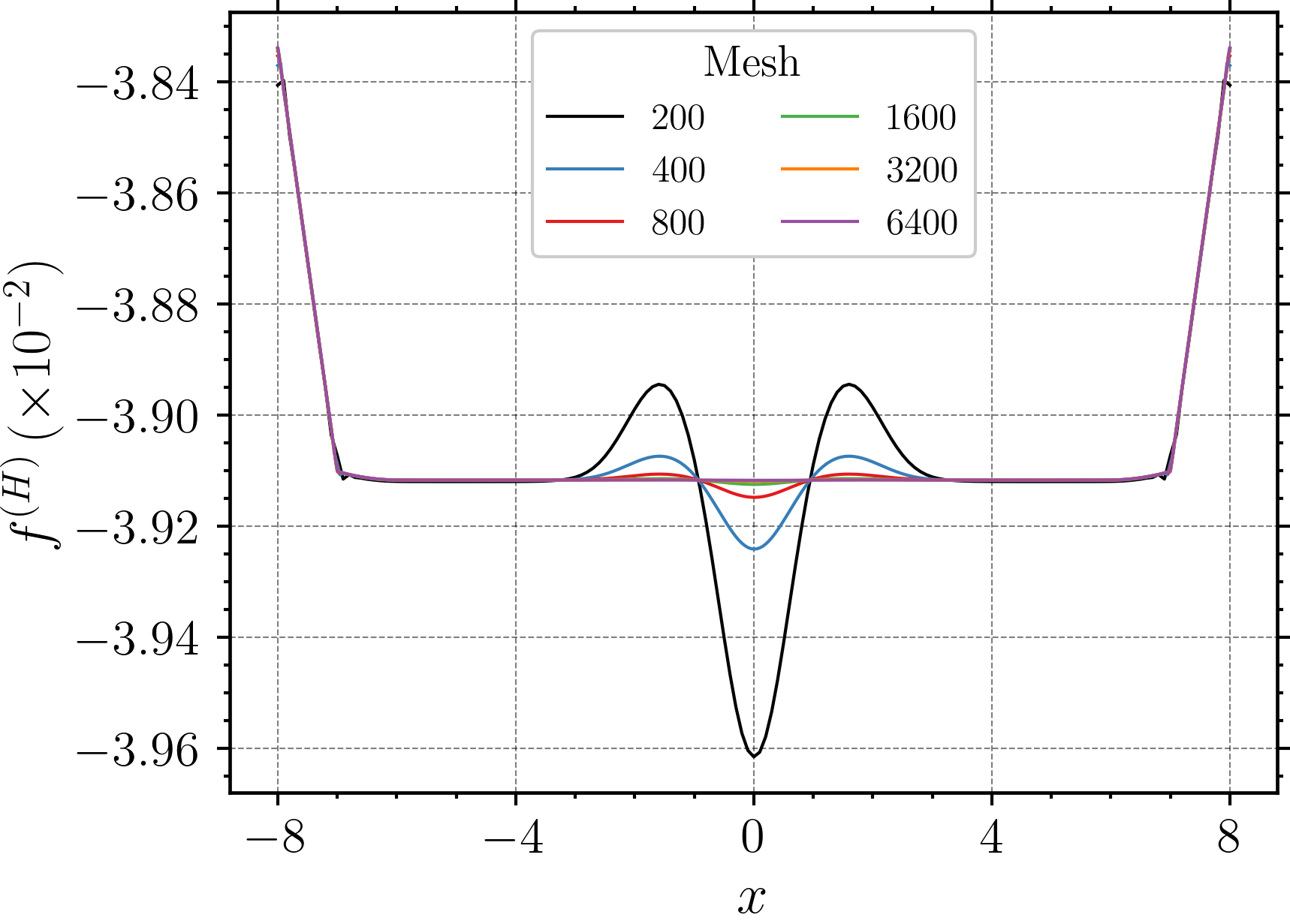} 
        \subcaption{$\tilde{\alpha}=0.2$ (Range of plot is $\mathcal{O}(10^{-3})$)}
        \label{fig:sub-CA_1}
    \end{minipage}

    \vspace{0.5cm}

    \begin{minipage}[b]{0.45\textwidth}
        \centering
        \includegraphics[width=\textwidth]{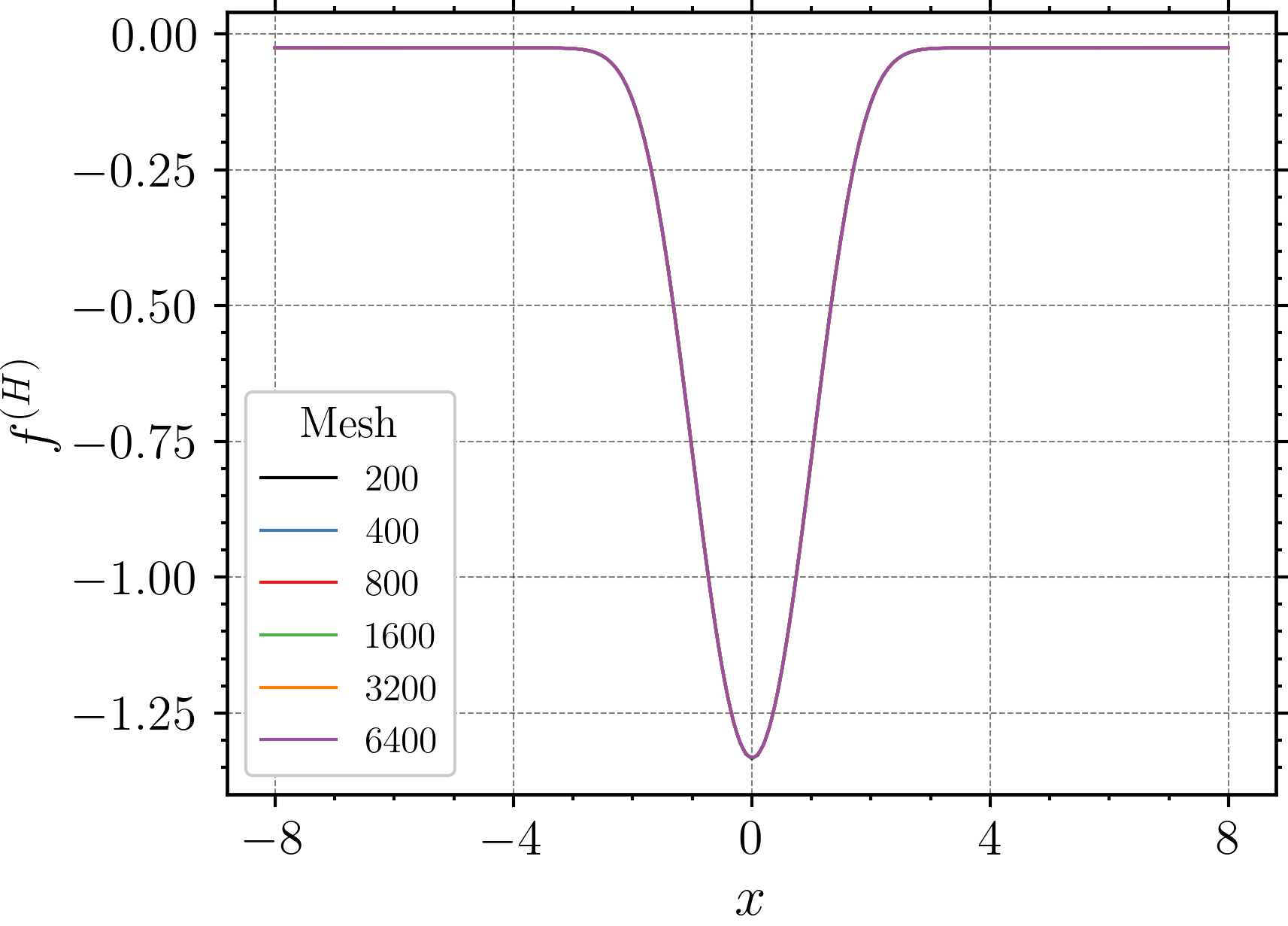} 
        \subcaption{$\tilde{\alpha}=0.8$}
        \label{fig:sub-CA_2}
    \end{minipage}
    \hspace{0.05\textwidth}
    \begin{minipage}[b]{0.45\textwidth}
        \centering
        \includegraphics[width=\textwidth]{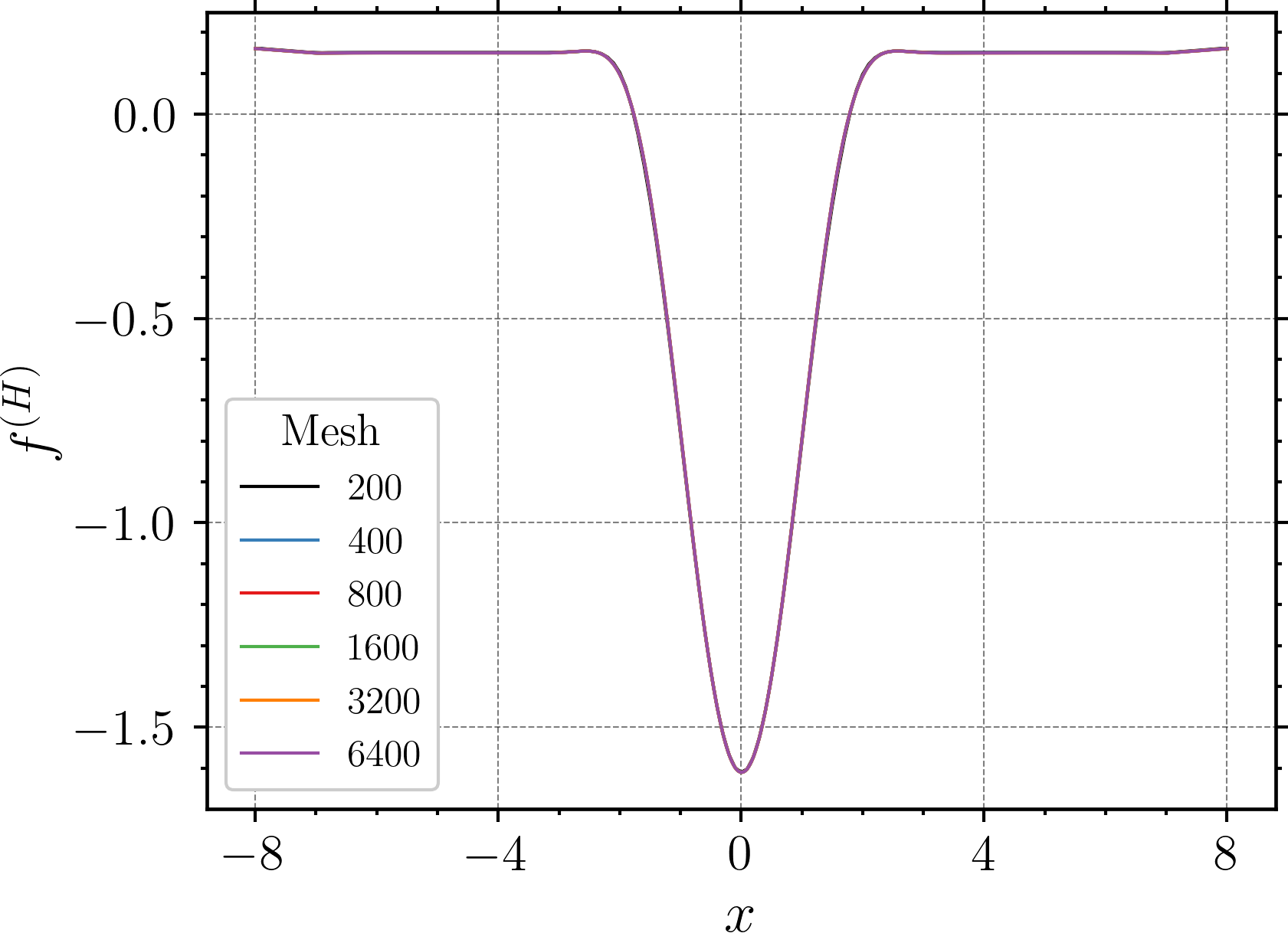} 
        \subcaption{$\tilde{\alpha}=2.3$}
        \label{fig:sub-CA_3}
    \end{minipage}
    
    \vspace{0.5cm}  

    \begin{minipage}[b]{0.45\textwidth}
        \centering
        \includegraphics[width=\textwidth]{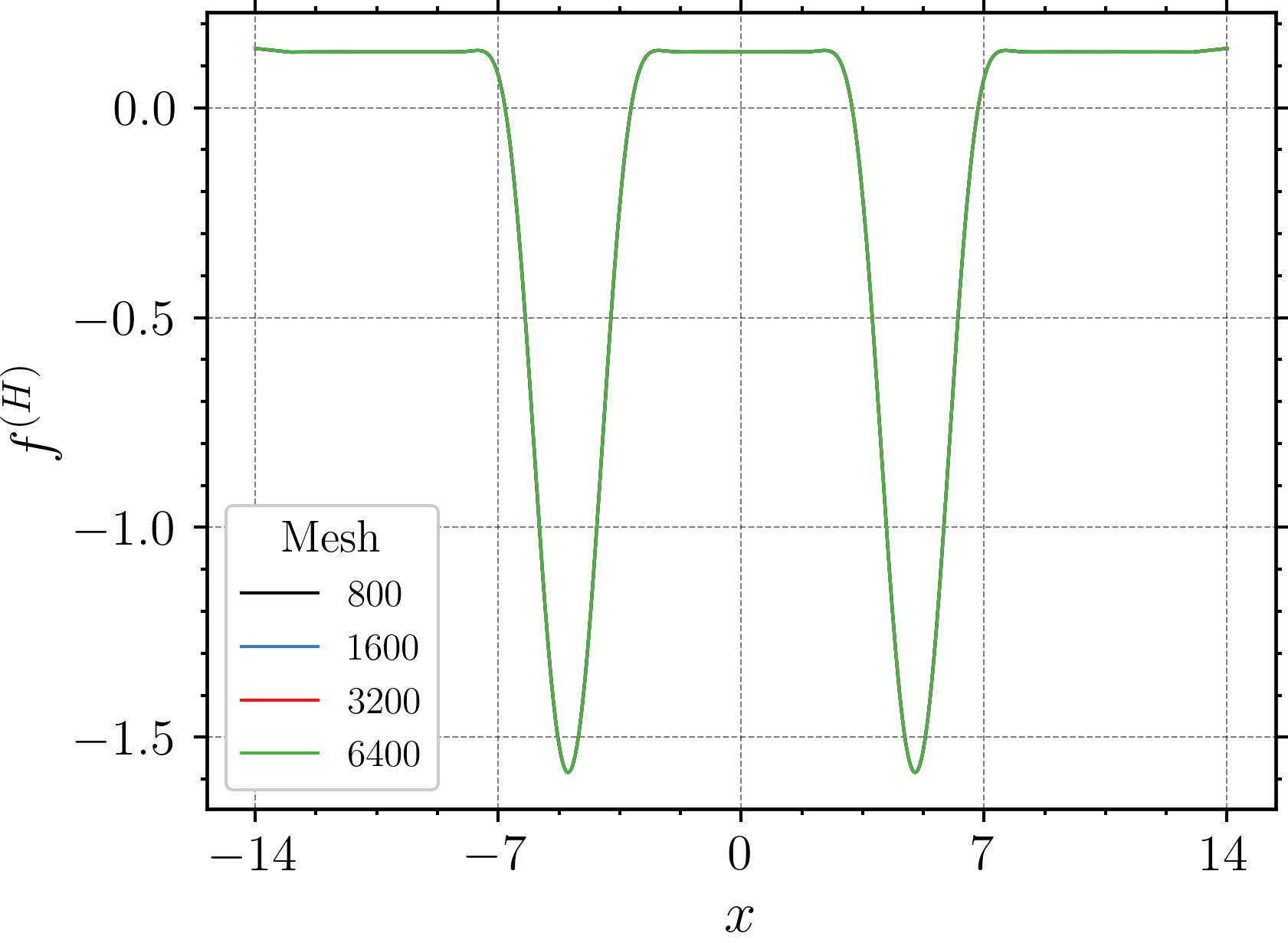} 
        \subcaption{$\tilde{\alpha}=2$ (Double Hump)}
        \label{fig:sub-CA_5}
    \end{minipage}
    
    \caption{Convergence of obtained primal profiles w.r.t mesh refinement with $\mathcal{PV}$ profiles scaled by a factor of $\tilde{\alpha}$ set as the base states: Fig.~(a)-(d) were produced with a single hump in the base state, whereas Fig.~(e) was produced using a double hump in the base state.} 
    \label{fig:CA}
\end{figure}

\begin{figure}
    \centering
    
    \begin{minipage}[b]{0.46\textwidth}
        \centering
        \includegraphics[width=\textwidth]{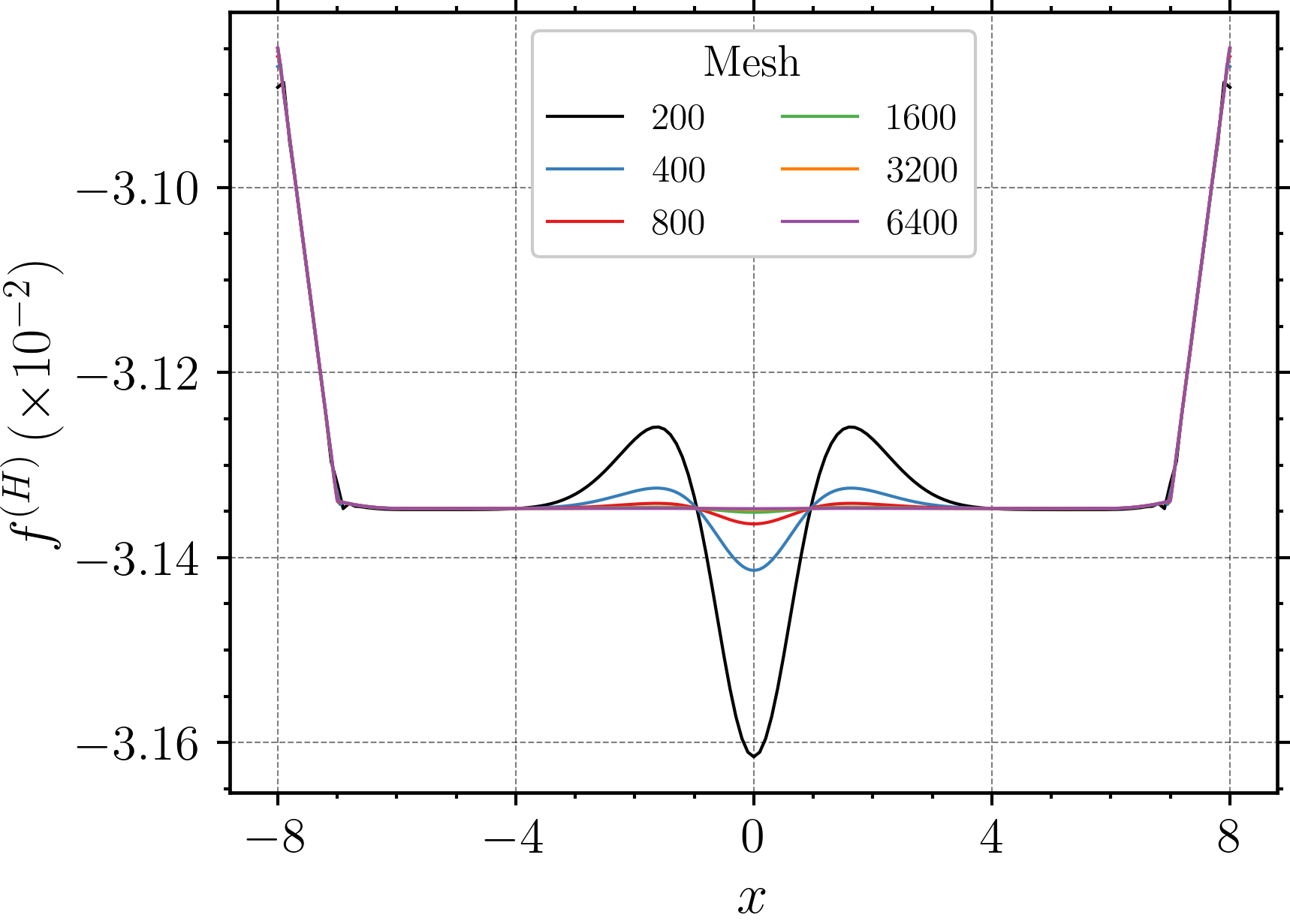} 
        \subcaption{$\gamma= -0.5$ (Range of plot is $\mathcal{O}(10^{-4})$)}
        \label{fig:sub-P5_2}
    \end{minipage}
    \hspace{0.04\textwidth}
    \begin{minipage}[b]{0.45\textwidth}
        \centering
        \includegraphics[width=\textwidth]{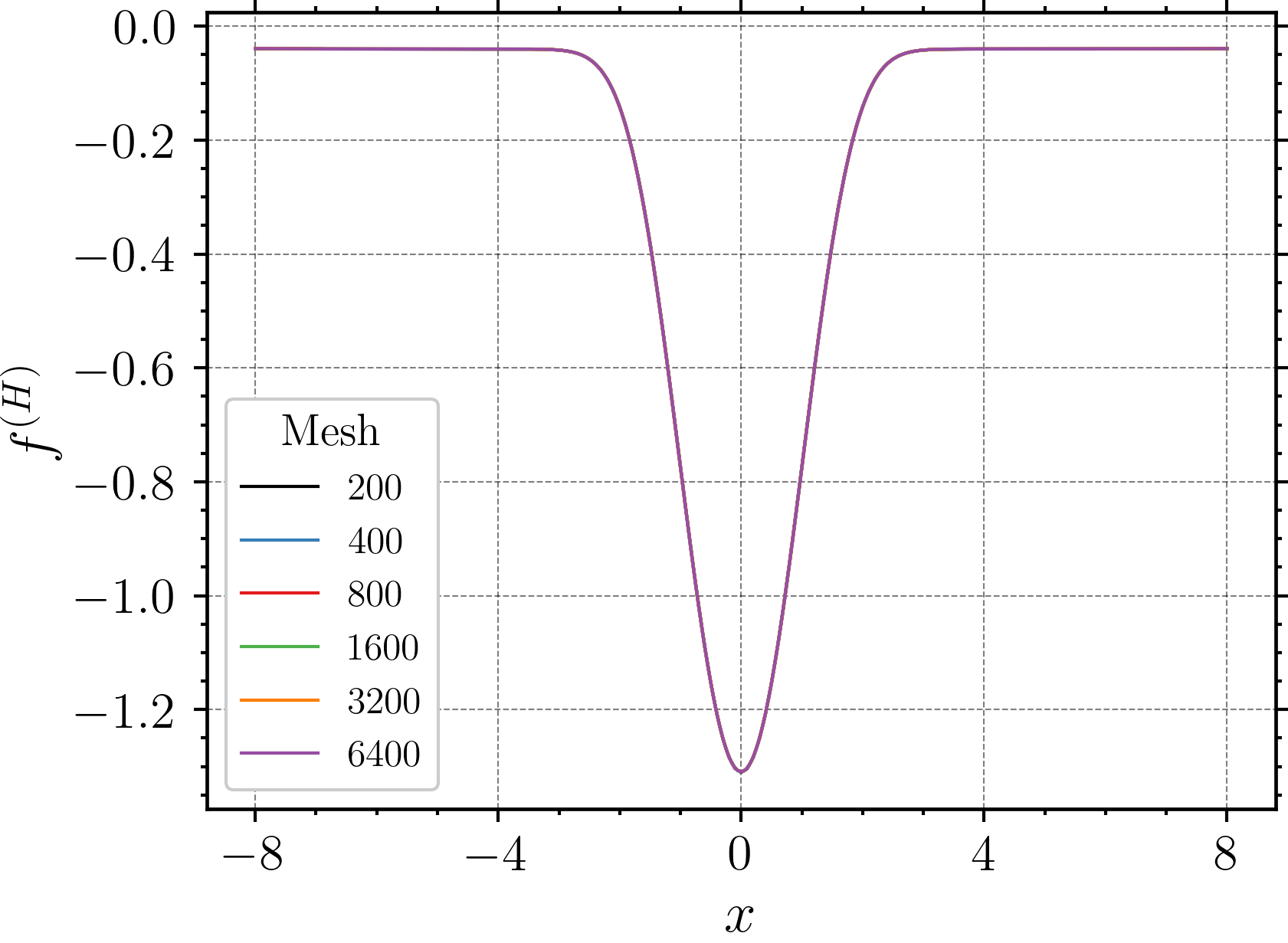} 
        \subcaption{$\gamma=-2.7$}
        \label{fig:sub-P5_3}
    \end{minipage}
    
    \vspace{0.5cm}  

    \begin{minipage}[b]{0.45\textwidth}
        \centering
        \includegraphics[width=\textwidth]{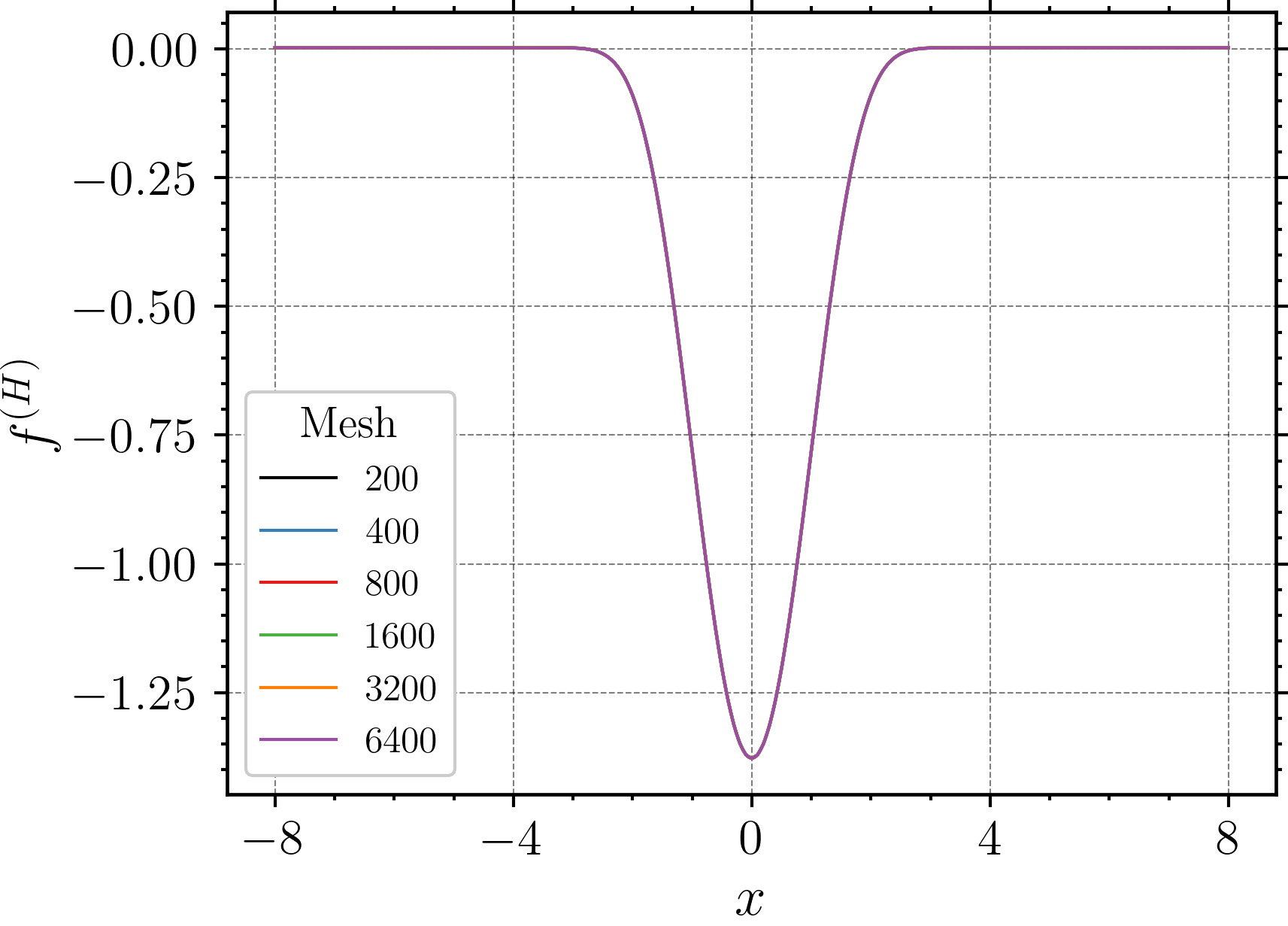} 
        \subcaption{$\gamma=-4.0$}
        \label{fig:sub-P5_4}
    \end{minipage}
    \hspace{0.05\textwidth}
    \begin{minipage}[b]{0.45\textwidth}
        \centering
        \includegraphics[width=\textwidth]{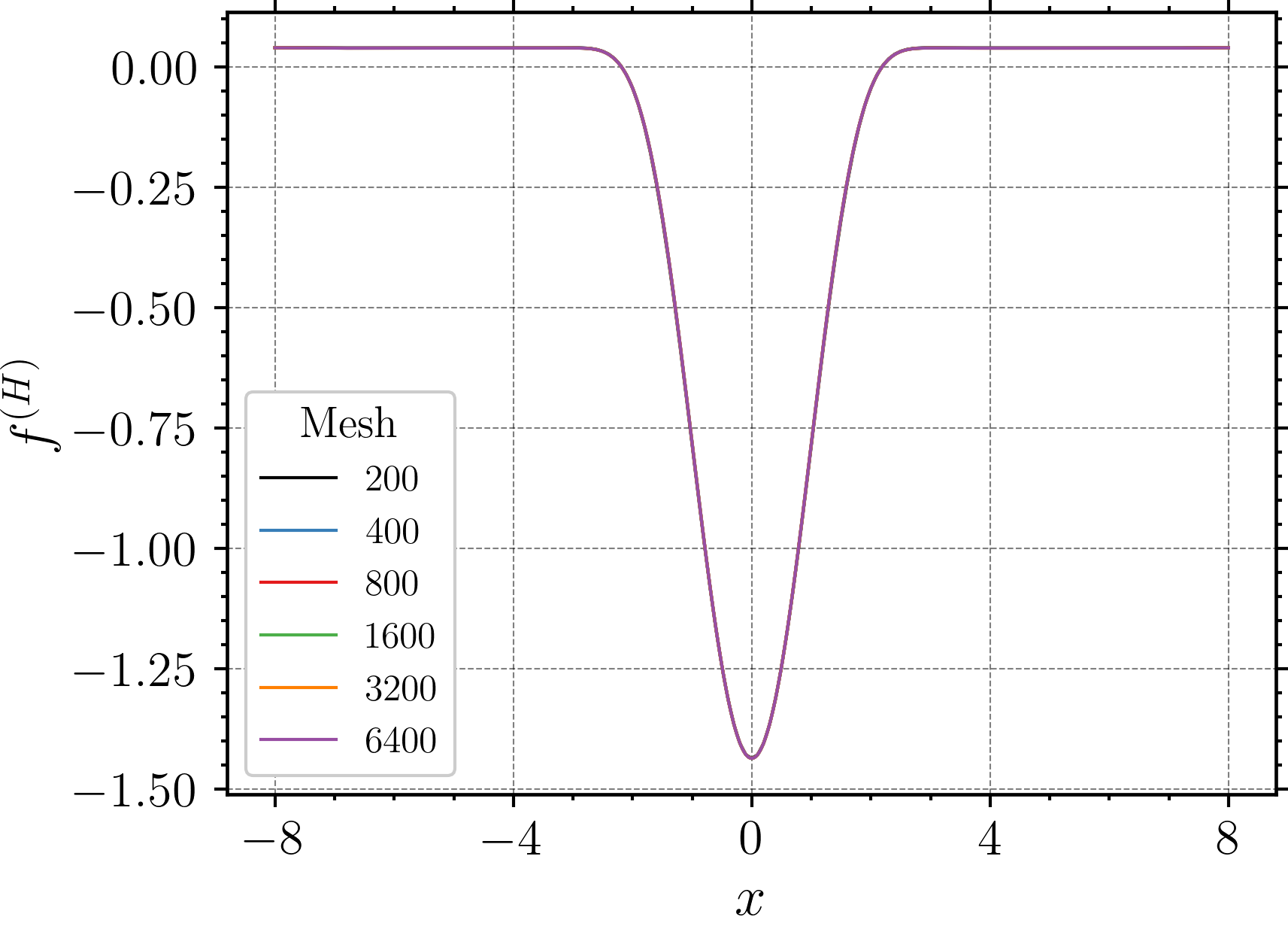} 
        \subcaption{$\gamma=-5.2$}
        \label{fig:sub-P5_5}
    \end{minipage}
  
    \caption{Convergence of obtained primal profiles w.r.t mesh refinement for scaled Gaussian  profiles (by a factor of $\gamma$) set as base states. Results were produced using a simple N-R scheme.} 
    \label{fig:CG}
\end{figure}

\begin{figure}
    \centering
    \begin{minipage}[b]{0.45\textwidth}
        \centering
        \includegraphics[width=\textwidth]{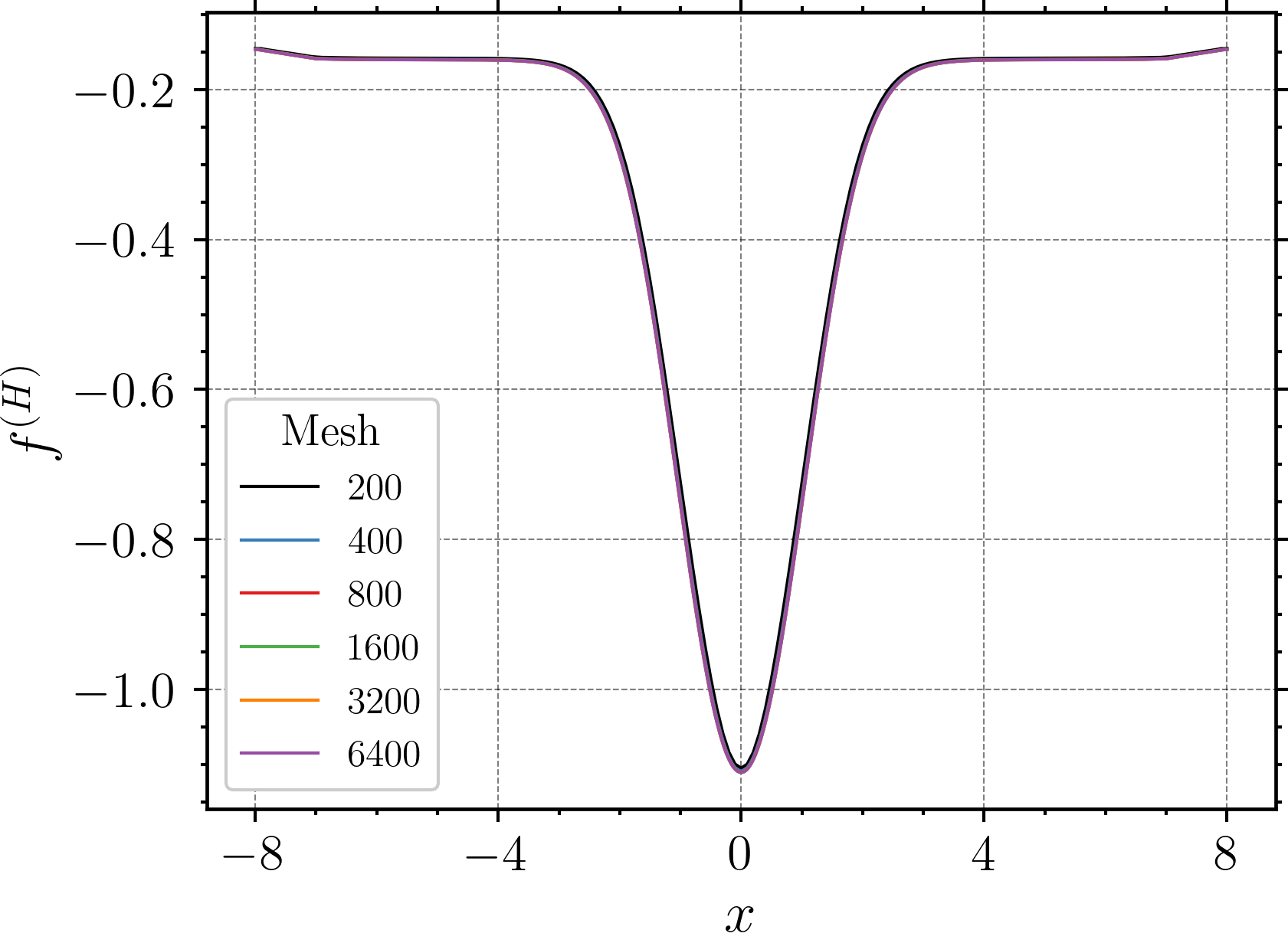} 
        \subcaption{$\gamma=-1.9$}
        \label{fig:sub-CGF_P1_1}
    \end{minipage}
    \hspace{0.05\textwidth}
    \begin{minipage}[b]{0.46\textwidth}
        \centering
        \includegraphics[width=\textwidth]{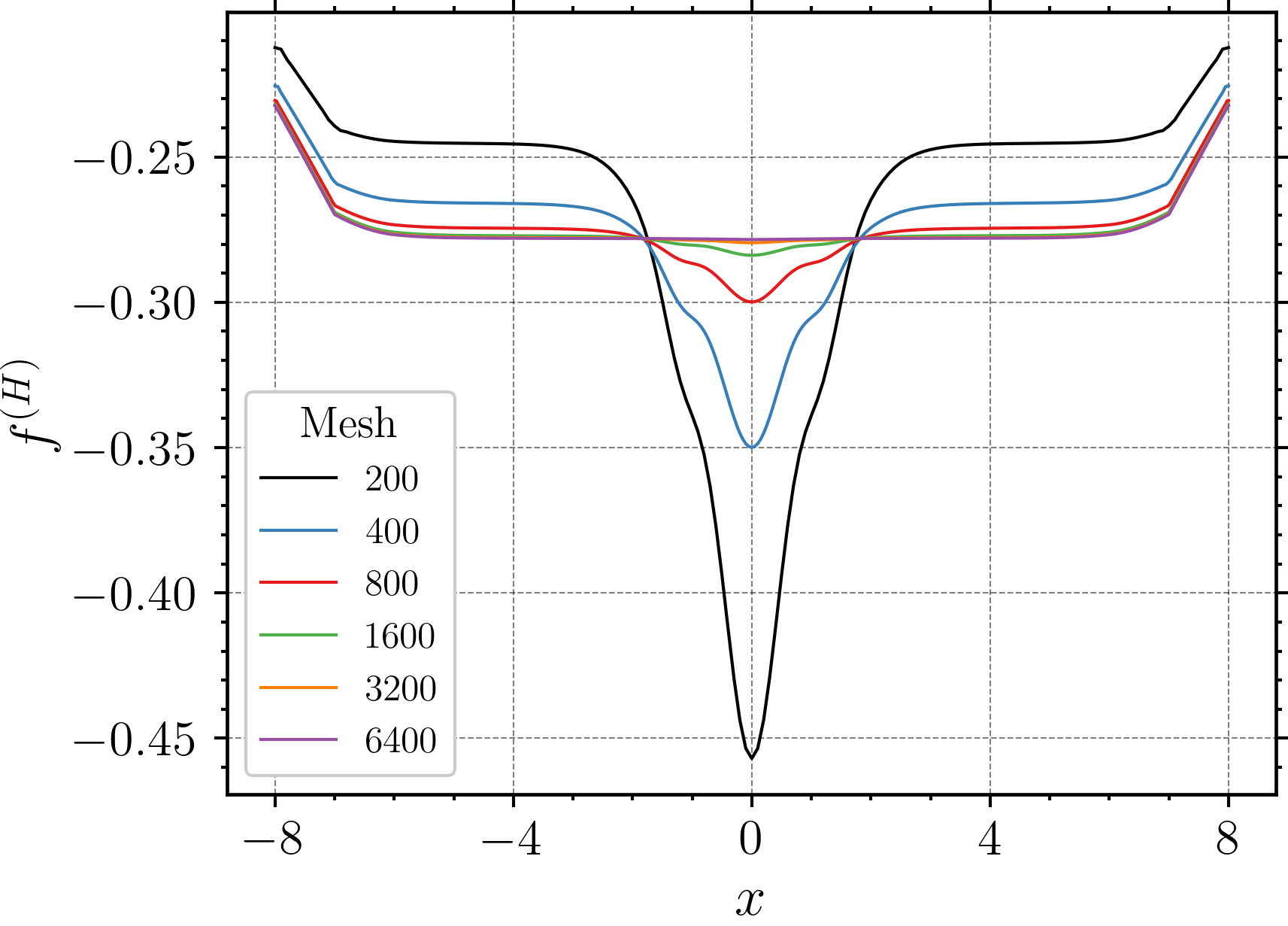} 
        \subcaption{$\gamma=-1.7$ (Range of plot is $\mathcal{O}(10^{-1})$)}
        \label{fig:sub-CGF_P1_2}
    \end{minipage}
    
    \caption{Convergence of obtained primal profiles w.r.t mesh refinement for two closely related scaled Gaussian profiles, each scaled by a factor of 
$\gamma$, set as the base state. Alg.~\ref{algo} was utilized.} 
    \label{fig:CGF_P1}
\end{figure}

\begin{figure}[h!]
    \centering
    \begin{minipage}[b]{0.46\textwidth}
        \centering
        \includegraphics[width=\textwidth]{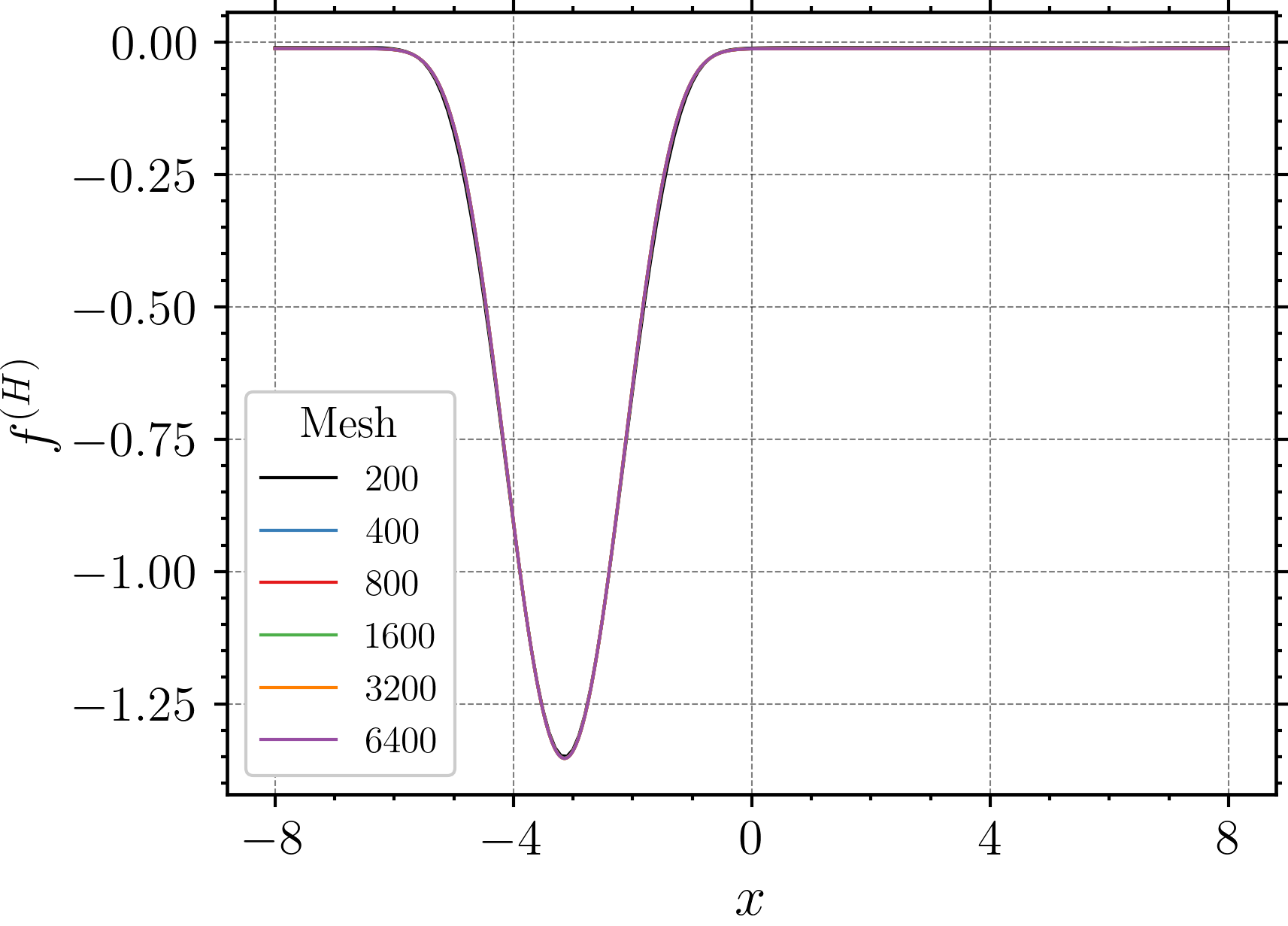} 
        \subcaption{$\omega=0.5$}
        \label{fig:sub-GF_P2_2}
    \end{minipage}
    \hspace{0.04\textwidth}
    \begin{minipage}[b]{0.46\textwidth}
        \centering
        \includegraphics[width=\textwidth]{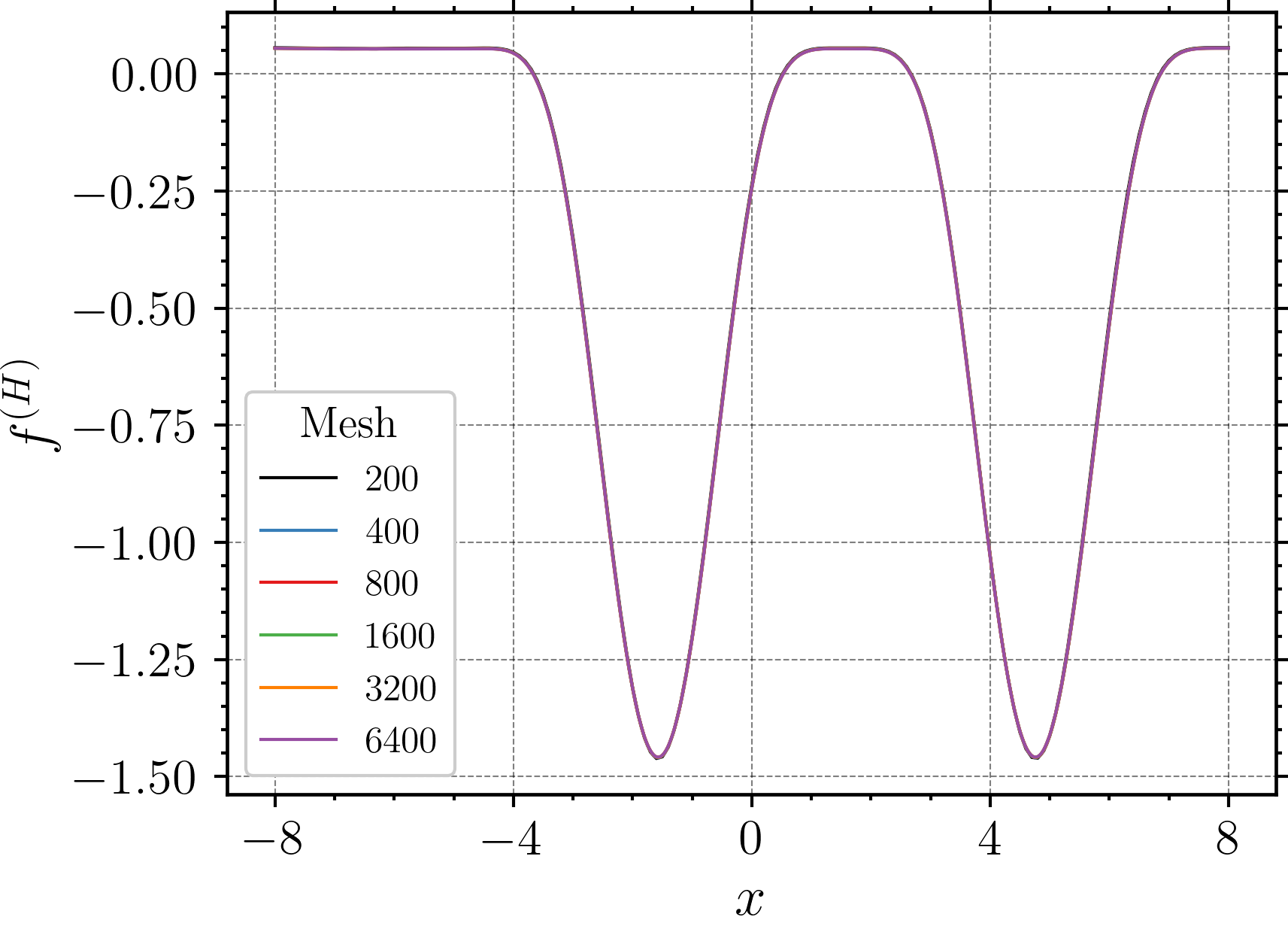} 
        \subcaption{$\omega=1$}
        \label{fig:sub-GF_P2_4}
    \end{minipage}

    \vspace{0.5cm}

    \begin{minipage}[b]{0.46\textwidth}
        \centering
        \includegraphics[width=\textwidth]{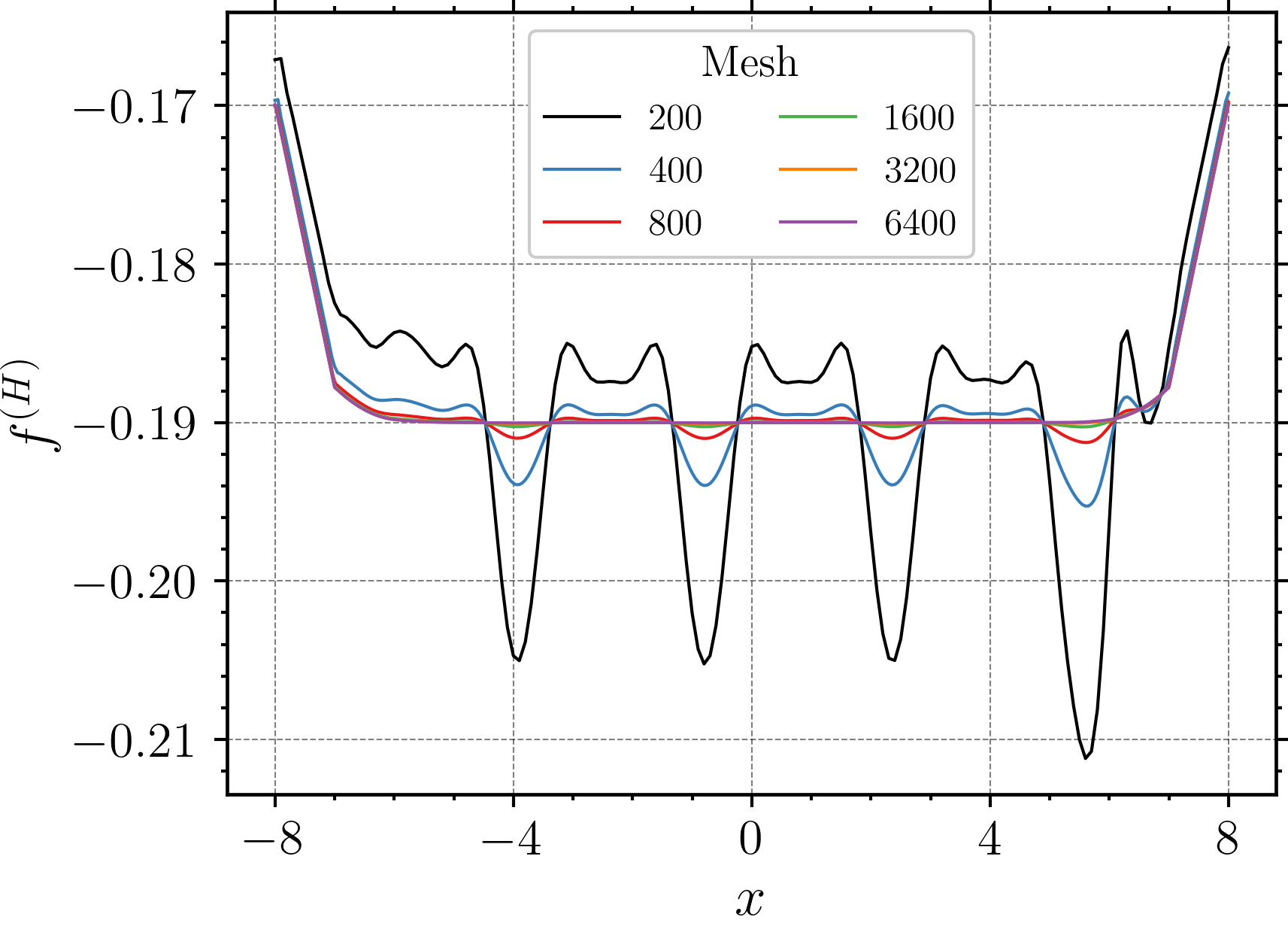} 
        \subcaption{$\omega=2$ (Range of plot is $\mathcal{O}(10^{-2})$)}
        \label{fig:sub-GF_P2_6}
    \end{minipage}
    \hspace{0.05\textwidth}
    \begin{minipage}[b]{0.45\textwidth}
        \centering
        \includegraphics[width=\textwidth]{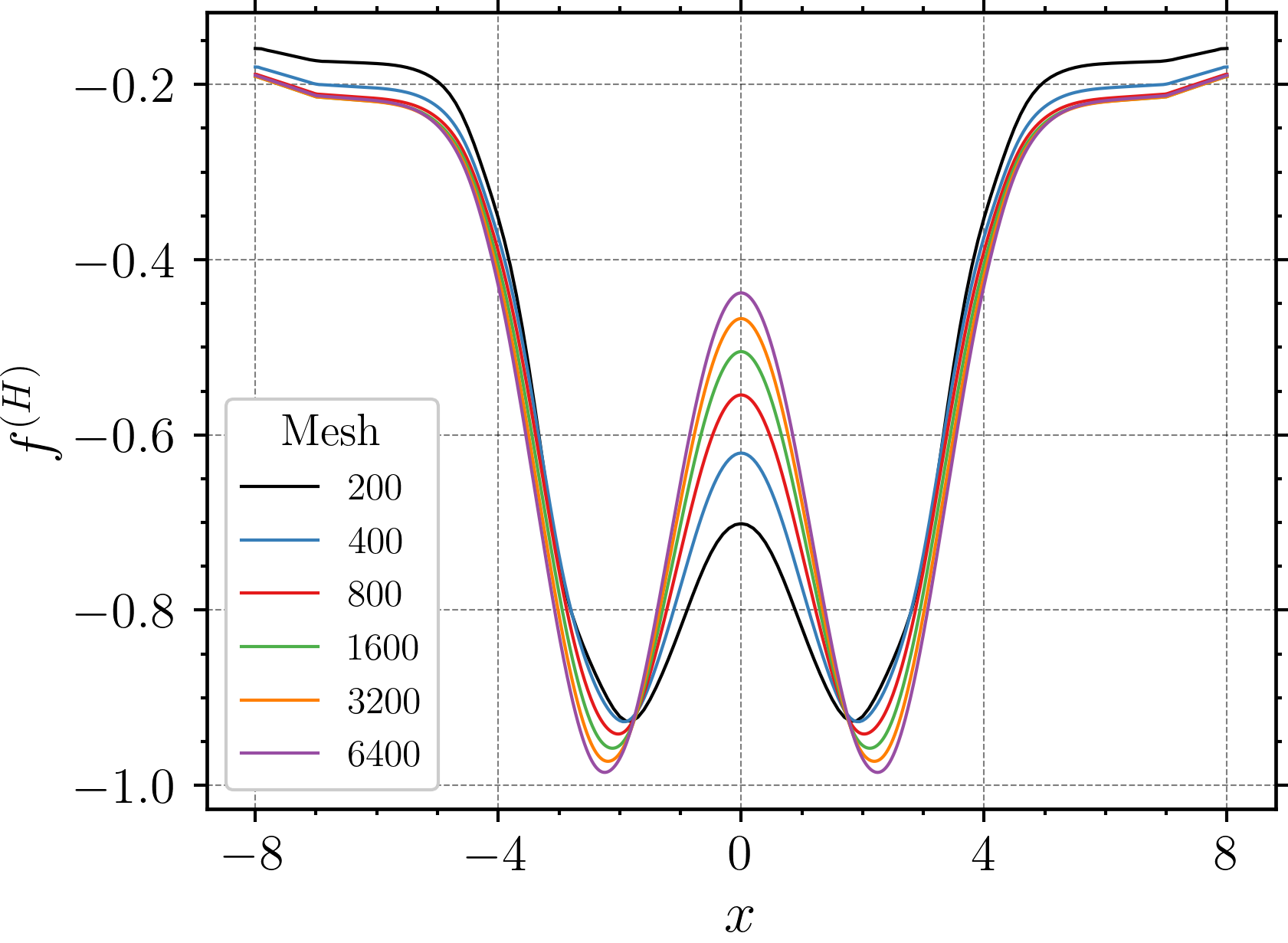} 
        \subcaption{$h=-0.4$}
        \label{fig:sub-GF_P3_2}
    \end{minipage}
    
    \vspace{0.5cm}  

    \begin{minipage}[b]{0.46\textwidth}
        \centering
        \includegraphics[width=\textwidth]{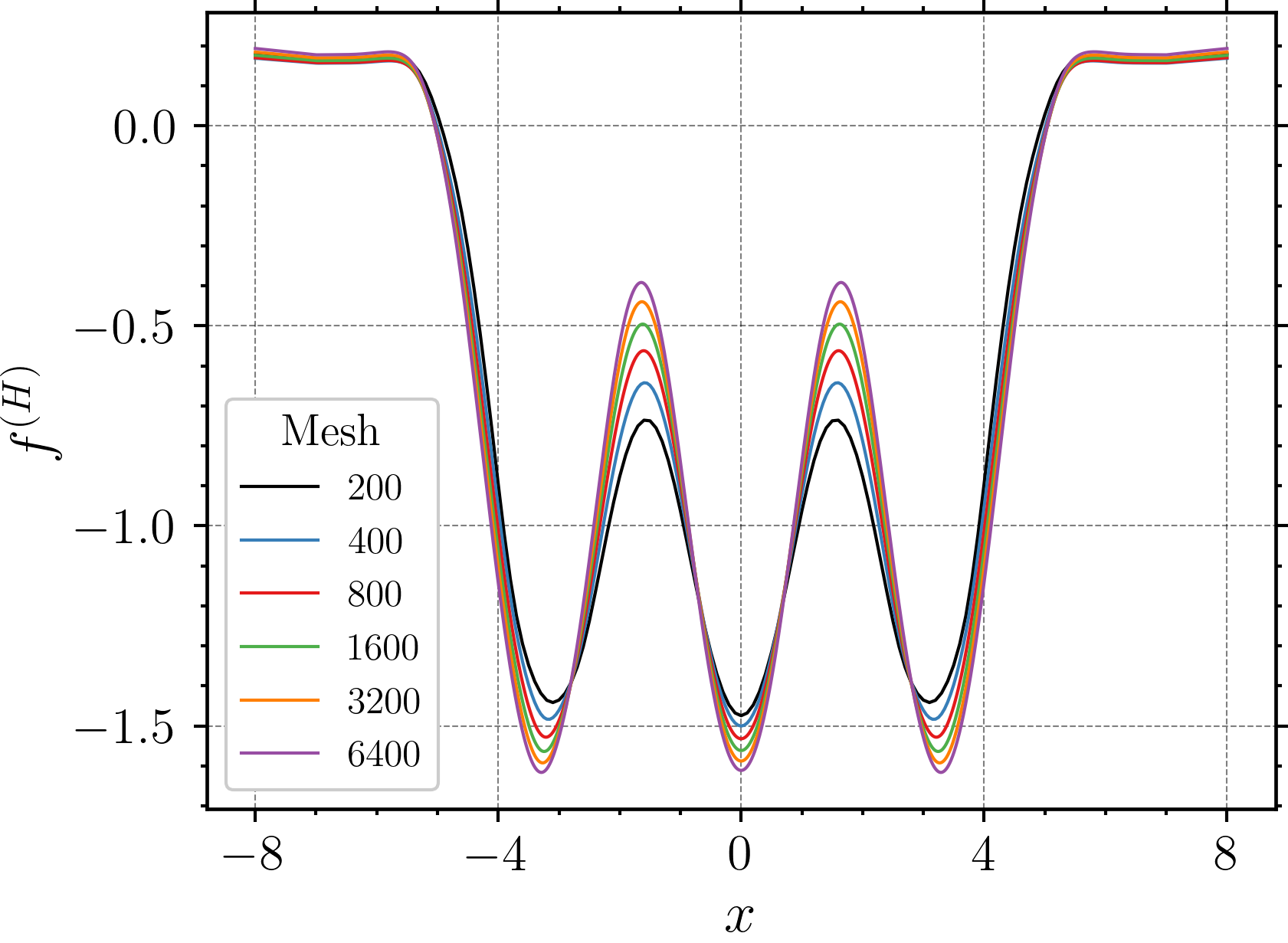} 
        \subcaption{$h=-0.8$}
        \label{fig:sub-GF_P3_4}
    \end{minipage}
    \hspace{0.05\textwidth}
    \begin{minipage}[b]{0.46\textwidth}
        \centering
        \includegraphics[width=\textwidth]{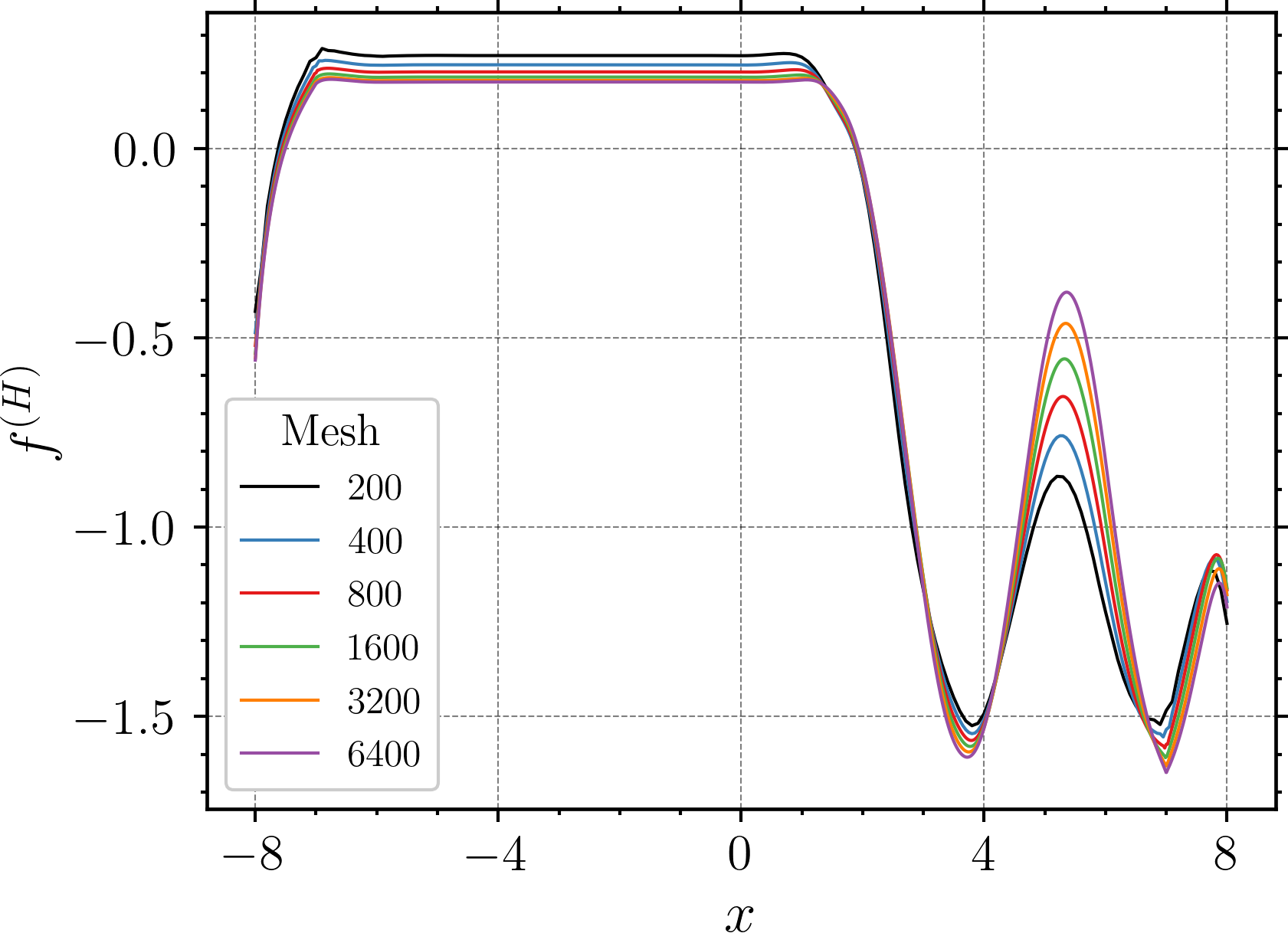} 
        \subcaption{$y=-0.25\,x$}
        \label{fig:sub-GF_P4_2}
    \end{minipage} 
    
    \caption{Convergence of obtained primal  profiles w.r.t mesh refinement for the following functions set as base states: sinusoidal functions \eqref{eq:sine}, negative hat functions \eqref{eq:hat} and a linear function \eqref{eq:linear} are shown in Fig.~(a)-(c), Fig.~(d)-(e) and Fig.~(f) respectively. Alg.~\ref{algo} was utilized.} 
    \label{fig:GFC_P2}
\end{figure}

\begin{figure}
    \centering
    \begin{minipage}[b]{0.45\textwidth}
        \centering
        \includegraphics[width=\textwidth]{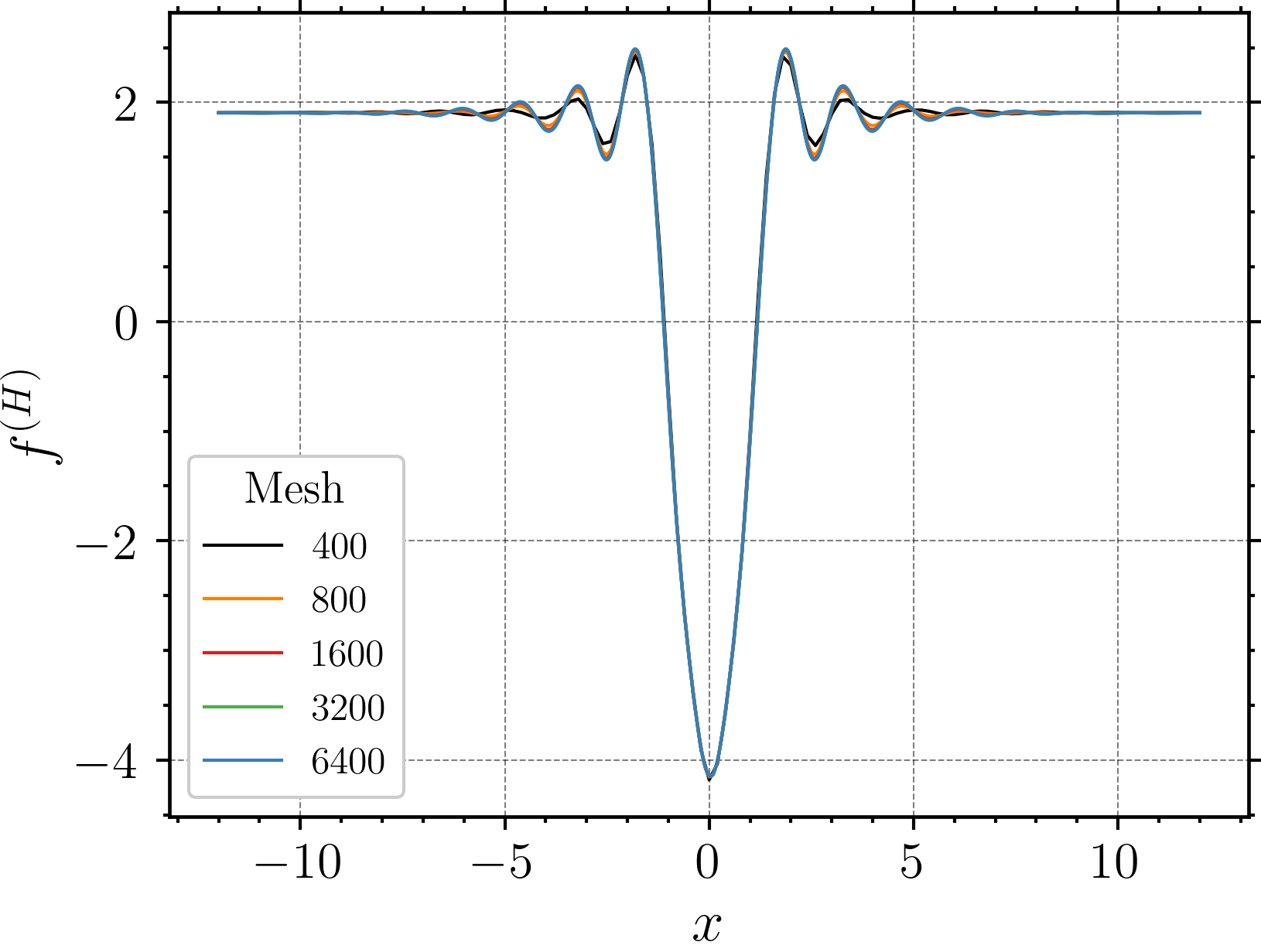} 
        \subcaption{$\bar{f}=3\mathcal{PV} +2$}
        \label{PV_waves_refine1}
    \end{minipage}
    \hspace{0.05\textwidth}
    \begin{minipage}[b]{0.45\textwidth}
        \centering
        \includegraphics[width=\textwidth]{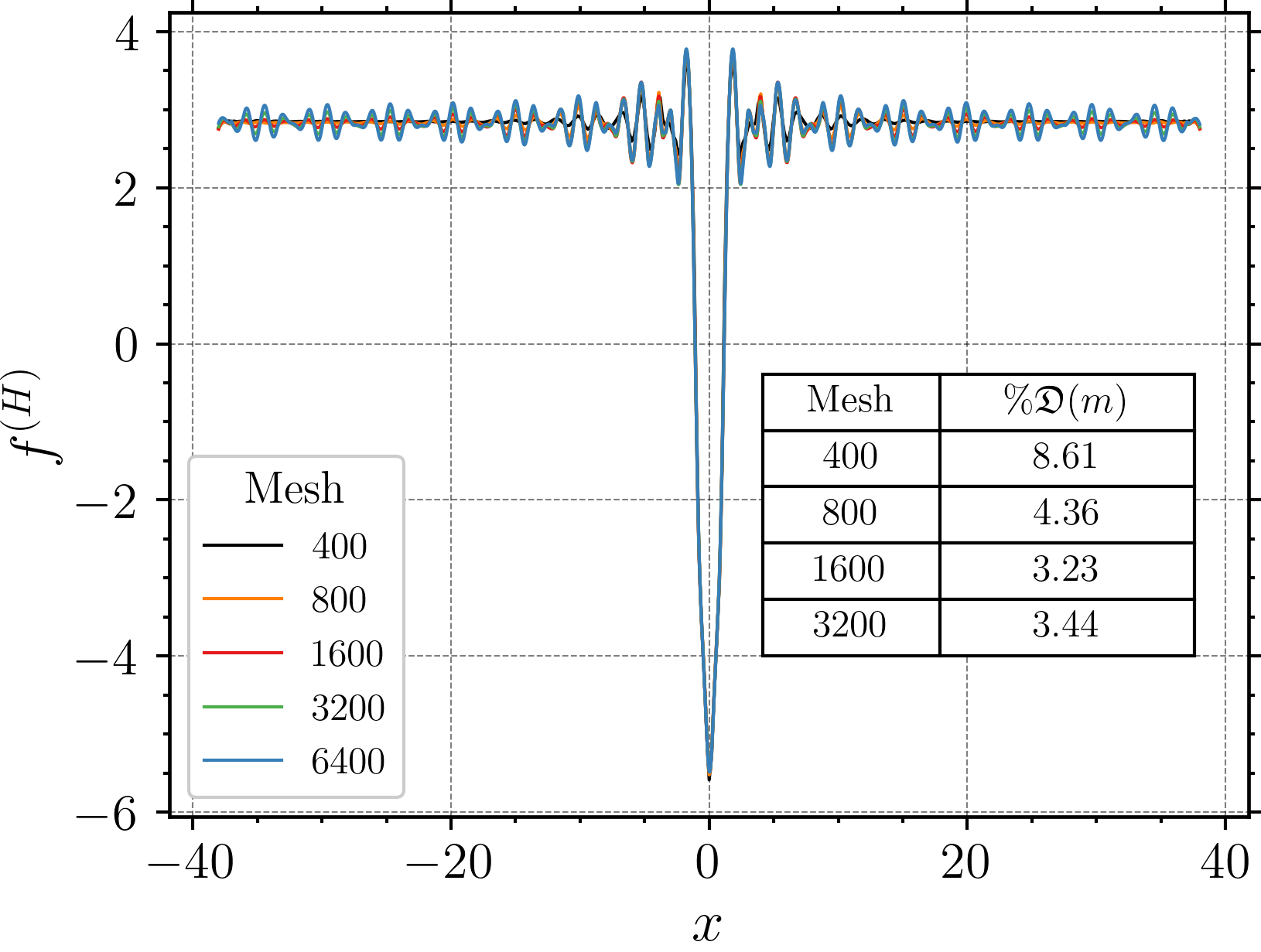} 
        \subcaption{$\bar{f}=4\mathcal{PV} +3$}
        \label{fig:PV_waves_refine2}
    \end{minipage} 
     
    \caption{Convergence of obtained primal profiles w.r.t mesh refinement for base states set as a scaled $\mathcal{PV}$ solution shifted by a  constant (Fig.~\ref{fig:PV_waves}). Profiles were drawn based on tests performed on $L=38$.} 
    \label{fig:PV_waves_refine}
\end{figure}

\clearpage
\section{Supporting tables} \label{app:tables}

\renewcommand{\arraystretch}{1.5} 
\begin{table}[h!]
\centering
\begin{tabular}{|c|c|c|c|c|c|c|c|c|c|}
\hline
\multicolumn{3}{|c|}{\rule{0pt}{0.7cm} Example \rule{0pt}{0.7cm}} &  \multicolumn{6}{|c|}{\rule{0pt}{0.7cm} Mesh \rule{0pt}{0.7cm}} & \multirow{2}{*}{Multiplier} \\ 
\cline{1-9} 
 Base State & Fig. & Parameter & 200   & 400 & 800   & 1600 & 3200   & 6400 & \\ \hline \hline
\multirow{4}{*}{Scaled $\mathcal{PV}$} & \ref{fig:sub-P2_2} & $\tilde{\alpha} = 2$   & 198 & 52 & 13 & 3 & 1 & 0.8 & \multirow{4}{*}{$\times 10^{-4}$}\\ \cline{2-9} 
                           & \ref{fig:sub-CA_1} &$\tilde{\alpha} = 0.2$ & 6.8 & 1.7 & 0.4 & 0.1 & 0.04 & 0.02 & \\ \cline{2-9} 
                           & \ref{fig:sub-CA_2} &$\tilde{\alpha} = 0.8$ & 61 & 16 & 4 & 1 & 0.4 & 0.4 & \\ \cline{2-9} 
                           & \ref{fig:sub-CA_3} &$\tilde{\alpha} = 2.3$   & 514 & 140 & 35 & 9 & 3 & 1 & \\ \hline
 \hline
\multirow{1}{*}{Double Hump} & \ref{fig:sub-CA_5} & $\tilde{\alpha} = 2$ & - & -  &  39 & 9.8 & 2.7 & 0.8 & $\times 10^{-4}$ \\ \hline
                            \hline

\multirow{6}{*}{Gaussian} & \ref{fig:sub-P5_2} & $\gamma = -0.5$ & 3.6 & 0.9& 0.22 & 0.05 &0.01 &0.003 & \multirow{4}{*}{$\times 10^{-4}$} \\ \cline{2-9} 
                           & \ref{fig:sub-P5_3} &$\gamma = -2.7$  & 37.9 & 9.48& 2.37& 0.59& 0.14& 0.03 & \\ \cline{2-9} 
                           & \ref{fig:sub-P5_4} &$\gamma = -4.0$ & 66 & 17 & 4 & 1 & 0.2 & 0.06 & \\ \cline{2-9} 
                           & \ref{fig:sub-P5_5} &$\gamma = -5.2$  & 108 & 28 & 7 & 1 & 0.4 & 0.1 & \\ \cline{2-10} 
                           & \ref{fig:sub-CGF_P1_1} &$\gamma =- 1.7$ & 153 & 60 & 18 & 5 & 1.6 & 0.3 & \multirow{2}{*}{$\times 10^{-3}$}\\ \cline{2-9} 
                           & \ref{fig:sub-CGF_P1_2} &$\gamma =- 1.9$  & 53 & 14.9 & 3.9 & 1 & 0.2 & 0.06 & \\ \hline
                            \hline

\multirow{3}{*}{Sine wave} & \ref{fig:sub-GF_P2_2} & $\omega = 0.5$ &  217& 55& 14& 3.5& 0.9& 0.2 & \multirow{3}{*}{$\times10^{-4}$} \\ \cline{2-9} 
                           & \ref{fig:sub-GF_P2_4} &$\omega = 1$ & 166& 42.8 &10.8 &2.7&0.7 &0.2 & \\ \cline{2-9} 
                           & \ref{fig:sub-GF_P2_6} &$\omega = 2$ & 396& 99 & 24 &6.2& 1.5 &0.3 & \\ \hline
                            \hline

\multirow{2}{*}{Hat functions} & \ref{fig:sub-GF_P3_2} & $h = -0.4$  & 13.3 & 7.17 & 3.86 & 2.35 & 2.02 & 1.81 & \multirow{2}{*}{$\times10^{-2}$} \\ \cline{2-9} 
                           & \ref{fig:sub-GF_P3_4} &$h =- 0.8$  &25.9 & 14.2 & 12.8 & 11.4 & 9.75 & 8.2 & \\

                            \hline \hline

\multirow{1}{*}{Linear} & \ref{fig:sub-GF_P4_2} & $y=-0.25\,x$ & 35.2 & 22.6 & 18 & 14.1 & 10.3 & 7.29 & $\times 10^{-2}$
\\ \hline \hline
\multirow{1}{*}{Soliton} & - & $\bar{f} = 3 \,\mathcal{PV} + 2 $ & - & 19.6 & 4.97 & 2.29 & 1.13 & 0.56 & $\times 10^{-3}$ 
                            \\ \hline d-soliton 
 &  - & $\bar{f} = 4 \,\mathcal{PV} + 3 $ & - & 11.4 & 9.3 & 5.7 & 3.56 & 2.14 & $\times 10^{-1}$ 
\\ \hline
\end{tabular}
\caption{Finite Difference Test: $Err^A$ based on \eqref{eq:fdm_approx}. Soliton refers to the dispersive solitary wave profile (Fig.~\ref{fig:sub-PV_waves_1}), while d-soliton denotes the disintegrated profile (Fig.~\ref{fig:sub-PV_waves_3}), both up to adjusting for value of $u_\infty$ as explained in text. All the tests (except the double-hump example) are performed on $L=8$.}
\label{table:err}
\end{table}

\begin{table}[h!]
\centering
\begin{tabular}{|c|c|c|c|c|c|c|c|}
\hline
\multicolumn{3}{|c|}{\rule{0pt}{0.7cm} Example \rule{0pt}{0.7cm}} &  \multicolumn{5}{|c|}{\rule{0pt}{0.7cm} {$\boldsymbol{\mathfrak{D}(m),(\%)}$} on Mesh  \rule{0pt}{0.7cm}} \\ 
\cline{1-8} 
 Base State & Fig. & Parameter & 200   & 400 & 800   & 1600 & 3200 \\ \hline \hline
\multirow{4}{*}{Scaled $\mathcal{PV}$} & \ref{fig:sub-P2_2} & $\tilde{\alpha} = 2$   & 1.06 & 0.27 & 0.06 & 0.02 & 0.004 \\ \cline{2-8} 
                           & \ref{fig:sub-CA_1} & $\tilde{\alpha}=0.2$ & 1.07 & 0.27 & 0.07 & 0.02 & 0.006 \\ \cline{2-8} 
                           & \ref{fig:sub-CA_2} & $\tilde{\alpha}=0.8$ & 0.67 & 0.17 & 0.04 & 0.01 & 0.002 \\ \cline{2-8} 
                           & \ref{fig:sub-CA_3} & $\tilde{\alpha}=2.3$ & 1.78 & 0.45 & 0.11 & 0.03 & 0.007 \\ \hline
 \hline
\multirow{1}{*}{Double Hump} & \ref{fig:sub-CA_5} & $\tilde{\alpha} = 2$ & - & - & 0.16 & 0.04 & 0.01 \\ \hline
                            \hline

\multirow{6}{*}{Gaussian} & \ref{fig:sub-P5_2} & $\gamma=-0.5$ & 0.71 & 0.18 & 0.05 & 0.01 & 0.005 \\ \cline{2-8} 
                           & \ref{fig:sub-P5_3} & $\gamma=-2.7$ & 0.40 & 0.10 & 0.03 & 0.006 & 0.002 \\ \cline{2-8} 
                           & \ref{fig:sub-P5_4} & $\gamma=-4.0$ & 0.43 & 0.11 & 0.03 & 0.007 & 0.002 \\ \cline{2-8} 
                           & \ref{fig:sub-P5_5} &$\gamma=-5.2$ & 0.89 & 0.22 & 0.06 & 0.01 & 0.003 \\ \cline{2-8} 
                           & \ref{fig:sub-CGF_P1_1} &$\gamma=-1.7$ & 43.62 & 20.40 & 6.53 & 1.76 & 0.44 \\ \cline{2-8} 
                           & \ref{fig:sub-CGF_P1_2} &$\gamma=-1.9$ & 5.55 & 1.60 & 0.41 & 0.10 & 0.03 \\ \hline
                            \hline

\multirow{3}{*}{Sine wave} & \ref{fig:sub-GF_P2_2} &$\omega=0.5$ & 2.96 & 0.72 & 0.18 & 0.04 & 0.01 \\ \cline{2-8} 
                           & \ref{fig:sub-GF_P2_4} &$\omega=1$ & 1.59 & 0.34 & 0.19 & 0.17 & 0.14 \\ \cline{2-8} 
                           & \ref{fig:sub-GF_P2_6} &$\omega=2$ & 9.44 & 2.38 & 0.58 & 0.14 & 0.03 \\ \hline
                            \hline

\multirow{2}{*}{Hat functions} & \ref{fig:sub-GF_P3_2} & $h=-0.4$ & 16.09 & 13.25 & 9.88 & 7.52 & 5.86 \\ \cline{2-8} 
                           & \ref{fig:sub-GF_P3_4} &$h=-0.8$ & 12.35 & 10.68 & 8.89 & 7.52 & 6.52 \\

                            \hline \hline

                            \multirow{1}{*}{Linear} & \ref{fig:sub-GF_P4_2} & $y=-0.25\,x$ & 17.5 & 17 & 16.5 & 15.61 & 13.64
                            \\ \hline \hline
                            \multirow{1}{*}{Soliton} & - & $\bar{f} = 3\, \mathcal{PV} + 2$ & - & 0.29 & 0.14 & 0.071 & 0.036
                            \\ \hline
                             d-soliton & - & $\bar{f} = 4 \,\mathcal{PV} + 3$ & - & 8.56 & 6.82 & 5.2 & 3.64
                            \\ \hline
\end{tabular}
\caption{Values for $\mathfrak{D}(m),(\%)$ are provided for all examples. In the table, the second column, titled {Fig.}, corresponds to the figures illustrating the convergence of obtained primal profiles w.r.t mesh refinement. Soliton refers to the dispersive solitary wave profile (Fig.~\ref{fig:sub-PV_waves_1}), and d-soliton denotes the disintegrated profile (Fig.~\ref{fig:sub-PV_waves_3}), both up to adjusting for value of $u_\infty$ as explained in text. All the tests (except the double-hump example) are performed on $L=8$.}
\label{table:convergence}
\end{table}
\end{document}